\documentclass[british]{article}
\usepackage[utf8]{inputenc}
\usepackage{geometry}
\usepackage{color}
\usepackage{booktabs}
\usepackage{textcomp}
\usepackage{mathtools}
\usepackage{url}
\usepackage{amsmath}
\usepackage{amsthm}
\usepackage{amssymb}
\usepackage{graphicx}
\usepackage{setspace}
\usepackage[authoryear,round]{natbib}
\makeatletter

\newcommand*\LyXThinSpace{\,\hspace{0pt}}
\DeclareTextSymbolDefault{\textquotedbl}{T1}
\providecommand{\tabularnewline}{\\}

\theoremstyle{plain}
\newtheorem{thm}{\protect\theoremname}
\theoremstyle{definition}
\newtheorem{example}[thm]{\protect\examplename}
\theoremstyle{remark}
\newtheorem{rem}[thm]{\protect\remarkname}
\theoremstyle{definition}
\newtheorem{defn}[thm]{\protect\definitionname}

\@ifundefined{date}{}{\date{}}
\providecommand{\tabularnewline}{\\}

\usepackage[ruled,vlined,linesnumbered]{algorithm2e}

\theoremstyle{plain}
\RequirePackage{amsthm,amsmath,amsfonts,amssymb}
\RequirePackage[authoryear]{natbib}%
\RequirePackage[colorlinks,citecolor=blue,urlcolor=blue]{hyperref}
\RequirePackage{graphicx,soul}

\theoremstyle{plain}
\newtheorem{theorem}{Theorem}[section]
\newtheorem{definition}{Definition}[section]
\newtheorem{lemma}[theorem]{Lemma}\theoremstyle{remark}
\usepackage{enumerate}
\usepackage{multirow}\usepackage{amsfonts}

\usepackage{dsfont}

\usepackage{pgfplots}
\usepackage{subcaption}
\usepgfplotslibrary{fillbetween}
\usetikzlibrary{patterns}
\usepackage{bm}
\usepackage{tikz}

\usepackage{caption}

\usepackage{tkz-graph}
\usepackage{pgfplots}
\usetikzlibrary{calc,patterns}

\usepackage{subcaption}
\usepackage{url}
\usepackage{fancyhdr}
\usepackage{indentfirst}

\usepackage{enumerate}
\usepackage{dsfont}
\usepackage[british]{babel}
\usepackage{comment}
\usepackage{tikz-cd}
\usepackage{amsfonts}
\usepackage{url}
\usepackage{bbm}
\usepackage{fancyhdr}
\usepackage{indentfirst}
\usepackage{enumerate}
\usepackage{dsfont}

\usepackage{cleveref}
\renewcommand{\cite}{\citet}
\usepackage{comment}

\usepackage{algpseudocode}
\renewcommand{\d}{\,\mathrm{d}}
\newcommand{\dd}{\overset{\mathrm{law}}{=}}
\newcommand{\p}{\mathbb{P}}

\newcommand{\Ran}{\mathrm{Ran}}

\newcommand{\var}{\mathrm{Var}} 
\newcommand{\Var}{\mathrm{Var}}   %
\usepackage{xcolor}
\newcommand{\E}{\mathbb{E}}    %
\newcommand{\R}{\mathbb{R}}    %
\newcommand{\N}{\mathbb{N}}    %

\newcommand{\osc}{\mathrm{osc}} 

\usepackage{tikz}
\usetikzlibrary{arrows.meta,positioning}

\DeclareMathOperator*{\argmax}{arg\,max}
\DeclareMathOperator*{\argmin}{arg\,min}

\newcommand{\ee}{\varepsilon}

\newcommand{\n}[1]{\left\lVert#1\right\rVert}

\newcommand{\bM}{\mathbf{M}}
\newcommand{\bN}{\mathbf{N}}

\newcommand{\bX}{\mathbf{X}}
\newcommand{\bx}{\mathbf{x}}
\newcommand{\cR}{\mathcal{R}}

\def\d{\mathrm{d}}
\DeclareMathOperator\supp{supp}

\def\lawis{\buildrel \mathrm{law} \over \sim}

\newcommand{\bone}{ {\mathbbm{1}} }

\newcommand{\G}{\mathcal{G}}

\usepackage{mathrsfs}

\newcommand{\lst}{\preceq_{\mathrm{st}}}

\renewcommand{\ge}{\geqslant}
\renewcommand{\le}{\leqslant}
\renewcommand{\geq}{\geqslant}
\renewcommand{\leq}{\leqslant}
\renewcommand{\epsilon}{\varepsilon}

\usepackage{babel}
\providecommand{\examplename}{Example}
\providecommand{\remarkname}{Remark}
\providecommand{\theoremname}{Theorem}

\author{Zhenyuan Zhang\thanks{Department of Mathematics, Stanford University. Email: zzy@stanford.edu} \and Hengrui Luo\thanks{Corresponding Author. Department of Statistics, Rice University; Computational Research Division, Lawrence Berkeley National Laboratory. Email: hrluo@rice.edu}}

\providecommand{\definitionname}{Definition}

\addto\captionsbritish{\renewcommand{\definitionname}{Definition}}
\addto\captionsbritish{\renewcommand{\examplename}{Example}}
\addto\captionsbritish{}
\addto\captionsbritish{\renewcommand{\remarkname}{Remark}}
\addto\captionsbritish{\renewcommand{\theoremname}{Theorem}}
\addto\captionsenglish{}
\addto\captionsenglish{\renewcommand{\definitionname}{Definition}}
\addto\captionsenglish{\renewcommand{\examplename}{Example}}
\addto\captionsenglish{}
\addto\captionsenglish{\renewcommand{\remarkname}{Remark}}
\addto\captionsenglish{\renewcommand{\theoremname}{Theorem}}

\makeatother

\usepackage{babel}
\providecommand{\definitionname}{Definition}
\providecommand{\examplename}{Example}
\providecommand{\remarkname}{Remark}
\providecommand{\theoremname}{Theorem}

\begin{document}
\title{Stabilizing the Splits through Minimax Decision Trees}
\maketitle

\begin{abstract}

By revisiting the end-cut preference (ECP) phenomenon associated with a single CART (Breiman et al.~(1984)), we introduce MinimaxSplit decision trees, a robust alternative to CART that selects splits by minimizing the worst-case child risk rather than the average risk. For regression, we minimize the maximum within-child squared error; for classification, we minimize the maximum child entropy, yielding a C4.5-compatible criterion. We also study a cyclic variant that deterministically cycles coordinates, leading to our main method of cyclic MinimaxSplit decision trees.

We prove oracle inequalities that cover both regression and classification, under mild marginal non-atomicity conditions. The bounds control the tree’s global excess risk by local worst-case impurities and yield fast convergence rates compared to CART. We extend the analysis to a random-dimension forest variant that subsamples coordinates per node.

Empirically, (cyclic) MinimaxSplit trees and their forests improve over baselines on structured heterogeneous data such as EEG amplitude regression over fixed time horizons and image denoising, framed as non-parametric regression on spatial coordinates. 

\vspace{0.3cm}
\end{abstract}
\tableofcontents{}

\section{Introduction}

Tree-based predictors are classical non-parametric learners, from early
systems such as THAID to the CART formalization \citep{morgan1973thaid,breiman1984classification},
with surveys charting both breadth and limitations \citep{gordon1984almost,loh2014fifty}
and recent work revisiting large-scale aspects \citep{klusowski2024large}.
A persistent and practically important pathology of single-tree greedy
splitting is the \emph{end-cut preference} (ECP): standard impurity/variance
criteria favor splits near node boundaries, especially under weak or no signal.
ECP has long been observed empirically \citep{morgan1973thaid,torgo2001study,su2024smooth,wang2025study}
and is now rigorously characterized in simple settings, where noisy
directions drive splits toward extremes \citep{breiman1984classification,cattaneo2022pointwise}. A simple example of ECP is given by Figure \ref{fig:Comparison_different_trees-1} below; see also Figure \ref{fig:sine}.

For a \emph{single} deep tree, ECP can be harmful: aggressive boundary
splits create highly unbalanced children, amplify estimation noise,
and can stall meaningful refinement along informative coordinates.
However, \citet{ishwaran2015effect} argues that the same tendency
can be \emph{beneficial} in \emph{deep random forests}: when a candidate
split falls on a noise variable, an end-cut keeps one child large,
preserving downstream sample size so later splits (possibly on strong
variables) can recover the signal; even along strong variables, ECP-like
behavior can be adaptive in regions with little local signal. Thus,
while ECP degrades a single tree's stability, it may aid an ensemble
by lowering tree correlation and enabling recovery via aggregation.
At the same time, the forest setting introduces its own costs: ECP
interacts with candidate-set multiplicity and variable-usage bias,
skewing split frequencies and importance metrics, and further reducing
model stability. Existing methods to mitigate ECP include the
delta splitting rule \citep{morgan1973thaid}, regularization of node
size \citep{torgo2001study}, and the Smooth Sigmoid Surrogate \citep{su2024smooth}.
We refer to \citet{wang2025study} for a more comprehensive review
of the ECP.

As illustrated in Figure \ref{fig:Comparison_different_trees-1},
ECP arises from the greedy split search itself. At each node, the
algorithm scans many thresholds and chooses the one with the largest impurity decay. Thresholds near the boundary isolate very
few points; the sample mean of such a tiny child has a high variance,
so even pure noise can create a large difference between child means
and an inflated impurity reduction. Taking the maximum over many candidate
cut-off points amplifies this selection bias---extremal thresholds are
disproportionately likely to `win'. 

Our focus is on \emph{explicitly mitigating ECP at the algorithmic
level} for trees, while retaining the transparency that makes single decision trees 
stable. 
Beyond
prediction, this design goal is 
to generate balanced, nondegenerate
splits that yield partitions whose geometry is easier to understand and
communicate than the heterogeneous, boundary-driven partitions that
arise under ECP in deep trees and ensembles. 

In this paper, we revisit ECP empirically and theoretically, clarifying
the cases in which standard CART-style criteria drift toward extreme cut-off points and
how this differs between single trees and forests \citep{breiman1984classification,morgan1973thaid,torgo2001study,su2024smooth,ishwaran2015effect,wang2025study,cattaneo2022pointwise}.
Then, we introduce a Minimax splitting rule with a cyclic coordinate
schedule that systematically discourages degenerate child nodes, thereby
mitigating ECP while preserving stability. As a consequence, our splitting rule admits a faster convergence guarantee compared to standard CART --- we establish exponential
error decay under mild conditions and relate our guarantees to existing
theory \citep{chi2022asymptotic,syrgkanis2020estimation,wager2015adaptive,mazumder2024convergence,cattaneo2024convergence}.
Conceptually, the analysis
connects to partition-based martingale approximations \citep{simons1970martingale,zhang2023exact}
and complements asymptotic and inferential developments for trees/forests
\citep{klusowski2024large,chi2022asymptotic,syrgkanis2020estimation,wager2015adaptive,mazumder2024convergence,cattaneo2024convergence},
as well as recent scalable partitioning frameworks \citep{liu2023convergence,luo2023sharded,luo2024efficient,liu2024spatial}
and complexity-theoretic perspectives on hierarchical splits \citep{blanc2020universal,o2005every}.


\begin{figure}[t!]
\centering \includegraphics[width=0.95\textwidth]{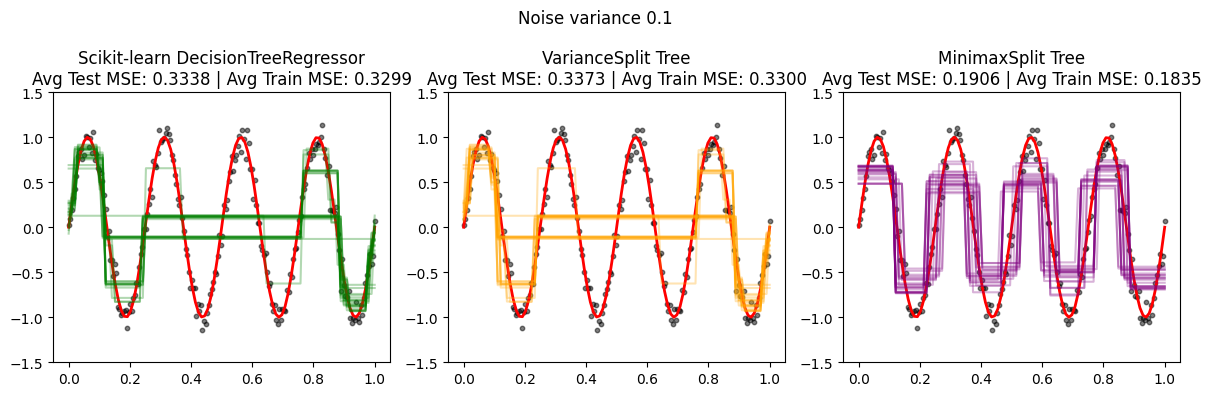}
\\

\includegraphics[width=0.95\textwidth]{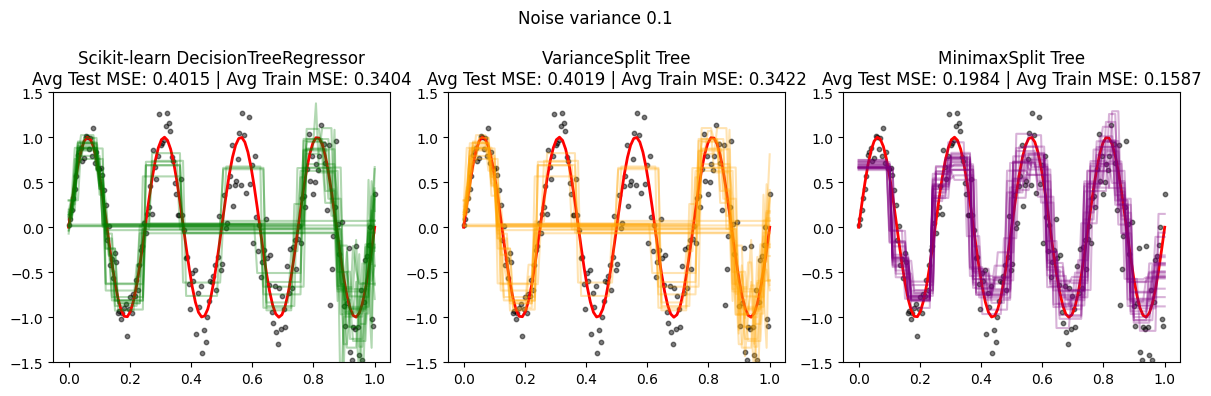} \caption{\label{fig:Comparison_different_trees-1}The comparison of the scikit-learn DecisionTreeRegressor, our MinimaxSplit tree, and VarianceSplit tree
on a sinusoidal target function (in red)  $Y=f(X)+\epsilon=\sin(4\pi X)+\epsilon$,
$X\lawis\mathrm{Unif}(0,1),$ and $\epsilon\lawis \mathcal{N}(0,0.1)$,
with maximum depth 4 (top row) and 5 (bottom row). 
We display 10 different model fits based on 10 different
batches of training sets of size 100, and the average mean squared error (MSE) evaluated
on a different training set across 10 batches.}
\end{figure}

\section{Minimax decision trees for regression}

\label{sec:minimax}

The goal of this section is to construct (cyclic) MinimaxSplit decision tree
algorithms that avoid ECP and attain exponential convergence. We first focus on the regression setting and we discuss the classification setting in Section \ref{sec:Minimax-classify}. In short, in
CART-style regression, a decision tree greedily splits in order to minimize
the \textit{sum} of variances over the responses of the child nodes \citep{liu2023convergence,breiman1984classification}; instead, we propose to greedily minimize the \textit{maximum} variance among the
child nodes. The roadmap of this section is as follows.
\begin{itemize}
    \item In Section \ref{sec:formulation}, we rigorously introduce
the algorithms (summarized in Figure \ref{fig:alg summary}). Crucially, we distinguish two distinct algorithms, where we may optimize over all dimensions at each split (the MinimaxSplit algorithm) or we may pre-select a cyclic schedule of dimensions for each depth (the cyclic MinimaxSplit algorithm). 
\item In Section \ref{sec:marg atomless}, we study exponential rates for marginally atomless
joint distributions (that may not arise from empirical measures $\mathbb{P}_{N}$) for the cyclic MinimaxSplit algorithm.
\item Section \ref{sec:empirical}
builds upon the marginally atomless case, covering more generally the non-marginally
atomless setting and, in particular, the empirical risk bounds for the cyclic MinimaxSplit algorithm.
\item In Section \ref{sec:ASP}, we discuss why there are no general analogous results for the MinimaxSplit algorithm, even though it is shown to perform well empirically even in high-dimensional settings (Appendix \ref{sec:high-dim}). Intuitively, the MinimaxSplit algorithm may fail to break symmetry if strong symmetry is imposed in the underlying samples.
\end{itemize}
  All
proofs can be found in Appendix \ref{sec:proofs}. Throughout, we use the notation $[d]=\{1,2,\dots,d\}$ for $d\in\N$. All splits we consider in this paper will be binary.

\subsection{Algorithm formulation}

\label{sec:formulation}

We begin by recapping the greedy splitting regime in the VarianceSplit
algorithm to construct an efficient decision tree. We need the following
notion of a splittable set. A similar notion is also introduced by
Definition 3.1 of \citet{klusowski2024large}.

\begin{definition} Consider a joint distribution $(\bX,Y)$ on $\R^{d+1}$. We say a set $A\subseteq\R^{d}$
is $(\bX,Y)$-\textit{unsplittable}, or \textit{unsplittable},
if any of the following occurs: \begin{itemize}\par 

\item $\p(\bX\in A)=0$;\par 

\item $\p(\bX\in A)>0$ and $Y\mid\bX\in A$ is a constant;\par 

\item $\p(\bX\in A)>0$ and $\bX\mid\bX\in A$ is a constant. \end{itemize}
Otherwise, we say that $A$ is $(\bX,Y)$-\textit{splittable}, or
\textit{splittable}. \end{definition}

The greedy VarianceSplit construction (Section 8.4 of \citet{breiman1984classification})
proceeds by introducing nested partitions $\{\pi_{k}\}_{k\geq0}$
of $\R^{d}$ with axis-aligned borders, in a way that sequentially
and greedily minimizes the total risk $\E[(Y-M_{k})^{2}]$ at each
step $k$ (or depth $k$, while splitting from the fixed partition
$\pi_{k-1}$), where $M_{k}$ is taken as the conditional mean of
$Y$ given the $\sigma$-algebra generated by the collection $\bone_{\{\bX\in A\}},\,A\in\pi_{k}$.
So, the partitions $\{\pi_{k}\}_{k\geq0}$ of the input space $\R^d$ are
nested and we require that each $\pi_{k}$ consists of sets that split
each splittable element $A\in\pi_{k-1}$ into two hyper-rectangles
with axis-aligned borders. 

Formally, suppose that a hyper-rectangle $A=[a_{1},b_{1})\times\dots\times[a_{d},b_{d})$
belongs to the partition $\pi_{k-1}$ and is splittable (if $A$ is unsplittable, we do not split $A$ and $A\in\pi_\ell$ for every $\ell\geq k$). In particular, $M_{k-1}(\omega)=\E[Y\mid\bX\in A]$
if $\bX(\omega)\in A$. 
To find a split of $A$ at depth $k$, the VarianceSplit algorithm
looks for a dimension $j\in[d]$ and covariate $x_{j}\in[a_{j},b_{j})$
such that the remaining risk is the smallest after splitting $A$
at $x_{j}$ in covariate $j$. In other words, one looks for 
\begin{align}
\begin{split}(j,\hat{x}_{j}) & =\argmin_{\substack{j\in[d]\\
x_{j}\in[a_{j},b_{j})
}
}\Big(\p(\bX\in A,\,X_{j}<x_{j})\var(Y\mid\bX\in A,\,X_{j}<x_{j})\\
 & \hspace{2cm}+\p(\bX\in A,\,X_{j}\geq x_{j})\var(Y\mid\bX\in A,\,X_{j}\geq x_{j})\Big)\\
 & =\argmin_{\substack{j\in[d]\\
x_{j}\in[a_{j},b_{j})
}
}\Big(\E\big[(Y-\E[Y\mid\bX\in A,\,X_{j}<x_{j}])^{2}\bone_{\{\bX\in A,\,X_{j}<x_{j}\}}\big]\\
 & \hspace{2cm}+\E\big[(Y-\E[Y\mid\bX\in A,\,X_{j}\geq x_{j}])^{2}\bone_{\{\bX\in A,\,X_{j}\geq x_{j}\}}\big]\Big)\\
 & =:\argmin_{\substack{j\in[d]\\
x_{j}\in[a_{j},b_{j})
}
}(V_{\text{left}}(j)+V_{\text{right}}(j)),
\end{split}
\label{eq:variance tree def}
\end{align}
where we break ties arbitrarily.\footnote{Strictly speaking, the $\argmin$ in \eqref{eq:variance tree def}
may not always be attained for general joint distributions $(\bX,Y)$,
unless we assume that $\bX$ is marginally either atomless or has
a finite support, which will always be the case in our analysis (such
as the atomicity condition in Theorem \ref{thm:atomic} below). However,
in the general setting, if we allow splits in which atoms can be duplicated
and assigned to both nodes (instead of only on the right node as in
\eqref{eq:variance tree def}), $\argmin$ can be attained. In this
paper, we will implicitly assume that $\argmin$ is always attained
unless otherwise stated. \label{foot}} Consequently, the splittable set $A$ splits into the sets $A_{L}:=[a_{1},b_{1})\times\dots\times[a_{j},\hat{x}_{j})\times\dots\times[a_{d},b_{d})$
and $A_{R}:=[a_{1},b_{1})\times\dots\times[\hat{x}_{j},b_{j})\times\dots\times[a_{d},b_{d})$
to form its two descendants in the set $\pi_{k}$, and define 
\begin{align}
M_{k}(\omega)=\begin{cases}
\E[Y\mid\bX\in A,\,X_{j}<\hat{x}_{j}] & \text{ if }\bX(\omega)\in A_{L};\\
\E[Y\mid\bX\in A,\,X_{j}\geq \hat{x}_{j}] & \text{ if }\bX(\omega)\in A_{R}.
\end{cases}\label{eq:Yk def}
\end{align}
In other words, at depth $k\geq0$ and after constructing the partition
$\pi_{k}$, define $M_{k}$ by 
\begin{align}
M_{k}(\omega)=\E[Y\mid\bX\in A],\quad\text{if }\bX(\omega)\in A,\quad\text{for }A\in\pi_{k}.\label{eq:Mk!!}
\end{align}
This leads to an associated process $\{M_{k}\}_{k\geq0}$ with the joint distribution $(\bX,Y)$.\footnote{The process $\{M_{k}\}_{k\geq0}$ is a martingale,
but we will not use this fact in the rest of the main text.} To connect $M_{k}$ with the estimator $\hat{m}_{k}$
of the regression function at step $k$, we note that for any $\bx\in\R^{d}$,
there exists a unique $A_{\bx}\in\pi_{k}$ containing $\bx$, and then
$\hat{m}_{k}(\bx)=\E[Y\mid\bX\in A_{\bx}]$. 

\begin{figure}[t!]
\centering \includegraphics[width=0.95\textwidth]{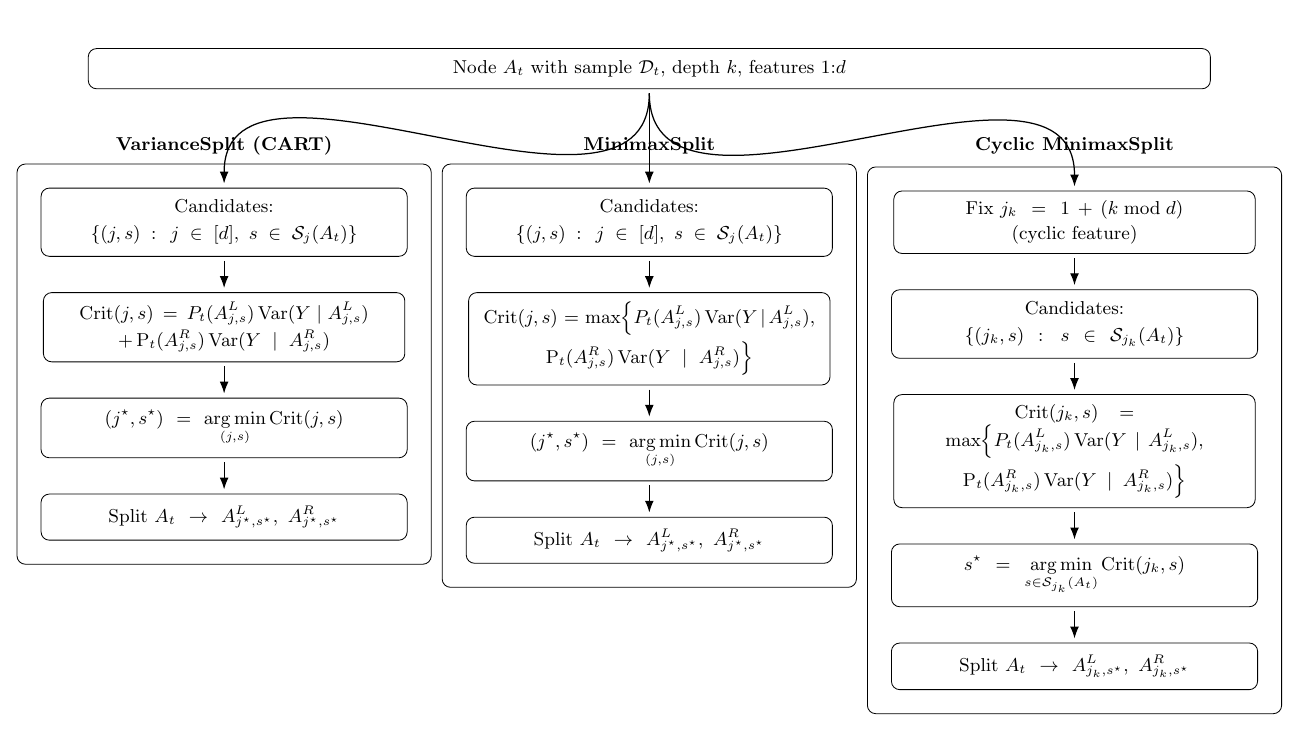}
\caption{The comparison of the VarianceSplit,
MinimaxSplit and cyclic MinimaxSplit trees. Here we use the notation:\ $A_{j,s}^{L}=A_{t}\cap\{x_{j}< s\}$,\ $A_{j,s}^{R}=A_{t}\cap\{x_{j}\geq s\}$;\ $\mathcal{S}_{j}(A_{t})$
are candidate thresholds (e.g., midpoints of sorted $x_{j}$ in $A_{t}$);\ $P_{t}(B)=\p(\bX\in B\mid \bX\in A_{t})$.
}
\label{fig:alg summary}
\end{figure}

\paragraph{The MinimaxSplit algorithm.}

We now introduce the MinimaxSplit algorithm. A common feature of the
MinimaxSplit and VarianceSplit algorithms is that they both start
from a nested sequence of partitions $\{\pi_{k}\}_{k\geq0}$ of $\R^{d}$
consisting of (at-most) binary splits into hyper-rectangles with axis-aligned
borders. The two algorithms differ in the way the partition is constructed:
in the MinimaxSplit setting, the split is chosen to minimize the maximum
variance among the two descendants. Formally, for a splittable set
$A=[a_{1},b_{1})\times\dots\times[a_{d},b_{d})\in\pi_{k-1}$, define
the split location as 
\begin{align}
\begin{split}(j,\hat{x}_{j}) & =\argmin_{\substack{j\in[d]\\
x_{j}\in[a_{j},b_{j})
}
}\max\Big\{\E[(Y-\E[Y\mid\bX\in A,\,X_{j}<x_{j}])^{2}\bone_{\{\bX\in A,\,X_{j}<x_{j}\}}],\\
 & \hspace{3cm}\E[(Y-\E[Y\mid\bX\in A,\,X_{j}\geq x_{j}])^{2}\bone_{\{\bX\in A,\,X_{j}\geq x_{j}\}}]\Big\}\\
 & =\argmin_{\substack{j\in[d]\\
x_{j}\in[a_{j},b_{j})
}
}\max\Big\{\p(\bX\in A,\,X_{j}<x_{j})\var(Y\mid\bX\in A,\,X_{j}<x_{j}),\\
 & \hspace{3cm}\p(\bX\in A,\,X_{j}\geq x_{j})\var(Y\mid\bX\in A,\,X_{j}\geq x_{j})\Big\}\\
 & =\argmin_{\substack{j\in[d]\\
x_{j}\in[a_{j},b_{j})
}
}\max\Big\{ V_{\text{left}}(j),V_{\text{right}}(j)\Big\},
\end{split}
\label{eq:minimaxpt}
\end{align}
where we break ties arbitrarily. We call this the \textit{MinimaxSplit
rule}. After obtaining the partition $\pi_{k}$, define $M_{k}$ through
\eqref{eq:Mk!!}. 

\paragraph{The cyclic MinimaxSplit algorithm.}

The cyclic MinimaxSplit algorithm is a variation of the MinimaxSplit
algorithm. Instead of optimizing over all covariates $j\in[d]$ in
\eqref{eq:minimaxpt}, we cycle through the $d$ covariates as the
tree grows. That is, for $k\geq0$, let $j=j_{k}=1+(k~(\text{mod } d))$.
For all splittable sets $A\in\pi_{k}$, we split $A$ in the $j$-th
coordinate to form its two descendants at depth $k+1$. The split location
is then defined as 
\begin{align}
\begin{split}\hat{x}_{j} & =\argmin_{x_{j}\in[a_{j},b_{j})}\max\Big\{\E[(Y-\E[Y\mid\bX\in A,\,X_{j}<x_{j}])^{2}\bone_{\{\bX\in A,\,X_{j}<x_{j}\}}],\\
 & \hspace{3cm}\E[(Y-\E[Y\mid\bX\in A,\,X_{j}\geq x_{j}])^{2}\bone_{\{\bX\in A,\,X_{j}\geq x_{j}\}}]\Big\}\\
 & =\argmin_{x_{j}\in[a_{j},b_{j})}\max\Big\{\p(\bX\in A,\,X_{j}<x_{j})\var(Y\mid\bX\in A,\,X_{j}<x_{j}),\\
 & \hspace{3cm}\p(\bX\in A,\,X_{j}\geq x_{j})\var(Y\mid\bX\in A,\,X_{j}\geq x_{j})\Big\},
\end{split}
\label{eq:minimax split point}
\end{align}
where we break ties arbitrarily. Define $M_{k}$ by \eqref{eq:Mk!!}. 

If $d=1$, the cyclic MinimaxSplit algorithm coincides with the MinimaxSplit
algorithm. If $d>1$, compared to the MinimaxSplit algorithm, the
cyclic MinimaxSplit algorithm is technically more tractable. Moreover,
since all features are considered equally, the cyclic MinimaxSplit
algorithm appears more robust in recovering the underlying geometry
of the regression function, which we illustrate numerically in Section
\ref{sec:minimax numerics}. On the other hand, the cyclic MinimaxSplit
algorithm may be less effective in higher dimensions because the relevant
dimensions may not be easily identified or reached in the first few
splits, yet it avoids a pathological example in Section \ref{sec:ASP}.\\

\subsection{Risk bound }

\label{sec:marg atomless}

The goal of this section is to justify the exponential convergence
under the cyclic MinimaxSplit rule. Assuming marginal non-atomicity,
we show that up to universal constants and model mis-specification,
the approximation error decays exponentially with rate $r^{k}$ for
some explicit $r\in(0,1)$ that depends only on dimension $d$. 

We note here that a crucial difference from the statistics literature is that our analysis of the MinimaxSplit rule first focuses on \textit{general
joint distributions} $(\bX,Y)$ on $\R^{d+1}$, where $\bX\in\R^{d}$
and $Y\in\R$.\footnote{This can be intuitively regarded as empirical measures with infinitely many samples.} Empirical measures are a special case. Note that this does \textit{not} mean that our analysis is restricted to the noiseless case; see the setting in Section \ref{sec:empirical}. The class of
joint distributions we consider will satisfy the following marginally
atomless property.

\begin{definition} Let $d\geq1$ and $\bX=(X_{1},\dots,X_{d})$ be
an $\R^{d}$-valued random vector. We say that $\bX$ (or its law)
is \textit{marginally atomless} if for any $j\in[d]$ and $x\in\R$,
$\p(X_{j}=x)=0$. \end{definition}

The reason for considering general joint distributions is twofold.
First, we provide more general probabilistic results that may be of
interest in other settings. Second, in our derivation of the empirical
risk bounds, it would be more natural and technically convenient to
first study the marginally atomless setting. Although the empirical
measure is not marginally atomless, only minor changes will be required
to take care of the atomic case (to be discussed in Section \ref{sec:empirical} below). 

Let
us start with some simple observations on the MinimaxSplit rule. The
following lemma is a simple consequence of the total variance formula
and will be proved in Appendix \ref{sec:minimax proof}.

\begin{lemma}\label{lemma:continuous} Suppose that $\bX$ is marginally
atomless and $j\in[d]$. Then the function 
\begin{align}
\varphi_{L}:x_{j}\mapsto\p(\bX\in A,\,X_{j}<x_{j})\var(Y\mid\bX\in A,\,X_{j}<x_{j})\label{eq:mono_loss}
\end{align}
is non-decreasing and continuous. Similarly, 
\[
\varphi_{R}:x_{j}\mapsto\p(\bX\in A,\,X_{j}\geq x_{j})\var(Y\mid\bX\in A,\,X_{j}\geq x_{j})
\]
is non-increasing and continuous. \end{lemma} 

\begin{rem}\label{rem:noncty}
    The marginally atomless property of $\bX$ is not required for the monotonicity properties of the maps $\varphi_L$ and $\varphi_R$, but it is necessary for the continuity. The continuity is essential for an exact $1/2$ risk decay bound (see below), but this argument will be relaxed later in the derivation of Theorem \ref{thm:atomic}.
\end{rem}

Choose an arbitrary element $\hat{x}_{j}\in\argmin_{x_{j}}\max\{\varphi_{L}(x_{j}),\varphi_{R}(x_{j})\}.$
Let us first assume that $\varphi_{L}(\hat{x}_{j})\leq\varphi_{R}(\hat{x}_{j})$.
Define $x_{j}^{*}:=\min\{x_{j}:\varphi_{L}(x_{j})=\varphi_{R}(x_{j})\}$.
By the monotonicity of $\varphi_{L}$ and $\varphi_{R}$, we have
$\hat{x}_{j}\leq x_{j}^{*}$, and thus 
\begin{align}
\varphi_{L}(\hat{x}_{j})\leq\varphi_{L}(x_{j}^{*})=\varphi_{R}(x_{j}^{*})\leq\varphi_{R}(\hat{x}_{j}).\label{eq:vp1}
\end{align}
Using \eqref{eq:vp1} and the fact that $\hat{x}_{j}\in\argmin_{x_{j}}\max\{\varphi_{L}(x_{j}),\varphi_{R}(x_{j})\}$,
we have 
\begin{align}
\varphi_{L}(x_{j}^{*})=\varphi_{R}(x_{j}^{*})=\varphi_{R}(\hat{x}_{j}).\label{eq:vp2}
\end{align}
Since the events $A_{L}=\left\{ \mathbf{X}\in A,X_{j}<x_{j}^{*}\right\} $
and $A_{R}=\left\{ \mathbf{X}\in A,X_{j}\geq x_{j}^{*}\right\} $
form a partition of $\{\bX\in A\}$, we have by the total variance
formula, 
\begin{align}
\begin{split}\varphi_{L}(\hat{x}_{j})+\varphi_{R}(\hat{x}_{j})\leq\varphi_{L}(x_{j}^{*})+\varphi_{R}(x_{j}^{*})\leq\E[(Y-\E[Y\mid\bX\in A])^{2}\bone_{\{\bX\in A\}}].
\end{split}
\label{eq:pvar2}
\end{align}
Combining \eqref{eq:vp2} and \eqref{eq:pvar2}, we have 
\begin{align}
\begin{split}\max\Big\{\varphi_{L}(\hat{x}_{j}),\varphi_{R}(\hat{x}_{j})\Big\}=\varphi_{R}(\hat{x}_{j})\leq\frac{1}{2}\E[(Y-\E[Y\mid\bX\in A])^{2}\bone_{\{\bX\in A\}}].\end{split}
\label{eq:1/2var}
\end{align}
 The other case $\varphi_{L}(\hat{x}_{j})\geq\varphi_{R}(\hat{x}_{j})$
is similar. In other words, the remaining risk on each descendant
is at most half the risk on the parent node. Inductively, we gain
a \textit{uniform} exponential decay of the maximum risk among the nodes
at depth $k$: 
\begin{align}
\max_{A\in\pi_{k}}\p(\bX\in A)\var(Y\mid\bX\in A)=\max_{A\in\pi_{k}}\E[(Y-\E[Y\mid\bX\in A])^{2}\bone_{\{\bX\in A\}}]\leq2^{-k}\var(Y).\label{eq:2-k}
\end{align}
This is the key to controlling the ``sizes of the nodes". In the
case of the VarianceSplit construction, such uniform control is not
possible---counterexamples exist due to ECP, which we explain in
the next example. 
\begin{example}
\label{ex:location regression} We revisit (a special case of) the
location regression model analyzed in \citet{cattaneo2022pointwise}.
Let $(X,Y)$ follow the law of an empirical measure of $\mathrm{Unif}(0,1)\times\mathcal{N}(0,1)$
on $[0,1]\times\R$ with $n$ samples.\footnote{$X$ is not marginally atomless, but the example can always be modified
by smoothing the data without affecting its validity.} Suppose that we perform VarianceSplit for one step. Theorem 4.1 of
\citet{cattaneo2022pointwise} implies that, for $n$ large enough,
the split lies between the first and $\sqrt{n}$-th order statistics
with probability exceeding $1/6$. With high probability (meaning
$1-o(1)$ as $n\to\infty$), the right node has remaining risk at
least $(1-o(1))$ while the total risk at the parent node is at most
$(1+o(1))$. By a union bound, we see that with an arbitrary $\ee>0$,
for sufficiently large $n$, 
\[
\p\Big(\max_{A\in\pi_{1}}\p(X\in A)\var(Y\mid X\in A)\geq(1-\ee)\var(Y)\Big)>\frac{1}{7}.
\]
This example also explains why assumptions such as the sufficient impurity decay (SID) condition
are favorable to ensure risk decay when using the splitting criterion 
\eqref{eq:variance tree def} \citep{chi2022asymptotic,mazumder2024convergence}. Roughly speaking, the SID condition assumes that the best population split removes at least a positive fraction of the within-node variance, leading to a lower bound on local CART progress and eventually an exponential convergence rate in depth for CART and random forests.
\end{example}

To gain further intuition about ECP, we consider the event $\{\mathbf{X}\in A\}$
and sub-events $A_{L}=\left\{ \mathbf{X}\in A,X_{j}<\hat{x}_{j}\right\} $
and $A_{R}=\left\{ \mathbf{X}\in A,X_{j}\geq\hat{x}_{j}\right\} $.
Without loss of generality, we may assume $\mathbb{P}(\bX\in A)=1$ and apply
the total variance formula to this partition, yielding 
\begin{align}
\mathbb{P}(\bX\in A)\var(Y)=\var(Y)&=  \E[\var(Y\mid\sigma(A_{L}))]+\var(\E[Y\mid\sigma(A_{L})])\nonumber \\
\begin{split}&=  \var(Y\mid A_{L})\mathbb{P}(A_{L})+\var(Y\mid A_{R})\mathbb{P}(A_{R})\\
 &\hspace{2cm} +\left(\mathbb{E}[Y\mid A_{L}]-\mathbb{E}[Y\mid A_{R}]\right)^{2}(1-\mathbb{P}(A_{L}))\mathbb{P}(A_{L}).
\end{split}
\label{eq:totalvar2}
\end{align}
From this, we see that when ECP occurs, $\mathbb{P}(A_{L})\approx1$
or $0$, the total remaining risk is dominated by the first two terms of \eqref{eq:totalvar2}:
$\var(Y\mid A_{L})\mathbb{P}(A_{L})+\var(Y\mid A_{R})\mathbb{P}(A_{R})=V_{\text{left}}(j)+V_{\text{right}}(j)$.
Therefore, when ECP occurs, the sum of variances does not necessarily
decay fast enough under VarianceSplit.

We consider the additive function class 
\begin{align*}
\G:=\{g(\bx):=g_{1}(x_{1})+\dots+g_{d}(x_{d})\},
\end{align*}
where $\bx=(x_{1},\dots,x_{d})$. For $g\in\G$, define $\n{g}_{\mathrm{TV}}$
as the infimum of $\sum_{i=1}^{d}\n{g_{i}}_{\mathrm{TV}}$ over all
such additive representations of $g$, where $\n{\cdot}_{\mathrm{TV}}$ denotes the total variation of a univariate function of bounded variation. Note that this is the same type of approximation target used by 
\citet{klusowski2024large,mazumder2024convergence} for CART/C4.5: their oracle inequality is also stated
as an infimum over the additive function class $\G$, with a penalty depending on the additive
total-variation norm.

However, for (cyclic) MinimaxSplit, $\max\{V_{\text{left}}(j),V_{\text{right}}(j)\}$
always displays exponential decay \eqref{eq:2-k}, hence ``avoiding''
the ECP phenomenon. In other words, the fast decay of the sum of variances
is not easily guaranteed from the VarianceSplit construction due to
ECP; but the exponential decay of the maximum variance can be assured
by the MinimaxSplit construction, as we show in Theorem \ref{thm:minimax exp decay 2}.
The MinimaxSplit construction approximately equalizes the two child risks, which implies that the larger child risk is at most one half of the parent risk. This local worst-child contraction is the basic mechanism behind the global error bound below.
\begin{thm}
\label{thm:minimax exp decay 2} Suppose that $\bX$ is marginally
atomless. Let the associated   $\{M_{k}\}_{k\geq0}$ be constructed
from the cyclic MinimaxSplit algorithm. Then uniformly for any $\delta>0$
and $k\geq0$, 
\begin{align}
\begin{split}\E[(Y-M_{k})^{2}] & \leq\inf_{g\in\G}\bigg(\Big((1+\delta)+\frac{2(1+\delta)(1+\delta^{-1})}{3\cdot2^{2\lfloor k/d\rfloor/3}}\Big)\,\E[(Y-g(\bX))^{2}]\\
 & \hspace{2cm}+(1+\delta^{-1})\bigg(\frac{1}{3}+\Big(\frac{1+\delta^{-1}}{4}\Big)^{2/3}\bigg)2^{-2\lfloor k/d\rfloor/3}\n{g}_{\mathrm{TV}}^{2}\bigg).
\end{split}
\label{eq:d-dim rate}
\end{align}
\end{thm}

\begin{rem}[The MinimaxSplit analogue]
Theorem \ref{thm:minimax exp decay 2} does not hold in general for
$d\geq2$ if we replace the cyclic MinimaxSplit algorithm by the MinimaxSplit
algorithm, a phenomenon we illustrate in Section \ref{sec:ASP} below.
Nevertheless, the MinimaxSplit algorithm often achieves desirable
risk decay in \textit{high-dimensional} settings in which cyclic MinimaxSplit is inefficient; see Appendix \ref{sec:high-dim}.\label{rem:minimax} 
\end{rem}

The first general result for the convergence rate of the VarianceSplit algorithm without the variance decay assumption was
obtained by \citet{klusowski2024large}. In a similar vein, \citet{cattaneo2024convergence}
studied the convergence rates of the \textit{oblique} CART and obtained
a $1/k$ upper bound for MSE in the general setting (Theorem 1 therein)
as well as an $r^{k}$ upper bound under further assumptions such
as node sizes (Theorem 4 therein). The unconditional rate guarantee
of $1/k$ appears significantly slower than the exponential decay
rates under variance decay assumptions. (A continuous version of)
Theorem 4.2 of \citet{klusowski2024large} states that if the associated
  $\{M_{k}\}_{k\geq0}$ is constructed from the VarianceSplit algorithm,
then for each $k\geq1$, 
\begin{align}
\E[(Y-M_{k})^{2}]\leq\inf_{g\in\G}\Big(\E[(Y-g(\bX))^{2}]+\frac{\n{g}_{\mathrm{TV}}^{2}}{k+3}\Big).\label{eq:klu result}
\end{align}
Let us compare the rates between \eqref{eq:klu result} and \eqref{eq:d-dim rate}.
Consider $Y=g(\bX)$ with $g\in\G$ and $\ee>0$. The following table
summarizes the convergence rates of the VarianceSplit and MinimaxSplit
algorithms. 
Note that Theorem \ref{thm:minimax exp decay 2} exhibits the curse
of dimensionality, which cannot be circumvented in the current setting
to the best of our knowledge \citep{cattaneo2022pointwise,tan2022cautionary}. 
\begin{center}
\begin{tabular}{ccc}
\toprule 
Splitting rule  & Risk decay guarantee  & Max level required to guarantee MSE$\,\leq\ee$\tabularnewline
\midrule
\midrule 
VarianceSplit \eqref{eq:variance tree def}  & \eqref{eq:klu result}  & $O(1/\ee)$ \tabularnewline
\midrule 
MinimaxSplit \eqref{eq:minimaxpt}  & N/A   & does not exist in general for $d\geq2$ \tabularnewline
\midrule 
Cyclic MinimaxSplit \eqref{eq:minimax split point}  & Theorem \ref{thm:minimax exp decay 2}  & $O(d\log(1/\ee))$ \tabularnewline
\bottomrule
\end{tabular}
\par\end{center}

Since the bound \eqref{eq:d-dim rate} is uniform over $\delta,k,d$,
one can further optimize over $\delta$ for fixed $k,d$. For example,
taking $\delta=2^{-\lfloor k/d\rfloor/10}$, the right-hand side of
\eqref{eq:d-dim rate} then has the more explicit (but cruder) upper
bound 
\begin{align*}
\inf_{g\in\G}\Big(5\E[(Y-g(\bX))^{2}]+2^{1-\lfloor k/d\rfloor/2}\n{g}_{\mathrm{TV}}^{2}\Big).
\end{align*}

\begin{rem}
\label{rem:better bound} In the case where $\n{g}_{\mathrm{TV}}$
is much larger than the range $\Delta g:=\sup g-\inf g$, \eqref{eq:d-dim rate}
can be further improved to 
\begin{align}
\begin{split}\E[(Y-M_{k})^{2}] & \leq\inf_{g\in\G}\bigg((1+\delta)\,\E[(Y-g(\bX))^{2}]+2^{1/3}(1+\delta^{-1})2^{-2\lfloor k/d\rfloor/3}\n{g}_{\mathrm{TV}}^{2/3}\E[(Y-g(\bX))^{2}]^{2/3}\\
 & \hspace{4cm}+2^{-2/3}(1+\delta^{-1})2^{-2\lfloor k/d\rfloor/3}\n{g}_{\mathrm{TV}}^{2/3}(\Delta g)^{4/3}\bigg).
\end{split}
\label{eq:better bound}
\end{align}
Theorems \ref{thm:atomic} and \ref{thm:oracle} below enjoy the same
specialized bounds. 
\end{rem}

In the remainder of this section, we explain the ideas underlying
the proof of Theorem \ref{thm:minimax exp decay 2}, in the special
case where $d=1$ and $Y=g(X)$ for some $g\in\G$. In this case,
\eqref{eq:d-dim rate} reduces to 
\begin{align}
\E[(Y-M_{k})^{2}]\leq L2^{-2k/3}\n{g}_{\mathrm{TV}}^{2}\label{eq:a}
\end{align}
for some $L>0$, which we verify next. 

We employ the following \textit{binary subtree representation} to
control $\E[(M_{k}-M_{k+1})^{2}]$. Let $(\tilde{V},\tilde{E})$ be
a binary tree with root $\emptyset$. We call a subtree of a binary
tree a \textit{binary subtree}. Then, given an at-most-binary sequence
of nested partitions $\{\pi_{k}\}_{k\geq0}$, there exists a subtree
$(V,E)$ of $(\tilde{V},\tilde{E})$ such that each vertex $v\in V_{k}$
can be identified as a set $A\in\pi_{k}$, where $V_{k}$ is the set
of all vertices in $V$ with the graph distance from the root equal
to $k$.\footnote{The tree may not be a binary tree due to possible unsplittable sets.}
With each edge $e=(A_{k},A_{k+1})\in E$ where $A_{k}\in\pi_{k}$
and $A_{k+1}\in\pi_{k+1}$, we may associate a coefficient $p_{e}:=\p(\bX\in A_{k+1})\in[0,1]$.
With each vertex $v=A_{k}$, we may associate a location $\ell_{v}:=\E[Y\mid \bX\in A_{k}]$.
See Figure \ref{fig:1} for an illustration. It follows that $M_{k}$
is supported on the discrete points $\{\ell_{v}\}_{v\in V_{k}}$ and
furthermore, 
\begin{align}
\E[(M_{k}-M_{k+1})^{2}]=\sum_{\substack{v_{k}\in V_{k},\,v_{k+1}\in V_{k+1}\\
v_{k}\sim v_{k+1}
}
}p_{(v_{k},v_{k+1})}(\ell_{v_{k}}-\ell_{v_{k+1}})^{2},\label{eq:binary rep}
\end{align}
where for two vertices $v,w$, we write $v\sim w$ if they are connected. 
Let $\mathcal{F}_{k}$ denote the $\sigma$-algebra generated by the collection $\bone_{\{\bX\in A\}},\,A\in\pi_{k}$. 
The nodewise predictor $M_{k}=\mathbb{E}[Y\mid\mathcal{F}_{k}]$
is a Doob martingale in $L^{2}$, and its risk satisfies 
\begin{align}
    \mathbb{E} \bigl[(Y-M_{k})^{2}\bigr]=\sum_{i=k}^\infty\mathbb{E} \bigl[(M_{i+1}-M_{i})^{2}\bigr]\quad\text{ and }\quad\mathbb{E} \bigl[(M_{k}-M_{k+1})^{2}\bigr]=\sum_{A\in\pi_{k}}\sum_{\substack{A'\in\pi_{k+1}\\A'\subset A}}\mathbb{P}(A')\bigl(\mu_{A'}-\mu_{A}\bigr)^{2},\label{eq:k,k+1}
\end{align}
where we have used that $Y$ is $\sigma(X)$-measurable. 
In other words, convergence rates reduce to controlling the one-step increments
created by each split. The VarianceSplit rule provides no uniform control:
under ECP, a child can have arbitrarily small mass,
so large within-node fluctuations can be confined to microscopic slivers,
and the increment need not contract at a fixed rate. Our MinimaxSplit
instead minimizes $\max\{\mathbb{P}(A_{L})\Var(Y \mid A_{L}),\,\mathbb{P}(A_{R})\Var(Y \mid A_{R})\}$,
ensuring that at every refinement the worst child risk decreases by
a constant factor. Summed along the tree, this yields an exponential
contraction of $\mathbb{E}[(M_{k+1}-M_{k})^{2}]$ and hence of $\mathbb{E}[(Y-M_{k})^{2}]$.
The \emph{cyclic} schedule strengthens this guarantee by forcing each
coordinate to be refined regularly, preventing all contraction from
being spent on a single axis and avoiding symmetry-breaking failures; see Section \ref{sec:ASP}.

\begin{figure}[t!]
\centering \includegraphics[width=0.76\linewidth]{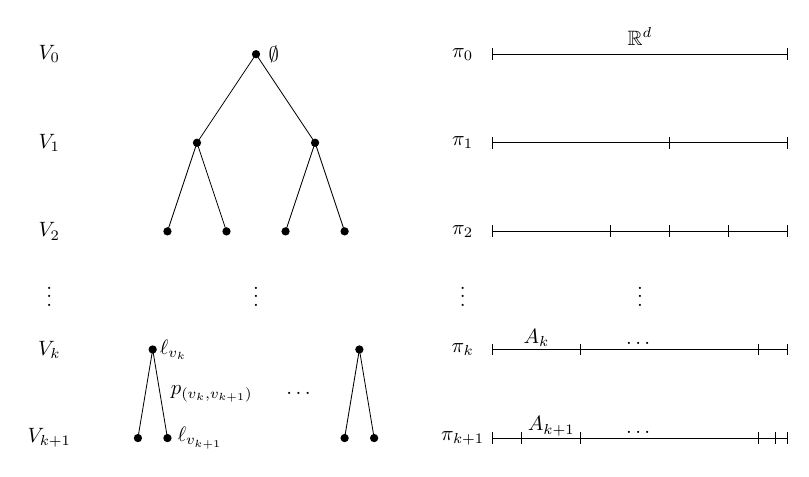}
\caption{Visualizing the nested sequence of partitions $\{\pi_{k}\}_{k\protect\geq0}$
and its associated binary subtree representation. 
}
\label{fig:1} 
\end{figure}

To put the above intuition into play, we apply Hölder's inequality and \eqref{eq:2-k} to obtain 
\begin{align}
\begin{split}\E[(M_{k}-M_{k+1})^{2}] & =\sum_{\substack{v_{k}\in V_{k},\,v_{k+1}\in V_{k+1}\\
v_{k}\sim v_{k+1}
}
}p_{(v_{k},v_{k+1})}(\ell_{v_{k}}-\ell_{v_{k+1}})^{2}\\
 & \leq\bigg(\sum_{\substack{v_{k}\in V_{k},\,v_{k+1}\in V_{k+1}\\
v_{k}\sim v_{k+1}
}
}p_{(v_{k},v_{k+1})}\bigg)^{1/3}\\
 & \hspace{1cm}\times\bigg(\max_{\substack{v_{k}\in V_{k},\,v_{k+1}\in V_{k+1}\\
v_{k}\sim v_{k+1}
}
}p_{(v_{k},v_{k+1})}(\ell_{v_{k}}-\ell_{v_{k+1}})^{2}\sum_{\substack{v_{k}\in V_{k},\,v_{k+1}\in V_{k+1}\\
v_{k}\sim v_{k+1}
}
}|\ell_{v_{k}}-\ell_{v_{k+1}}|\bigg)^{2/3}\\
 & \leq\var(Y)^{2/3}2^{-2k/3}\bigg(\sum_{\substack{v_{k}\in V_{k},\,v_{k+1}\in V_{k+1}\\
v_{k}\sim v_{k+1}
}
}|\ell_{v_{k}}-\ell_{v_{k+1}}|\bigg)^{2/3}.
\end{split}
\label{eq:.A}
\end{align}
Since each split introduces a difference term $|\ell_{v_{k}}-\ell_{v_{k+1}}|$,
which is bounded by the values of $Y$ conditioned on a partition
in the covariate space, we have 
\begin{align}
\sum_{\substack{v_{k}\in V_{k},\,v_{k+1}\in V_{k+1}\\
v_{k}\sim v_{k+1}
}
}|\ell_{v_{k}}-\ell_{v_{k+1}}|\leq\sum_{A\in\pi_{k+1}}\Big(\sup_{A}g-\inf_{A}g\Big)\leq\n{g}_{\mathrm{TV}}.\label{eq:.B}
\end{align}
On the other hand, $\var(Y)=\var(g(X))\leq\n{g}_{\mathrm{TV}}^{2}$.
Inserting \eqref{eq:.B} into \eqref{eq:.A} and using \eqref{eq:k,k+1}, we obtain $\E[(Y-M_{k})^{2}]\leq L2^{-2k/3}\n{g}_{\mathrm{TV}}^{2}$ for some $L>0$.

Our approach is fundamentally different from \citet{klusowski2024large}
in the following sense. On the one hand, the approach of \citet{klusowski2024large} relies on aggregating one-step lower bounds on the impurity gain (or the decay of the risk). On the other hand, recall from \eqref{eq:totalvar2} that the
total risk decreases by the amount 
\begin{align}
\left(\mathbb{E}[Y\mid A_{L}]-\mathbb{E}[Y\mid A_{R}]\right)^{2}(1-\mathbb{P}(A_{L}))\mathbb{P}(A_{L})\label{eq:gain}
\end{align}
after $A$ splits into $A_{L}\cup A_{R}$. The intuition is that although
the gain \eqref{eq:gain} can be small at a certain step, it reveals
information on the underlying joint distribution that guarantees that
either it has already gone through significant risk decay in previous
steps, or it is soon happening, or the total variation $\n{g}_{\mathrm{TV}}$
must be large. Let us illustrate the above intuition by a quick example. 
\begin{example}
Consider the case where $d=1$, $X\lawis\mathrm{Unif}(0,1)$, and $Y=\sin(2\pi2^{p}X)$.
Figure \ref{fig:sine} below plots the remaining risk $\E[(Y-M_{k})^{2}],~1\leq k\leq K$
with power parameters $p=1,2,3,4$ and maximum tree depth $K=12$.
Due to the symmetry of the sine function, the first few splits of
the MinimaxSplit do not reduce the remaining risk, since the splitting
points are points of symmetry. However, as soon as the splitting breaks
through the symmetries, the remaining risk decays exponentially
to zero. The VarianceSplit algorithm makes mild progress at every
step, but is quickly surpassed by MinimaxSplit once the latter breaks
through the symmetries. The reason is that due to ECP, the splitting
points are near the boundary of the nodes, leading to unbalanced nodes.
In other words, while the MinimaxSplit algorithm may not reduce the
risk as steadily as VarianceSplit, the steps with little risk reduction
often implicitly overcome certain hard constraints inherent to the
regression problem (such as the symmetries of the sine function),
paving the way for a much faster risk decay in later stages. Note
that the occurrence of minimal risk reduction from the first few steps
of MinimaxSplit does not contradict the risk bound from Theorem \ref{thm:minimax exp decay 2},
because the upper bound already incorporates the total variation of
the signal function.

\begin{figure}[h!]
\centering \includegraphics[width=0.7\linewidth]{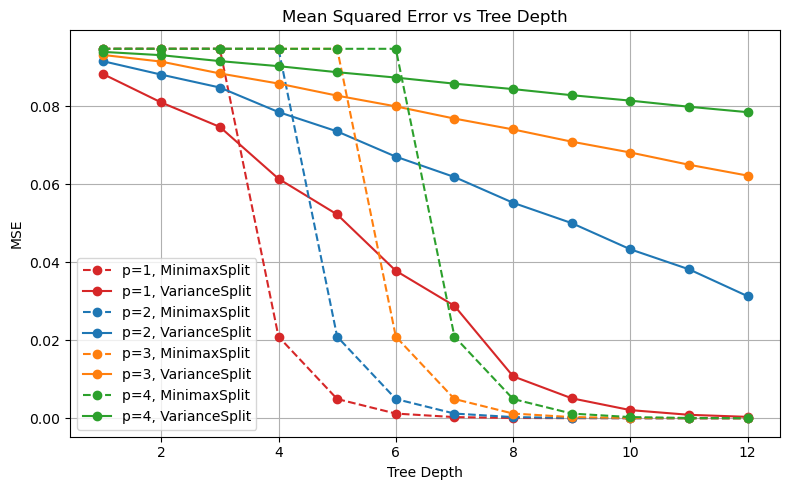} \caption{Comparison of the mean squared errors for the MinimaxSplit and VarianceSplit
regimes, in the setting where $X\lawis\mathrm{Unif}(0,1)$ and $Y=\sin(2\pi2^{p}X)$
with $p=1,2,3,4$ and maximum tree depth $12$. }
\label{fig:sine} 
\end{figure}
\end{example}

\subsection{Empirical risk bound for cyclic MinimaxSplit}

\label{sec:empirical}

So far, we have discussed the cases where the joint distribution $(\bX,Y)$ is marginally atomless, which does not include the class of empirical measures. 
We now focus on the regression setting
and derive finite-sample performance guarantees. Suppose that under
the original law $\p_{*}$, $\bX_{*}$ is marginally atomless and
$Y_{*}=g_{*}(\bX_{*})+\ee,$ where $g_{*}(\bX_{*}):=\E[Y_{*}\mid\bX_{*}]$ and 
the error $\ee=Y_{*}-g_{*}(\bX_{*})$ is sub-Gaussian, i.e., for some
$\sigma>0$, $\p_{*}(|\ee|\geq u)\leq2\exp(-{u^{2}}/({2\sigma^{2}})),~ u\geq0.$
Also, we consider the empirical law of $(\bX,Y)$ based on $N$ samples
from $\p_{*}$.

A limitation of Theorem \ref{thm:minimax exp decay 2}
is that it applies only if the covariate $\bX$ is marginally atomless
(for \eqref{eq:1/2var} to hold). Therefore, we need the following
result 
that deals with measures whose marginals may contain atoms. For a random variable $Y$, denote by $\supp Y$ the support of $Y$. 

When the marginal law of $X$ has atoms, the continuity argument used in the proof of the exact worst-child halving property is no longer available. Nevertheless, if every coordinate atom has mass at most $1/N$, then the failure of exact balancing can be quantified by an additive perturbation of order $N^{-1}$. The next theorem shows that the population bound from the atomless case survives, up to a discrete-resolution correction term.
\begin{thm}
\label{thm:atomic} Suppose that $\bX$ is purely atomic and has a finite support, and for some
$N>0$, 
\begin{align}
\max_{j\in[d]}\,\max_{u\in\R}\p(X_{j}=u)\leq\frac{1}{N}.\label{eq:1/N}
\end{align}
Let $\Delta Y:=\sup\supp Y-\inf\supp Y$, and let the associated   $\{M_{k}\}_{k\geq0}$
be constructed from the cyclic MinimaxSplit algorithm. Then uniformly
for any $\delta>0$ and $k\geq0$, 
\begin{align}
\begin{split}\E[(Y-M_{k})^{2}] & \leq(1+\delta^{-1})2^{-2\lfloor k/d\rfloor/3}\frac{2^{k+2}(\Delta Y)^{2}}{3N}+\inf_{g\in\G}\bigg(\Big((1+\delta)+\frac{2(1+\delta)(1+\delta^{-1})}{3\cdot2^{2\lfloor k/d\rfloor/3}}\Big)\,\E[(Y-g(\bX))^{2}]\\
 & \hspace{6cm}+(1+\delta^{-1})\bigg(\frac{2}{3}+\Big(\frac{1+\delta^{-1}}{4}\Big)^{2/3}\bigg)2^{-2\lfloor k/d\rfloor/3}\n{g}_{\mathrm{TV}}^{2}\bigg).
\end{split}
\label{eq:atomic}
\end{align}
\end{thm}

Comparing \eqref{eq:atomic} and \eqref{eq:d-dim rate}, we see that there is an additional compensation term of $(1+\delta^{-1})2^{-2\lfloor k/d\rfloor/3}\frac{2^{k+2}(\Delta Y)^{2}}{3N}$ for the atomic case. This arises because the functions $\varphi_L$ and $\varphi_R$ are no longer continuous in Lemma \ref{lemma:continuous}, meaning that the maximum risk decay may not always achieve the $1/2$ rate; see Remark \ref{rem:noncty}.


The previous results control the approximation error of the population partition generated by cyclic MinimaxSplit. To obtain a finite-sample oracle inequality, one combines that deterministic approximation bound with the same truncation and empirical-process steps used in the CART analysis of Klusowski and Tian (2024). The effect of the MinimaxSplit rule enters only through the improved approximation term.
\begin{thm}
\label{thm:oracle} Assume that $\bX_{*}$ is marginally atomless.
We have for $\delta\geq2^{-2\lfloor k/d\rfloor/3}$, 
\begin{align*}
\E\big[\n{g_{k}-g_{*}}^{2}\big] & \leq\frac{C2^{k}(\log N)^{2}\log(Nd)}{N}+2\inf_{g\in\G}\bigg(\Big((1+\delta)+\frac{2(1+\delta)(1+\delta^{-1})}{3\cdot2^{2\lfloor k/d\rfloor/3}}\Big)\n{g-g_{*}}^{2}\\
 & \hspace{5.5cm}+(1+\delta^{-1})\bigg(\frac{2}{3}+\Big(\frac{1+\delta^{-1}}{4}\Big)^{2/3}\bigg)2^{-2\lfloor k/d\rfloor/3}\n{g}_{\mathrm{TV}}^{2}\bigg),
\end{align*}
where $C>0$ is a constant depending only on $\n{g_{*}}_{\infty}$
and $\sigma^{2}$. 
\end{thm}

For a fixed sample size $N$, if we assume that $g_{*}\in\G$ and
$k=\lceil(3d\log_{2}N)/(3d+2)\rceil$ is a multiple of $d$, Theorem \ref{thm:oracle}
has the further consequence that (under the same assumptions) 
\begin{align}
\E\big[\n{g_{k}-g_{*}}^{2}\big]\leq CN^{-\frac{2}{3d+2}}\big(\n{g_{*}}^2_{\mathrm{TV}}+(\log N)^{2}\log(Nd)\big).\label{eq:oracle2}
\end{align}
The proof of Theorem \ref{thm:oracle} follows essentially the same
path as Theorem 4.3 for CART of \citet{klusowski2024large} while
replacing Theorem 4.2 therein by our Theorem \ref{thm:atomic}.

As explained in Remark \ref{rem:minimax}, Theorem \ref{thm:oracle} does not hold without extra assumptions if the cyclic MinimaxSplit algorithm is replaced by the MinimaxSplit algorithm. On another note, in \textit{high-dimensional} settings, the MSE decay rate $2^{-2\lfloor k/d\rfloor/3}$ deteriorates under the cyclic schedule. Nevertheless, we provide numerical support in Appendix \ref{sec:high-dim} that the MinimaxSplit algorithm in \textit{high dimensions} still often outperforms other baseline models such as scikit-learn and VarianceSplit.

In addition to the performance and convergence rate guarantees, we briefly comment on the computational efficiency of MinimaxSplit compared to VarianceSplit. Typically, the time complexity consists of sorting and identifying the split dimension and location \citep[Section 5]{louppe2014understanding}. The time complexity for sorting does not depend on the splitting criteria. In terms of solving for the split dimension and location, our MinimaxSplit algorithm \eqref{eq:minimaxpt} (or its cyclic variant \eqref{eq:minimax split point}) is advantageous over VarianceSplit \eqref{eq:variance tree def}, because of the monotonicity of the risk as a function of the split location, thus reducing the time complexity from linear in the node size to logarithmic in the node size.

\subsection{Why cyclic MinimaxSplit outperforms MinimaxSplit}

\label{sec:ASP}

The key reason why the above risk bounds and the oracle inequality are only derived for the
\textit{cyclic} MinimaxSplit algorithm is that there is an \textit{anti-symmetry-breaking
preference} (ASBP) for MinimaxSplit, which we briefly explain in this
section. The gist is that if the underlying model possesses different symmetry conditions across different dimensions, the MinimaxSplit algorithm is likely to always split over the
asymmetric dimensions, thus leaving the symmetric dimensions untouched.
This leads to an undesirable convergence rate, or even inconsistency
in the large-sample regime. We start with a simple example in the
marginally atomless case. 
\begin{example}
\label{ex:ASBP}

Consider $\bX=(X_{1},X_{2})$ uniformly distributed on $[-1,1]^{2}$
and $Y=X_{1}+|X_{2}|$. Clearly, $Y=f(\bX)$ for some $f\in\G$ with
$\n{f}_{\mathrm{TV}}=4$. Recall \eqref{eq:minimaxpt} and consider
a MinimaxSplit step applied to a splittable set $A=[\alpha,\beta)\times[-1,1]$
where $-1\leq\alpha<\beta\leq1$. Along the second dimension $x_{2}$,
the split location of MinimaxSplit in \eqref{eq:minimaxpt}
is always given by $x_{2}=0$, due to the symmetry of the law of $(\bX,Y)$
along $x_{2}=0$ on the set $A$. However, there is no risk decay
given by this split, again due to the above symmetry. On the other
hand, along the first dimension $x_{1}$, the split location of MinimaxSplit in \eqref{eq:minimaxpt} is given by $x_{1}=(\alpha+\beta)/2$
and the risk decay is strictly positive. It follows that the $\argmin$ in \eqref{eq:minimaxpt}
always satisfies $j=1$. By induction, if we
apply the MinimaxSplit algorithm to the joint distribution $(\bX,Y)$,
the split will always be along the first dimension $x_{1}$.

Therefore, if the associated $\{M_{k}\}_{k\geq0}$ is constructed
from the MinimaxSplit algorithm, each $M_{k}$ can be written as a
function of $X_{1}$ only, say $M_{k}=h_{k}(X_{1})$. But then by
the independence of $X_{1}$ and $X_{2}$, for any $k$, 
\begin{align}
\begin{split}\E[(Y-M_{k})^{2}] & =\E[(X_{1}-h_{k}(X_{1})+|X_{2}|)^{2}]\\
 & =\E[(X_{1}-h_{k}(X_{1}))^{2}]+\E[|X_{2}|^{2}]+2\E[(X_{1}-h_{k}(X_{1}))]\E[|X_{2}|]\\
 & \geq\var(X_{1}-h_{k}(X_{1}))+\var(|X_{2}|)+(\E[X_{1}-h_{k}(X_{1})]+\E[|X_{2}|])^{2}\\
 & \geq\var(|X_{2}|)=\frac{1}{12},
\end{split}
\label{eq:1/12}
\end{align}
and thus the MSE does not even decay to zero, meaning that the counterpart of 
\eqref{eq:d-dim rate} cannot hold. 
\end{example}

Example \ref{ex:ASBP} does not involve samples, but it is intuitively
true that ASBP should still hold if one considers empirical measures
of $(\bX,Y)$ with the goal of estimating $f$. This can be made rigorous
using a Donsker class argument, which we discuss in Appendix \ref{sec:Additional-Example-for}.
Example \ref{ex:ASBP} also generalizes to higher dimensions, such as using the
function $Y=f(\bX)$ where $f(x_{1},\dots,x_{d})=x_{1}+\dots+x_{d-1}+|x_{d}|$. 

This ASBP phenomenon has an important consequence for the inadmissibility
of classic random forests in the context of MinimaxSplit. We follow
the setting of Section 7 of \citet{klusowski2024large}, which is
a slight adaptation of Breiman's classic random forest \citep{breiman2001random}. 
A random forest consists of a collection of independent subsampled
trees, where at each node, the split dimensions are optimized among
a random pre-selected set $\mathcal{S}\subseteq[d]$. The set $\mathcal{S}$
is independently randomly generated with a fixed size $m_{\mathrm{try}}\in[d]$ for
each node. The notation $m_{\mathrm{try}}$ will be kept throughout this paper. Consider a random forest under the MinimaxSplit regime,
with $m_{\mathrm{try}}\geq2$ and $g(\bx)=x_{1}+\dots+x_{d-1}+|x_{d}|$ on $[-1,1]^{d}$.
Then each $\mathcal{S}$ must include a dimension in $[d-1]$, along
which the split always reduces more risk than the $d$-th dimension.
This implies that the $d$-th dimension is never reached and the risk
is strictly bounded away from zero for any $k$, using an argument similar to \eqref{eq:1/12}. Therefore, the consistency of the random forest
under the MinimaxSplit regime is not always guaranteed in the large-sample
regime (even if the signal is noiseless), unless we pick $m_{\mathrm{try}}=1$. In
the next section, we analyze the case $m_{\mathrm{try}}=1$, which we call a \textit{random-dimension random forest}, and establish the corresponding
consistency for the MinimaxSplit random forest. The empirical performance will be presented in Section \ref{sec:app denoise}.

\section{Ensemble models}

\label{sec:RF}

\subsection{MinimaxSplit random forests}

Although a single decision tree can capture patterns in data through
hierarchical partitioning, it often suffers from high variance in
predictions, which makes it sensitive to small changes in the training
dataset. 
To address these limitations, an ensemble approach can be utilized
to combine multiple decision trees, leading to improved stability
and predictive accuracy. In this section, we extend the MinimaxSplit
and cyclic MinimaxSplit methods to an ensemble context. 

The ensemble method we consider here is Breiman's classic random forest \citep{breiman2001random},
which aggregates multiple decision trees to improve the precision
and robustness of the prediction. Each tree is grown to a fixed depth
$k$ according to some splitting rule, and the split dimension of each node is optimized among independent selections of random subsets of $[d]$ with a fixed cardinality $m_{\mathrm{try}}$. 
Each tree
in the ensemble is trained on a bootstrap sample, yet our theoretical results do not consider bootstrap for simplicity. Predictions are made by averaging the output of individual trees,
effectively reducing the overall variance of the model (see Algorithm
1 in Appendix \ref{sec:Algorithms}). We call the ensemble the \textit{MinimaxSplit (resp.~VarianceSplit) random forest}, if the base learners follow the MinimaxSplit (resp.~VarianceSplit) algorithm.\footnote{Note that the cyclic MinimaxSplit case does not apply (unless $m_{\mathrm{try}}=1$) due to the lack of an optimization procedure over the dimensions.}



\subsection{Empirical risk bound for MinimaxSplit random-dimension random forests}

\label{sec:RF0}

Our discussion in Section \ref{sec:ASP} indicates that the random
forest approach for the  MinimaxSplit algorithms is better
implemented with $m_{\mathrm{try}}=1$ to guarantee consistency. In other words, we randomize the dimension
at \textit{every} node. As a consequence, applying the MinimaxSplit and the cyclic MinimaxSplit rules yield the same ensemble in distribution. We call the resulting ensemble the \textit{MinimaxSplit random-dimension random forest}.

Recall the regression setting from Section \ref{sec:empirical} and
consider a random forest consisting of trees in which split dimensions
at each node are i.i.d.~sampled from $[d]$, independently of the
data, similar to ExtraTrees \citep{geurts2006extremely} but the randomness lies in choosing the dimension instead of the split location. Fix a depth
$k$ and the number of samples $N$. Denote by $n$ the number of
trees in the forest. Let $\Sigma$ denote the law of the splitting
dimensions and $\E_{n}$ denote the expectation with respect to the
empirical measure on $n$ samples from $\Sigma$. Let $g_k$ be the estimator constructed at depth $k$ from the MinimaxSplit random-dimension random forest. Our goal is
then to bound the $L^{2}$ error $\E\big[\n{\E_{n}[g_{k}]-g_{*}}^{2}\big],$ where
the outer expectation is taken in terms of randomness both from the
empirical samples from $(\bX_{*},Y_{*})$ and from the empirical measure
$\E_{n}$. We will provide empirical evidence of the effectiveness
of such a randomized construction in Section \ref{sec:app denoise}.
\begin{thm}
\label{thm:oracle forest} Assume the same regression setting as in Theorem
\ref{thm:oracle}, and the above random forest definition. We have
\begin{align}
\begin{split}\E\big[\n{\E_{n}[g_{k}]-g_{*}}^{2}\big]\leq & \frac{C2^{k}(\log N)^{2}\log(Nd)+2^{k}(\log N)dn^{-1/3}}{N}\\
 & +2\inf_{g\in\G}\bigg(2\n{g-g_{*}}^{2}+\frac{8}{3}\Big(\n{g-g_{*}}^{2}+\n{g}_{\mathrm{TV}}^{2}\Big)\Big(e^{-\frac{k}{6d}}+\frac{d}{n^{1/3}}\Big)\bigg),
\end{split}
\label{eq:oracle forest}
\end{align}
where $C>0$ is a constant depending only on $\n{g_{*}}_{\infty}$
and $\sigma^{2}$. In particular, if $N/(2^{k}(\log N)^{2}\log(Nd))\to\infty$,
$k/d\to\infty$, and $n/d^{3}\to\infty$, then the MinimaxSplit random-dimension random forest is
consistent. 
\end{thm}

Let us compare Theorem \ref{thm:oracle forest} with the empirical
risk bound for a single tree given by Theorem \ref{thm:oracle}. First,
the exponent $-k/(6d)$ now does not contain a floor function but
has an asymptotically slower rate if $k/d\to\infty$. This is due
to the non-cyclicity of the (random) selection of dimensions. Second,
the formula \eqref{eq:oracle forest} does not involve $\delta$ for
simplicity, although the same argument works for more general choices
of $\delta>0$. 



The restriction $m_{\mathrm{try}}=1$ is theoretically motivated: for $m_{\mathrm{try}}\ge 2$, Example \ref{ex:ASBP} shows that the obstruction to extending Theorem \ref{thm:minimax exp decay 2} from cyclic MinimaxSplit to the unrestricted
MinimaxSplit rule is not the local minimax contraction itself, but rather a \emph{coordinate-starvation}
phenomenon: a signal-bearing coordinate may never be selected along some root-to-leaf paths. 
The next condition rules this out at the level of the \emph{selected} split coordinate.

Let $\Xi$ denote all randomization used by the tree construction (candidate-set sampling and
tie-breaking), and let $\{\pi_k^\Xi\}_{k\ge 0}$ be the associated random partition sequence.
For a splittable cell $A\in \pi_k^\Xi$, write $J^\Xi(A)\in[d]$ for the coordinate along which $A$
is split.

\begin{definition}[Positive visibility on $S$]
\label{def:positive-visibility}
Fix $S\subseteq [d]$ and $\rho\in(0,1]$. We say that the randomized MinimaxSplit tree has
\emph{positive visibility $\rho$ on $S$} if for every $k\ge 0$, every splittable cell
$A\in \pi_k^\Xi$, and every $j\in S$,
\[
\p\!\left(J^\Xi(A)=j \,\middle|\, \pi_k^\Xi\right)\ge \rho
\qquad \text{a.s.}
\]
\end{definition}

In the following, for $S\subseteq[d]$, we consider the additive function class 
\begin{align*}
\G_S:=\Big\{g(\bx):=\sum_{j\in S}g_j(x_j)\Big\},
\end{align*}
where $\bx=(x_{1},\dots,x_{d})$.

\begin{thm} 
\label{thm:random-minimax-visible}
Suppose that $\bX$ is marginally atomless. Let $\{M_k^\Xi\}_{k\ge 0}$ be the predictor sequence
associated with a randomized MinimaxSplit tree, and assume that the tree has positive visibility
$\rho\in(0,1]$ on $S\subseteq [d]$ in the sense of Definition~\ref{def:positive-visibility}. Then uniformly
for any $\delta>0$ and $k\ge 0$,
\begin{align}
\label{eq:visible-random-rate}
\begin{split}
    \E[(Y-M_k^\Xi)^2]
\le
\inf_{g\in \mathcal G_S}
\Bigg(
\left(
(1+\delta)
+
\frac{2(1+\delta)(1+\delta^{-1})}{3}e^{-\rho k/3}
\right)
\E[(Y-g(\bX))^2]+   \\  
(1+\delta^{-1})
\left(
\frac13+\left(\frac{1+\delta^{-1}}{4}\right)^{2/3}
\right)
e^{-\rho k/3}
\|g\|_{\mathrm{TV}}^2
\Bigg),
\end{split}
\end{align}
where the expectation is taken over both $(\bX,Y)$ and the tree randomization $\Xi$.
\end{thm}

In particular, if $\rho$ is bounded below by a positive multiple of $1/s$, where $s:=|S|$, then the
approximation error decays at rate $\exp(-ck/s)$ for a universal constant $c>0$. This reveals that the stochastic visibility scale $1/\rho$ for randomized 
trees with $m_{\mathrm{try}}\ge 2$ is essential to ensure risk bounds. This kind of assumption shares the spirit of mild design assumptions used by \citep{rovckova2019theory}. The same observation also shows that randomized MinimaxSplit \textit{adapts to exact sparsity} under the positive visibility condition. Indeed, as the approximation error decays at rate $\exp(-ck/s)$, the effective dimension is $s$ rather than the ambient dimension $d$.

\section{Minimax decision trees for classification}

\label{sec:Minimax-classify}

Following the C4.5 generalization in \citet{klusowski2024large},
we also provide corresponding results when using MinimaxSplit
for entropy loss. Throughout this section, we consider binary response $Y\in\{\pm 1\}$ with covariates $\bX\in\mathbb{R}^{d}$ (so we consider $(\bX,Y)$ as a joint distribution),
and $\eta(\bx)\coloneqq\mathbb{P}(Y=1\mid \bX=\bx)$ denotes the regression
function. For any measurable cell $A\subset\mathbb{R}^{d}$ we write
\[
\eta_{A}\;\coloneqq\;\mathbb{P}(Y=1\mid \bX\in A),\qquad w(A)\;\coloneqq\;\mathbb{P}(\bX\in A),
\]
and we measure impurity on $A$ with the Shannon entropy 
\begin{align}
    H(Y\mid \bX\in A)\;=\;h(\eta_{A})\;\coloneqq\;-\eta_{A}\log\eta_{A}-(1-\eta_{A})\log(1-\eta_{A}).\label{eq:Hh}
\end{align}
Axis-aligned splits refine a cell $A$ into left and right children
$A_{L}=\{\bx\in A:x_{j}< z\}$ and $A_{R}=A\setminus A_{L}$ for some $j\in[d]$ and $z\in\R$. Standard
C4.5 chooses $(j,z)$ to minimize the weighted sum $w(A_{L})h(\eta_{A_{L}})+w(A_{R})h(\eta_{A_{R}})$.
In the \textit{MinimaxSplit} version, we instead choose $(j,z)$ to minimize $\max\{w(A_{L})h(\eta_{A_{L}}),\,w(A_{R})h(\eta_{A_{R}})\}$,
which explicitly prevents one child from remaining both large and
highly impure. As in the regression setting, when splitting at level $k$, the \textit{cyclic MinimaxSplit} variant fixes
the coordinate at level $k$ to $j_{k}=1+( k\ (\mathrm{mod}\ d))$ and optimizes only over the threshold $z$. 

Given a split rule, we obtain a nested sequence of axis-aligned partitions
$\{\pi_{k}\}_{k\ge0}$ by applying a single binary split per node per level.
For any $\omega$, we let $A_{k}(\omega)\in\pi_{k}$ be the unique
cell containing $\bX(\omega)$. The node statistic we track is the log-odds
of the average class-probability in that cell, 
\begin{align}
    M_{k}(\omega)\;\coloneqq\;\log\Big(\frac{\eta_{A_{k}(\omega)}}{1-\eta_{A_{k}(\omega)}}\Big).\label{eq:Mk-1}
\end{align}
If a cell is pure, we allow $M_k=\pm\infty$ in the natural way; then
$\log(1+e^{-Y M_k})=0$ on that cell. 
Let $g_k(\bx):=\log(\eta_{A_k}/(1-\eta_{A_k}))$ for $\bx\in A_k,~A_k\in\pi_k$. 
The predictor associated with depth $k$ is the plug-in classifier
$2\bone_{\{g_{k}(\bx)\geq 0\}}-1$ (or equivalently, $2\bone_{\{\eta_{A_{k}}\ge{1}/{2}\}}-1$ for the unique $A_k\in\pi_k$ that contains $\bx$) and the canonical
convex surrogate is the conditional cross-entropy on the partition.

With this preparation, the results in this section follow the same
logical arc as in the regression setting. Theorem~\ref{thm:minimax exp decay 2-1}
treats the marginally atomless case and shows that, under the cyclic
MinimaxSplit schedule, the cross entropy contracts exponentially in $\lfloor k/d\rfloor$. 
Theorem~\ref{thm:atomic-1} covers the complementary purely
atomic setting. Here, we need a mild anti-concentration assumption
on coordinate masses; the theorem then delivers an explicit upper bound for $\E[\log(1+e^{-YM_{k}})]$ that reflects the discrete
resolution $N$ of the atoms. Finally, Theorem~\ref{thm:oracle-1} translates these bounds into an oracle inequality for the excess
classification risk of $g_{k}$ under the 0-1 loss via standard calibration
between cross-entropy and misclassification.

\begin{thm}
\label{thm:minimax exp decay 2-1} Suppose that $\bX$ is marginally
atomless. Let the associated process $\{M_{k}\}_{k\geq0}$ be constructed
from \eqref{eq:Mk-1} using the cyclic MinimaxSplit algorithm. Then uniformly for any $k\geq0$,
\begin{align}
\begin{split}\E[\log(1+e^{-YM_{k}})] & \leq\inf_{g\in\G}\Big(\E[\log(1+e^{-Yg(\bX)})]+\sqrt{\frac{\log2}{4}}\,2^{-\lfloor k/d\rfloor/2}\n{g}_{\mathrm{TV}}^{1/2}\Big).\end{split}
\label{eq:d-dim rate-1}
\end{align}
\end{thm}

\begin{thm}
\label{thm:atomic-1} Suppose that $\bX$ is purely atomic and has a finite support, and for
some $N>0$, 
\begin{align}
\max_{j\in[d]}\,\max_{u\in\R}\p(X_{j}=u)\leq\frac{1}{N}.\label{eq:1/N-1}
\end{align}
Let the associated process
$\{M_{k}\}_{k\geq0}$ be constructed from \eqref{eq:Mk-1} using the cyclic MinimaxSplit
algorithm. Then uniformly for any $k\geq0$, 
\begin{align}
\begin{split}\E[\log(1+e^{-YM_{k}})] & \leq2^{-\lfloor k/d\rfloor/2}\frac{2^{k}\log (eN)}{N}+\inf_{g\in\G}\Big(\E[\log(1+e^{-Yg(\bX)})]+\sqrt{\frac{\log2}{4}}\,2^{-\lfloor k/d\rfloor/2}(\n{g}_{\mathrm{TV}}^{1/2}+\frac{\n{g}_{\mathrm{TV}}}{2})\Big).\end{split}
\label{eq:atomic-1}
\end{align}
\end{thm}

The proofs of Theorems \ref{thm:minimax exp decay 2-1} and \ref{thm:atomic-1} follow essentially the same principles as in the regression setting, except that now $\{M_k\}_{k\geq 0}$ no longer forms a martingale. The key is to apply the monotonicity of the map $z\mapsto\p(\bX\in A,\,X_{j}<z)\,h(\p(Y=1\mid\bX\in A,\,X_{j}<z))$, a fact that mirrors Lemma \ref{lemma:continuous} and will be justified in Lemma \ref{lemma:continuous-1} below.

In the next result, we state the corresponding oracle inequality. Following \citet{klusowski2024large}, for a map $g$ we consider the misclassification risk 
$\mathrm{Err}(g):= \p(Y\neq 2\bone_{\{g(\bX)\geq 0\}}-1).$ 
Further, we set $g_*(\bx):=\log(\p(Y=1\mid\bX=\bx)/\p(Y=-1\mid\bX=\bx))$. Recall that $N$ is the number of samples, $d$ is the underlying dimension, and $g_k$ is the output of the cyclic MinimaxSplit construction with depth $k$.

\begin{thm}
\label{thm:oracle-1} Assume that $\bX_{*}$ is marginally atomless.
We have that under the cyclic MinimaxSplit construction,
\begin{align}
\begin{split}\E[\mathrm{Err}(g_{k})]-\mathrm{Err}(g_{*})\leq L\Big( & \frac{2^{k}(\log N)^{2}\log(Nd)}{N}\Big)^{1/4}\\
 & +\inf_{g\in\G}\Big(\n{g-g_{*}}^{2}+\Big({\frac{\log2}{4}}\Big)^{1/4}\,2^{1-\lfloor k/d\rfloor/4}(\n{g}_{\mathrm{TV}}^{1/4}+\frac{\n{g}_{\mathrm{TV}}^{1/2}}{\sqrt{2}})\Big),
\end{split}
\label{eq:C4.5 oracle}
\end{align}
where $L>0$ is a positive universal constant. 
\end{thm}

\begin{rem}\label{rem:c4.5}
The bound \eqref{eq:C4.5 oracle} can be relaxed to 
\[
\E[\mathrm{Err}(g_{k})]-\mathrm{Err}(g_{*})\leq L\Big(\frac{2^{k}(\log N)^{2}\log(Nd)}{N}\Big)^{1/4}+\inf_{g\in\G}\Big(\n{g-g_{*}}^{2}+L2^{-\lfloor k/d\rfloor/4}(\n{g}_{\mathrm{TV}}+1)\Big)
\]
for some large constant $L>0$, using the rough inequality $x\leq \min\{x^2+1,x^4+1\}$. 
Compared to the oracle inequality of \citet{klusowski2024large},
we see that we have achieved exponential convergence rates and removed
the factor $\n{g}_{\infty}$. 
\end{rem}

We provide additional classification experiments using MinimaxSplit in Appendix \ref{sec:appendix-classification-synthetic}.



\section{Numerical experiments and applications}

\label{sec:minimax numerics}

In this section, we provide experiments that verify the advantages
of MinimaxSplit and cyclic MinimaxSplit methods in low-dimensional
domains, as well as integrated ensemble methods applied to approximating
a real-valued function and denoising images. Performance metrics are
defined in Appendix \ref{subsec:Performance-Metrics}, and further
numerical experiments (including those in high dimensions) are provided
in Appendix \ref{app:further numerics}.

\subsection{Avoiding end-cut preference for stable cuts}

\label{sec:ecp numerics}

To show that the ECP is remedied in MinimaxSplit, we consider a special case of the
location regression model in \citep{cattaneo2022pointwise}, similar to the discussion in Example \ref{ex:location regression}. In our experiment, we generate independent
replicates with $X\lawis\mathrm{Unif}(0,1)$ and $Y$
following the standard normal or Student's $t$ distributions (pure noise), sort by $X$,
and scan all possible splits in the experiments in Figure \ref{fig:min_child_prop}.
We also provide two baselines: Uniform random cutting picks a split
point entirely at random within the current predictor range, ensuring that there is
no systematic bias, yet ignoring any data-driven signal. Sampling an
observed value similarly randomizes split location among actual samples.
Both serve as benchmarks: they eliminate ECP, 
but do not offer an impurity-guided improvement mechanism.

For each replicate, we record ${\min(n_{L},n_{R})}/{n}$, the relative size of the smaller child, and we choose either VarianceSplit (the standard CART impurity reduction) or MinimaxSplit. As the tail of the noise becomes heavier, VarianceSplit becomes
increasingly prone to end-cuts, isolating tiny child nodes almost
every time under Student's t with one degree of freedom. In contrast, the
output of MinimaxSplit remains more tightly concentrated near perfectly
balanced halves across all three distributions, especially for the light-tailed noise cases, thus preventing 
small-node failures. Random uniform and random observed--point cuts
likewise avoid pathological end-cuts but do so at the cost of entirely
unguided splits. Thus, if one requires a data‐adaptive binary split
that resists extreme outliers without surrendering balance, the MinimaxSplit
criterion clearly outperforms VarianceSplit (or CART) and random baselines.

\begin{figure}[h!]
\centering 
 \includegraphics[width=0.32\linewidth]{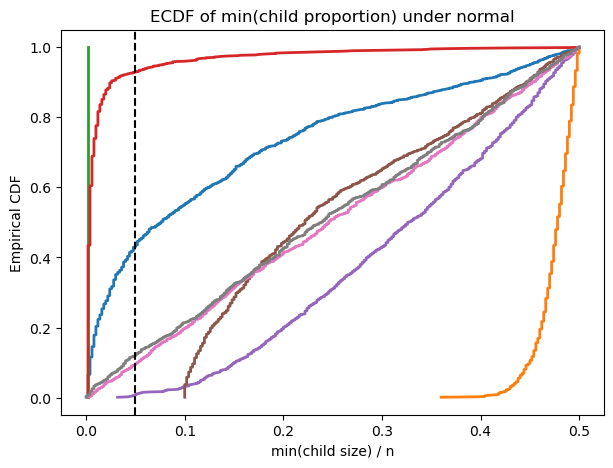}
\includegraphics[width=0.32\linewidth]{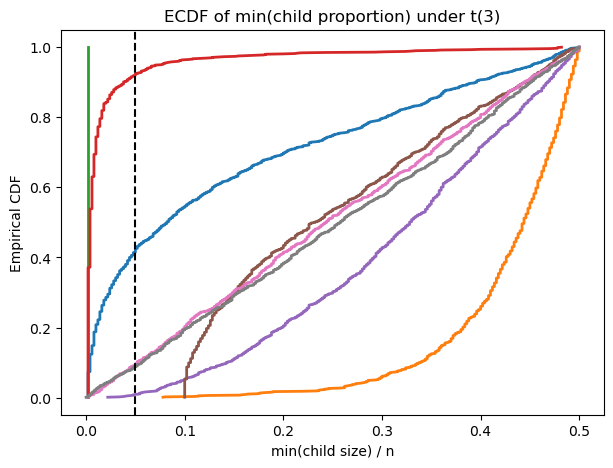}
\includegraphics[width=0.32\linewidth]{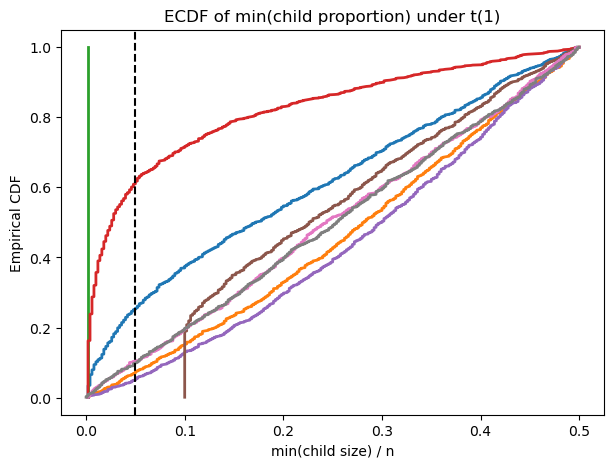}
\includegraphics[width=0.9\linewidth]{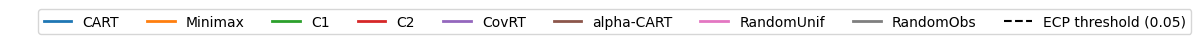}
\caption{Empirical CDFs of the smaller child proportion $\min(n_{L},n_{R})/n$ for sample size $n=500$, based on $1000$ replicates when splitting
pure-noise data generated by standard normal (left), $t(3)$ (middle),
and $t(1)$ (right) distributions. We also use dashed lines to denote the threshold 
0.05.}
\label{fig:min_child_prop} 
\end{figure}

\begin{figure}[h!]
\centering \includegraphics[width=0.48\linewidth]{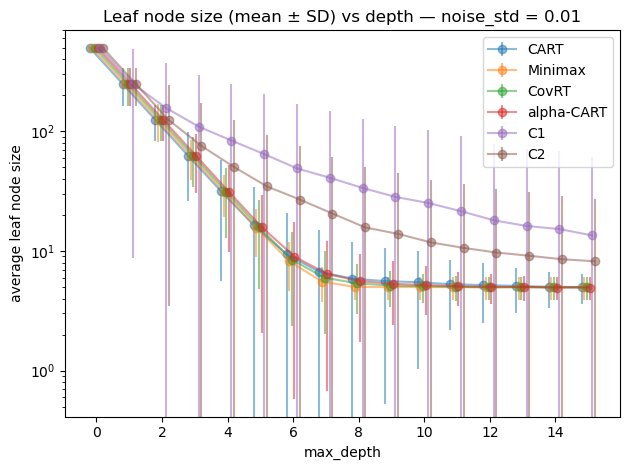}
\includegraphics[width=0.48\linewidth]{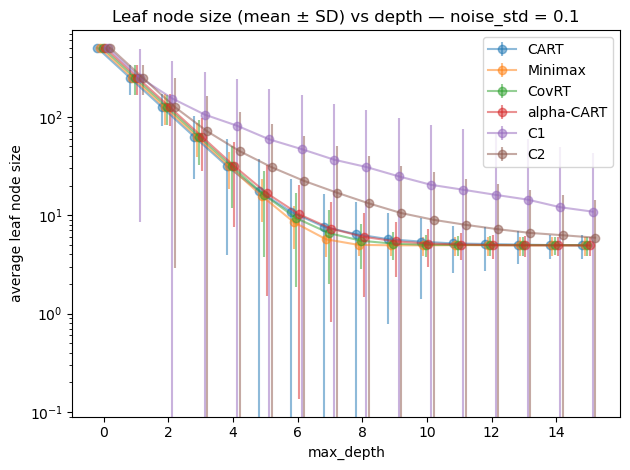} \caption{Average leaf size (mean number of samples per leaf) versus maximum
tree depth for six splitting criteria (VarianceSplit (CART), MinimaxSplit, C1, C2 \citep{buja2001data}, CovRT \citep{zhang_ma_2026_covrt}, and alpha-CART \cite{wager2015adaptive}) at two noise levels. Error bars show one standard deviation across repeated
fits. The vertical axis is on a log scale to expose imbalance: small
and collapsing leaf means with low variability (as for MinimaxSplit and,
to a lesser extent, CART) indicate balanced partitioning, whereas
the slower decay and large spread for C1/C2 reflect persistent ECP that peels off tiny fragments and leaves dominant branches.
The signal function is given by \eqref{eq:signal f}.  Increasing the centered Gaussian noise from $\sigma=0.01$ to $\sigma=0.1$
slightly attenuates rates but does not erase the structural contrast.}
\label{fig:leaf_size_depth} 
\end{figure}

End-cut preference propagates imbalance down the tree: a tiny ``peeled-off’’
child at one split creates a cascade of uneven leaves. To quantify
this, in the next experiment, for each fully grown tree we compute 
\begin{align*}
    &\text{average leaf node size}=\frac{1}{L}\sum_{j=1}^{L}n_{j},\\
    &\text{standard deviation of leaf node size}=\sqrt{\frac{1}{L}\sum_{j=1}^{L}(n_{j}-\text{avg of leaf node size})^{2}},
\end{align*}
where $L$ is the number of leaves and $n_{j}$ is the size of leaf
$j$. Small average leaf node sizes and collapsed error bars signal
balanced partitions; large means with huge variance indicate persistent
end-cuts leaving one dominant branch and tiny fragments. Our main methodological contribution, the MinimaxSplit algorithm, can
be used with ensembles and is applicable to real data. It is also
related in spirit to the C1 and C2 criteria of \citet{buja2001data}, 
which emphasize one-sided purity and can better identify extreme subsets
than traditional CART. Figure \ref{fig:leaf_size_depth} records the
leaf distributions as the maximum depth of the tree grows, with $X\lawis\mathrm{Unif}(0,1)$ and signal
\begin{align}
    f(x)=\begin{cases}
\sin(x), & 0\leq x<\frac{1}{3};\\
-2x, & \frac{1}{3}\le x<\frac{2}{3};\\
0, & \text{otherwise}.
\end{cases}\label{eq:signal f}
\end{align}
MinimaxSplit quickly reduces average leaf size and variability, yielding
balanced leaves. VarianceSplit (CART), C1, and C2 shrink mean leaf size slowly and
retain a large spread: they repeatedly cut tiny end pieces while leaving
a large branch, showing ECP.

\subsection{Fixed-horizon EEG amplitude regression}

We now examine whether the theoretical advantages of MinimaxSplit and
cyclic MinimaxSplit carry over to a non-parametric time-series regression
problem built from the Bonn electroencephalography (EEG) corpus \citep{andrzejak2001}.\footnote{\sloppy as explained at \url{https://www.mathworks.com/help/wavelet/ug/time-frequency-convolutional-network-for-eeg-data-classification.html}}
Each EEG record is a single-channel segment of length $T\approx23.6$
seconds sampled at $173.61$ Hz. 
We formulate a scalar regression task: given
time $t\in[0,T]$ within a fixed horizon $T$, predict the standardized
EEG amplitude $y(t)$. This time-only design isolates the splitting
behavior along a single coordinate, precisely the regime in which
ECP is the most consequential for greedy impurity
splits and where our minimax approach is expected to stabilize partitions.

Our preprocessing follows two principles. First, we standardize the amplitudes
within each segment to zero mean and unit variance. This removes irrelevant
scale without obscuring the short-range temporal variability that
the tree must fit. Second, we optionally downsample by an integer
factor to control computational budget while preserving the temporal
ordering. Both steps are label-free and therefore introduce no leakage.
For evaluation, we fix a horizon $T=20\,$s, draw train-test splits
inside $[0,T]$ by a random holdout, and
keep the split fixed across methods and depths so that any performance
difference reflects only the splitting rule. These choices echo our
theoretical emphasis on how the split criterion shapes the geometry
of the partition.

\begin{figure}[t!]
\centering

\includegraphics[width=0.46\textwidth]{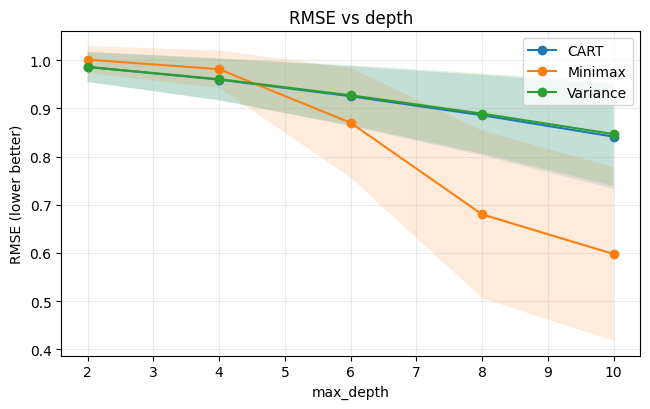}\includegraphics[width=0.46\textwidth]{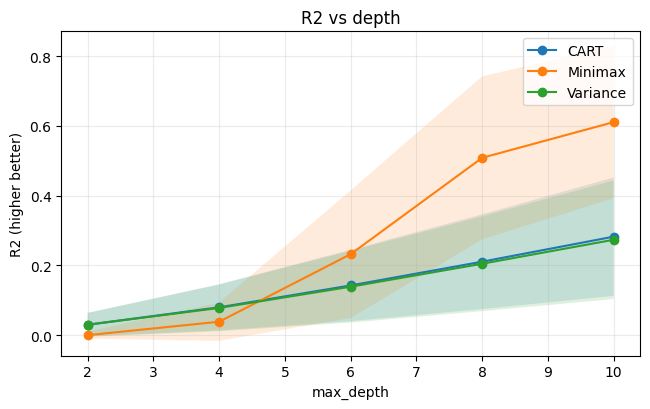}

\includegraphics[width=0.95\textwidth]{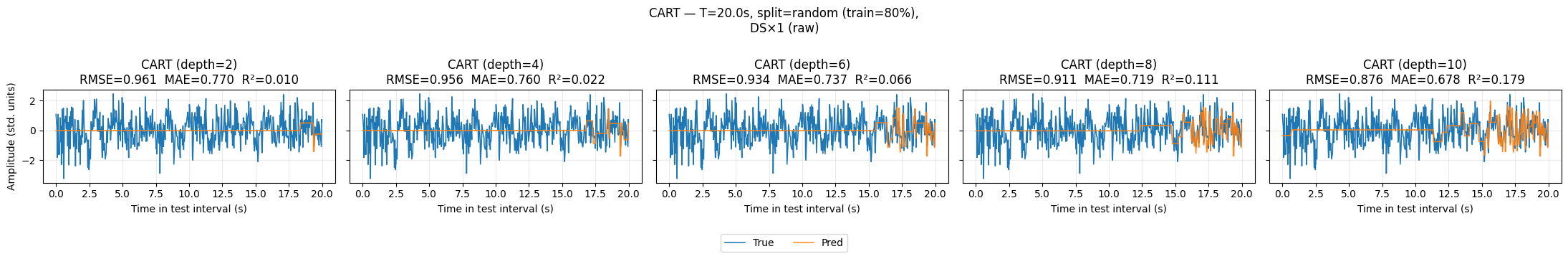}

\includegraphics[width=0.95\textwidth]{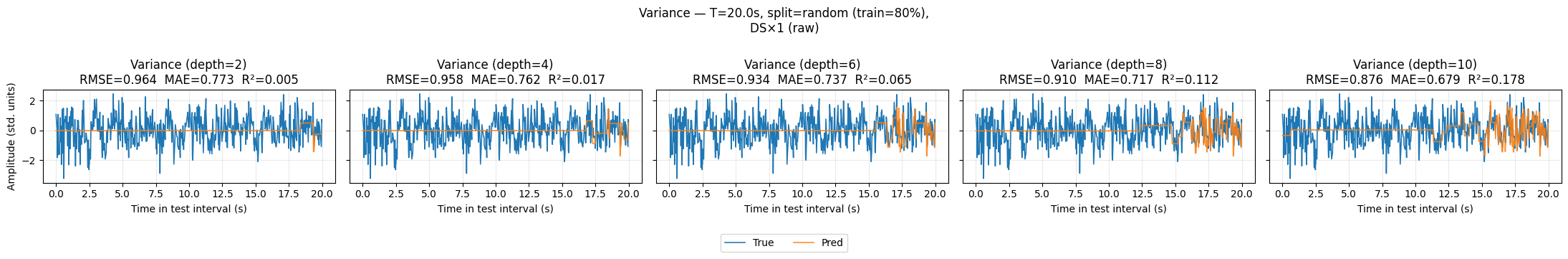}

\includegraphics[width=0.95\textwidth]{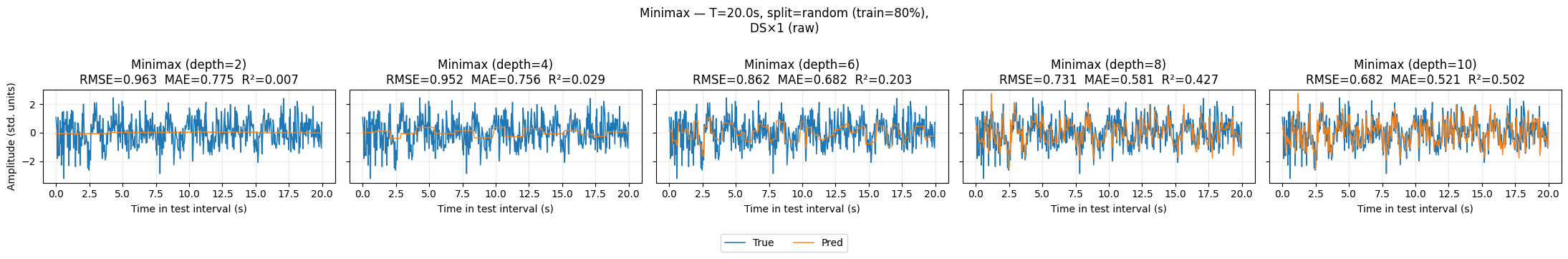}

\caption{\label{fig:EEG-compare}Fixed-horizon EEG amplitude regression with
trees. The top row summarizes performance across Bonn EEG segments
when predicting standardized amplitude $y(t)$ from time only over
a $T=20\LyXThinSpace $s window (random 80/20 split within the window;
identical splits reused for all methods; no downsampling). Curves
show mean RMSE (left; lower is better) and mean $R^{2}$ (right; higher
is better) versus tree depth, with shaded bands denoting $\pm$standard
deviation across segments. The mosaics below plot, for one representative
segment, the true (blue) and predicted (orange) test-interval traces for
each method at depths $k\in\{2,4,6,8,10\}$ along with the reported RMSE/MAE/$R^{2}$.}
\end{figure}

In Figure \ref{fig:EEG-compare}, we compare three trees with identical
hyperparameters (maximum split size $(=2)$ and stopping rules kept fixed):
standard CART using squared-error reduction (with pruning), its variant VarianceSplit, and our
MinimaxSplit which minimizes the worst child variance. 
The experimental protocol is intentionally minimal so that the fitted
functions reveal the partition geometry. Input features consist only
of time $t$ and targets are the standardized amplitudes
$y(t)$. We train each tree on the training subset within $[0,T]$
and report test RMSE, MAE, and the coefficient of determination $R^{2}$ on the holdout portion, aggregating
results across all available EEG segments. Because the same split
is reused across methods, depth-wise comparisons are matched at the
sample level. We also visualize the predicted curves against the ground
truth on a representative segment to expose the qualitative effect
of each split rule.

The results align with the theoretical expectations. 
As the tree depth increases, CART and the sum-of-variances
baselines increasingly display end-cut behavior, producing piecewise constant
fits that flatten toward the edges of the interval and degrade extrapolation
between local extrema. MinimaxSplit, by contrast, maintains near-balanced
children throughout the path, which yields visibly tighter fits to
the oscillatory structure of the signal and systematically lower errors. In addition, MinimaxSplit reduces error and raises $R^{2}$ markedly as depth
increases, while CART and the VarianceSplit baseline improve more
slowly. 
The trend is most pronounced in the random holdout regime where interpolation
dominates; here, the additional leaves available to deeper trees are
useful only if the split rule does not collapse sample sizes, and
the MinimaxSplit criterion preserves effective sample size by construction.
This observation matches our finding that MinimaxSplit concentrates
the smaller-child proportion away from zero and one across noise laws,
thereby eliminating pathological small-node failures observed under
impurity-based CART.

\subsection{Image denoising with ensembles}

\label{sec:app denoise}

\begin{figure}[t!]
\centering

\includegraphics[width=0.95\textwidth]{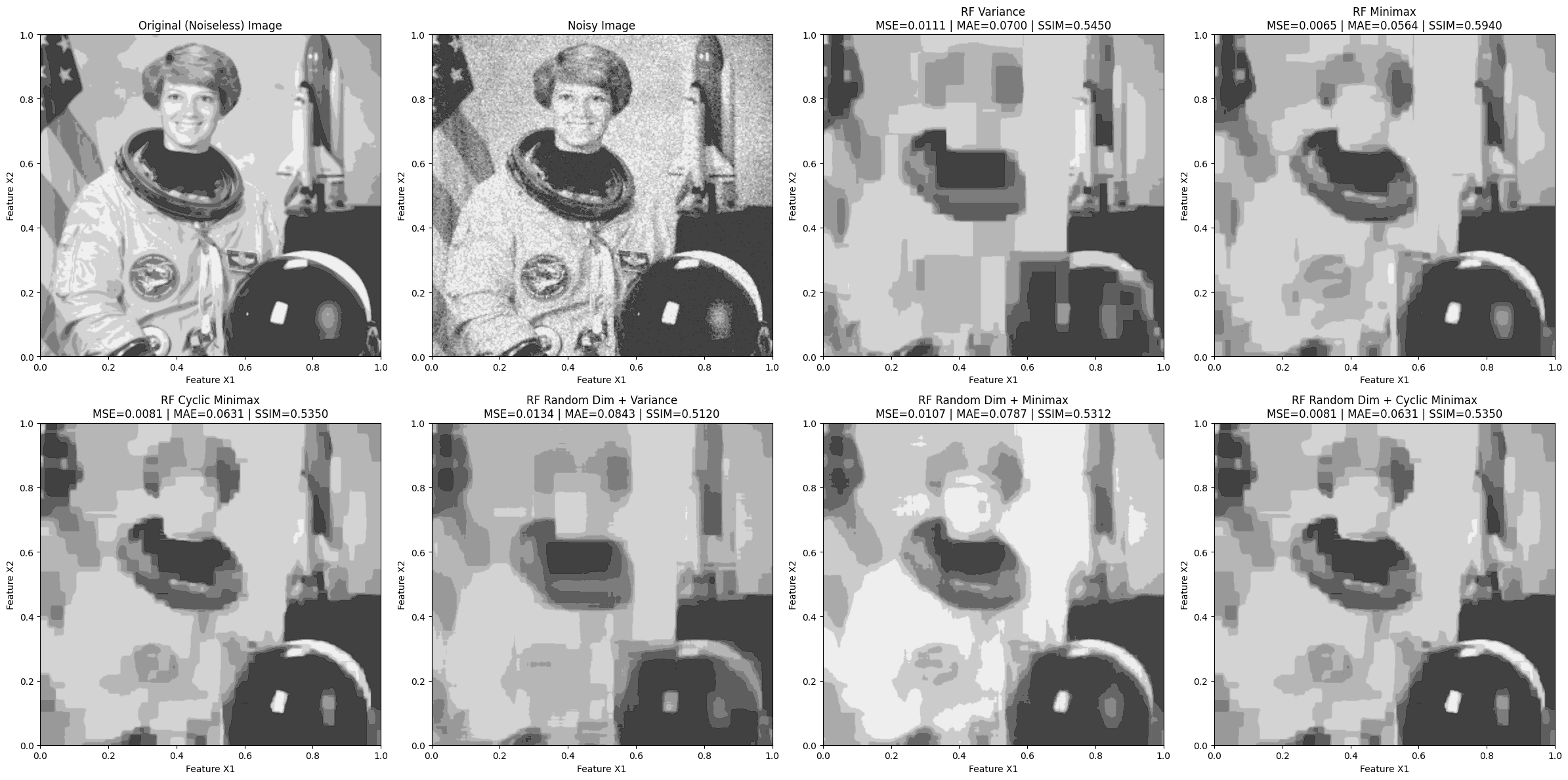}

\caption{\label{fig:Astr_forest}Comparative analysis of various random forest (depth
$K=10$, 50 estimators) approaches for image denoising using different
splitting strategies. The experiment compares six random forest configurations
on a noisy astronaut image: original noiseless image; noisy
input; three  
greedy splitting methods ($m_{\mathrm{try}}=d$) using
consistent splitting criteria across all weak learners (trees) --- 
VarianceSplit, MinimaxSplit, and  Cyclic MinimaxSplit; three  random-dimension methods
where feature selection is randomized at each node (i.e., $m_{\mathrm{try}}=1$) while location
selection follows the three different criteria, as described in Section
\ref{sec:RF0}. }
\end{figure}

\begin{figure}[t!]
\centering

\includegraphics[width=0.95\textwidth]{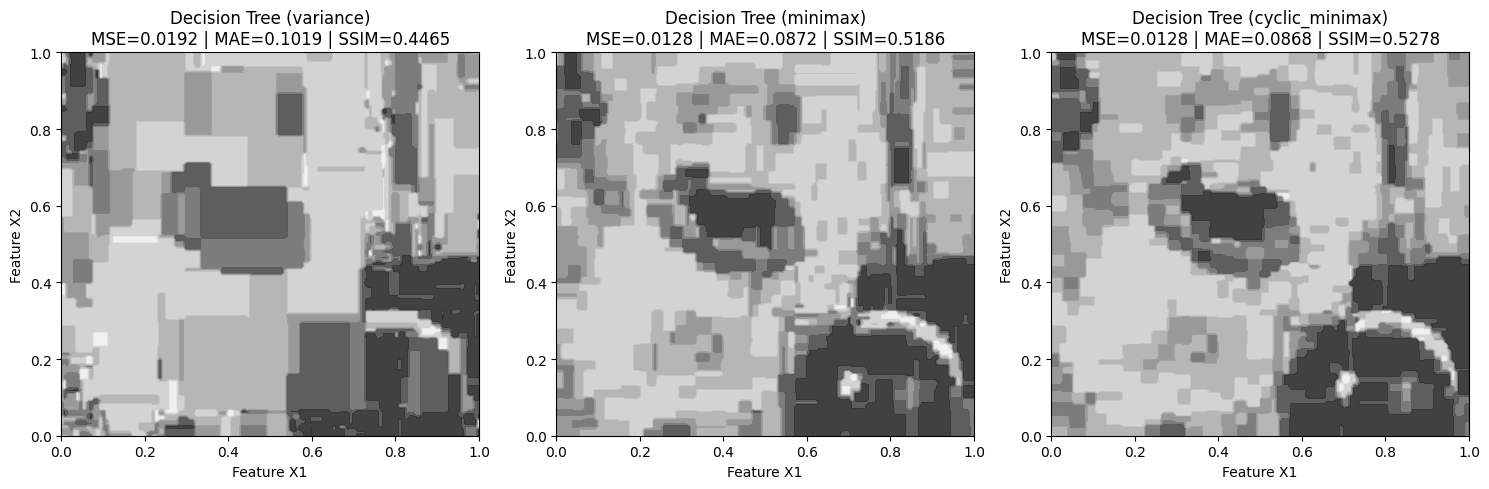}

\caption{\label{fig:Astr_singletree}Comparison of decision tree regression
methods (depth $K=10$) for image reconstruction using different
splitting criteria --- VarianceSplit, MinimaxSplit, and cyclic MinimaxSplit. Each method was trained on noisy astronaut
image data with spatial coordinates as features and pixel intensities
as targets. }
\label{fig:single tree }
\end{figure}

Another notable experiment involves applying decision trees to denoise images \citep{luo2024efficient}. We apply the different tree
variants to predict pixel values based on their locations, effectively
treating this as a regression problem. The experiments shown in Figures
\ref{fig:Astr_forest} and \ref{fig:Astr_singletree} compare decision
tree and random forest regression methods for denoising the classic Astronaut \citep{van2014scikit}
grayscale image, which has been preprocessed to a $256\times256$ image. 

The goal is to recover a clean grayscale image $f:[0,1]^{2} \to [0,1]$
from noisy observations $Y=f(\bX)+\varepsilon$ at pixel centers $\bX=(X_{1},X_{2})$
laid out on a $256\times256$ grid (equipped with the uniform distribution). We normalize intensities to $[0,1]$
and add i.i.d.\ Gaussian noise $\varepsilon\lawis\mathcal{N}(0,\,0.1^{2})$.
Each pixel becomes one regression sample, and we fit axis-aligned
trees that approximate the unknown $f$ as a piecewise constant function
on a dyadic partition, which is precisely the regime covered by our
analysis of axis-aligned trees and their ensembles.

We compare six random forest variants built from our splitting rules:
VarianceSplit, MinimaxSplit, and cyclic MinimaxSplit, each with or
without random feature subsampling at every node ($m_{\mathrm{try}}=1$ feature
per split). All forests consist of $50$ trees, each grown
to depth $K=10$; trees are bootstrapped and predictions are averaged.

Figure~\ref{fig:Astr_forest} shows qualitative reconstructions for
the six forests. The VarianceSplit forest leaves noticeable grain
in flat regions and blurs edges at depth~$10$.
In contrast, both Minimax ensembles suppress speckle while preserving
high-contrast features such as facial contours and texture edges.
Further, across all metrics,
the MinimaxSplit forests outperform the VarianceSplit forest, with the largest gains
in SSIM, indicating better structural fidelity. 
The improvements
are coherent with Theorem~\ref{thm:oracle forest}:  the worst-child control given by the MinimaxSplit rule  prevents high-variance
children from persisting across levels. The single-tree performance is shown in Figure \ref{fig:single tree }.

The experiments support two messages from the theory. First, enforcing
a worst-child improvement at each split translates into visibly sharper
reconstructions and uniformly better pixelwise metrics. Second, randomizing
the split dimension ($m_{\mathrm{try}}=1$) is not an arbitrary choice and ensures the consistency of Minimax forests and practical
stability.

\section{Conclusion}

In this study, we reexamine the well-known phenomenon of ECP, which is associated with traditional CART decision trees \cite{breiman1984classification}. To remedy this undesirable behavior for single decision trees \cite{ishwaran2010consistency}, we introduce a novel decision tree splitting strategy,
the MinimaxSplit algorithm, along with its multivariate variant, the
cyclic MinimaxSplit algorithm. Unlike traditional VarianceSplit methods,
which aim to reduce overall variance in the child nodes, our MinimaxSplit
algorithm seeks to minimize the maximum variance within the split
partitions, thereby reducing the risk of overly biased partitions.
The cyclic MinimaxSplit algorithm further ensures that each dimension
is used in a balanced manner throughout the tree construction process,
avoiding dominance by a subset of features. We do not claim that this
new method could outperform the classical CART-like splitting under
all circumstances, yet we point out that our construction avoids ECP
and attains a better convergence rate guarantee as the tree grows deeper. The (non-cyclic) MinimaxSplit construction features the anti-symmetry-breaking preference---consistency is not guaranteed in certain symmetric cases,  which can, however, be alleviated by using ensemble methods. 

A key theoretical methodological development of this study is that
the cyclic MinimaxSplit algorithm achieves an exponential MSE (or logistic loss) decay
rate given sufficient data samples, without requiring additional
variance decay assumptions in regression and classification settings (Theorems \ref{thm:atomic} and \ref{thm:atomic-1}). Furthermore,
we derive empirical risk bounds for this method, establishing its
robustness in different settings (Theorems \ref{thm:oracle} and \ref{thm:oracle-1}). Using the same powerful techniques used for trees, we prove novel
results on the convergence rates of univariate partition-based martingale
approximations as a by-product (see Appendix \ref{sec:m}),
which are of independent interest.

In addition to single-tree regression analysis, we explore ensemble
learning approaches, where we combine decision trees constructed
using MinimaxSplit, cyclic MinimaxSplit, and VarianceSplit methods
within a random forest framework. We also  demonstrate a
hybrid approach 
with different splitting techniques at different layers when dealing with non-uniform data distributions
(Section \ref{sec:2}). These findings suggest that the MinimaxSplit
approach can provide more stable and adaptive decision trees compared
to traditional methods, especially in low-dimensional settings. In future work, we hope to introduce changes in the cut family and search scope \citep{ishwaran2026super} to enhance its high-dimensional behavior. 

\section*{Acknowledgment}

The authors thank Ruodu Wang for his valuable input. We provide our
reproducible code at \url{github.com/hrluo}. 

 \bibliographystyle{plainnat}
\bibliography{main}

\appendix

\section{On partition-based martingale approximations}

\label{sec:m}

In this section, we develop the framework of partition-based martingale
approximations and show that a number of constructions feature exponential convergence
rates. These problems are of independent interest (see the end of Appendix \ref{subsec:Basic-concepts}) and have implications in the regression tree setting. The main takeaway is that:
\begin{itemize}
 \item The techniques developed in this work for analyzing convergence rates of the MinimaxSplit algorithm extend to many other settings, such as convergence rates for martingale approximations;
\item Precise (asymptotic) convergence rates of martingale approximations depend on the initial law, but a uniform exponential rate is always guaranteed under mild conditions; 
   
    \item Applying VarianceSplit to samples $\{(\bX_i,Y_i)\}_{1\leq i\leq n}$ satisfying monotonicity in all $d$ dimensions (meaning that for every $j\in[d]$ and distinct $i,i'\in[n]$, $((\bX_i)_j-(\bX_{i'})_j)(Y_i-Y_{i'})=\sigma_j$ where $\sigma_j\in\{\pm 1\}$ depends only on $j$) does not incur ECP.
\end{itemize}

\subsection{Basic concepts }

\label{subsec:Basic-concepts}

Consider nested partitions $\{\pi_{k}\}_{k\geq0}$ of $\R$ and a
real-valued random variable $U$ with a finite second moment. We use the abbreviation $\sigma(\pi_{k})$ to denote the $\sigma$-algebra
generated by events of the form $\{U\in A\}_{A\in\pi_{k}}$. Define $M_k:=\E[U\mid\sigma(\pi_k)]$. 
If $\Pi:=\{\pi_{k}\}_{k\geq0}$
is a nested sequence of partitions, $\{\sigma(\pi_{k})\}_{k\geq0}$
is a filtration generated by indicator functions of sets in the partitions $\{\pi_{k}\}_{k\geq0}$ \citep{doob1953book}. It follows from the tower property of conditional
expectations that the sequence $\{M_{k}\}_{k\geq0}=\{\E[U\mid\sigma(\pi_{k})]\}_{k\geq0}$
is a martingale (in fact, a Doob martingale \citep{doob1940regularity,doob1953book}).
We call this martingale the \textit{$\Pi$-based martingale approximation}
of the random variable $U$, or in general a \textit{partition-based
martingale approximation} that approximates the random variable $U$ in $L^2$.

Unless $U$ is a constant,
 there are different partition-based martingale approximations depending
on the distribution of $U$. Our goal in this section is to identify
a few partition-based martingale approximations that efficiently approximate
$U$, where the construction algorithm is universal. The efficiency
criterion is given by the decay of the MSE $\E[(U-M_{k})^{2}]$. The MSE is non-increasing in $k$ by the
nested property of $\{\pi_{k}\}_{k\geq0}$ and the total variance
formula. 

Without loss of generality, we give our construction of a partition
of a generic interval $[a,b)\subset\mathbb{R}$, where $a,b\in\R\cup\{\pm\infty\}$.\footnote{Here we slightly abuse notation that $[a,b)=(-\infty,b)$ if $a=-\infty$.}
The sequence of partitions $\Pi$ will be constructed
recursively, where $\pi_{0}=\{\R\}$ and for every $k\geq0$, each
interval $A\in\pi_{k}$ splits into two intervals, by following the
same construction, forming the elements in $\pi_{k+1}$. In the following,
we introduce four distinct splitting rules that define partition-based
martingale approximations. For simplicity, we assume that $U$ is
atomless, so that the endpoints of the intervals do not matter and
that no trivial split occurs. The general case with atoms can be analyzed in a similar way but with more technicality.
\begin{defn}
\label{def:4martingales} Suppose we are given an atomless law of
$U$ and a non-empty interval $I=[a,b)$, where $a,b\in\R\cup\{\pm\infty\}$.

\begin{enumerate}
\item[(i)] Define 
\begin{align}
u_{\text{var}}=\argmin_{u\in I}\big(\p(U\in[a,u))\var(U\mid U\in[a,u))+\p(U\in[u,b))\var(U\mid U\in[u,b))\big).\label{eq:var m def}
\end{align}
If the minimizer is not unique, we pick the largest minimizer. The
\textit{variance} splitting rule (corresponding to the VarianceSplit
algorithm in Section \ref{sec:formulation}) splits $I$ into the
two sets $[a,u_{\text{var}})$ and $[u_{\text{var}},b)$. 
\item[(ii)] Define $u_{\text{Simons}}=\E[U\mid U\in I]$. The \textit{Simons}
splitting rule splits $I$ into the two sets $[a,u_{\text{Simons}})$
and $[u_{\text{Simons}},b)$. 
\item[(iii)] Define 
\[
u_{\text{minimax}}=\argmin_{u\in I}\max\big\{\p(U\in[a,u))\var(U\mid U\in[a,u)),\,\p(U\in[u,b))\var(U\mid U\in[u,b))\big\}.
\]
If the minimizer is not unique, we pick the largest minimizer. The
\textit{minimax} splitting rule (corresponding to the MinimaxSplit
algorithm in Section \ref{sec:formulation}) splits $I$ into the
two sets $[a,u_{\text{minimax}})$ and $[u_{\text{minimax}},b)$. 
\item[(iv)] Define 
\[
u_{\text{median}}=\sup\{u\in I:\p(a\leq U<u)=\p(u\leq U<b)\}.
\]
The \textit{median} splitting rule splits $I$ into two sets $[a,u_{\text{median}})$
and $[u_{\text{median}},b)$. 
\end{enumerate}
In turn, when applying each splitting rule recursively starting from
$\pi_{0}=\{\R\}$, we obtain a nested sequence of partitions $\Pi=\{\pi_{k}\}_{k\geq0}$
depending on the law of $U$. We call the resulting $\Pi$-based martingale
approximation $\{M_{k}\}_{k\geq0}=\{\E[U\mid\sigma(\pi_{k})]\}_{k\geq0}$
the variance (resp.~Simons, minimax, median) martingale (with respect
to $U$). See Figure \ref{fig:Mapprox_ex2} for examples of the four martingale approximations and their convergence rates.

\begin{figure}[t!]
\begin{subfigure}{0.315\textwidth} \includegraphics[width=1\linewidth]{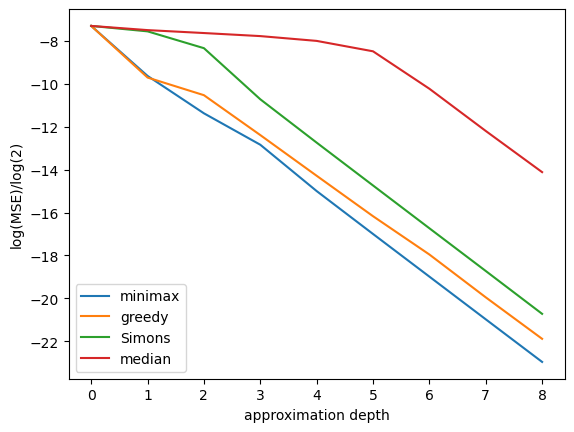}
\caption{}
\label{fig:3a} \end{subfigure} \begin{subfigure}{0.655\textwidth}
\centering \includegraphics[width=0.48\linewidth]{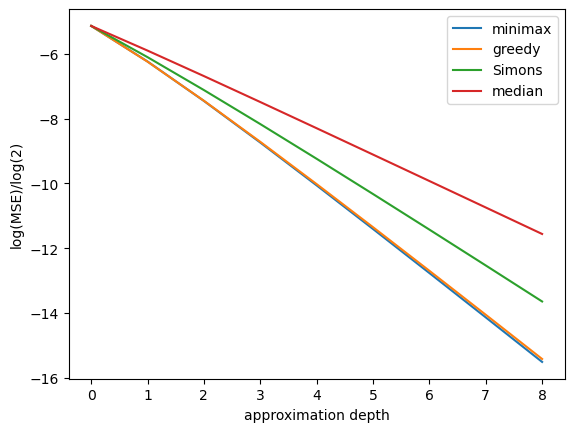}
\includegraphics[width=0.48\linewidth]{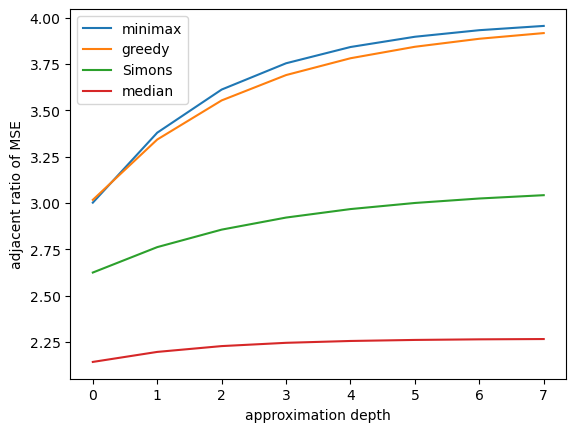} \caption{}
\label{fig:3b} \end{subfigure} \caption{\label{fig:The-left-panel}Plots of the log-MSE, $\log_{2}(\E[(U-M_{k})^{2}])$
versus the approximation depth $k$ for four different methods, where
the density of the law $U$ is given: (a) by $f(u)\propto \bone_{[0,0.9]}(u)+(1+10^{4}(u-0.9))\bone_{[0.9,1]}(u)$ and
(b) by $f(u)\propto u^{10}\bone_{[0,1]}(u)$. The right
panel of (b) shows the ratios $\E[(U-M_{k+1})^{2}]/\E[(U-M_{k})^{2}]$
for the four methods as a function of $k$. }
\label{fig:Mapprox_ex2}
\end{figure}
\end{defn}

\begin{example}
\label{ex:uniform} If $U$ is uniformly distributed on a compact
interval, all four martingales coincide. For example, if $U\lawis\mathrm{Unif}(0,1)$,
each of the four martingales from Definition \ref{def:4martingales}
satisfies $M_{k}=\E[U\mid\sigma(\pi_{k})]$, where $\pi_{k}:=\{[j/2^{k},(j+1)/2^{k}):0\leq j<2^{k}\}$
and it follows that $\E[(U-M_{k})^{2}]=4^{-k}/12$. 
\end{example}

The intuition for the minimax and median martingales is that at each
splitting step, we balance the ``sizes'' of $U\bone_{\{U\in I\}}$
on the two sets. The ``size'' corresponds to the (unconditional)
variance for the minimax martingale and the total probability for
the median martingale. Intending to minimize the MSE $\E[(U-M_{k})^{2}]$
at each step $k$, the variance martingale naturally arises as an
algorithm that greedily reduces the remaining risk within $U$ in
each iteration through layers. 

Motivated by a martingale embedding problem, \citet{simons1970martingale}
first introduced the Simons martingale and established the
a.s.~convergence $M_{k}\to U$ (and hence also in $L^{2}$), but
did not analyze the convergence rate. A recent motivation for studying
the convergence rate of the Simons martingale arises from the construction
of powerful e-values in hypothesis testing \citep{ramdas2024hypothesis}.
\citet{zhang2023exact} (Lemma 5.6) proved that if $U\in L^{2+\delta}$
for some $\delta>0$ and $\{M_{k}\}$ is the Simons martingale, then
there exist $C>0$ and $r\in(0,1)$ such that $\E[(U-M_{k})^{2}]\leq Cr^{k}$,
and $r<0.827$ is feasible if $U$ is bounded. Our Theorem \ref{thm:martingale approx}(ii)
improves the rate to $r=1/2$ and hence provides a tighter theoretical
bound. We also show by example that the rate $r=1/2$ is optimal.

The terminology \textit{martingale approximation} has been used extensively
in the probability literature with different meanings. In \citet{ruschendorf1985wasserstein},
it refers to the best approximation of a random vector by a (single)
martingale based on ideas from optimal transport. In \citet{borovskikh2019martingale}
and \citet{hall2014martingale}, it refers to techniques from martingale
theory (such as inequalities and CLT rates) with various applications
in statistics. Similar techniques are also used in the study of stationary
ergodic sequences \citep{wu2004martingale,zhao2008martingale} and
Markovian walks \citep{grama2018limit}. Note that in our setting,
$U$ is a sum of martingale differences, but whose variance is bounded,
and hence one cannot apply the martingale CLT or its convergence
rates.

In the next two sections, we discuss two types of results: \textit{uniform}
rates for bounded $U$ (Appendix \ref{sec:uniform}) and \textit{non-uniform}
rates (Appendix \ref{sec:nonuniform}) where the asymptotic constant
may depend on the possibly unbounded law of $U$. 

\subsection{Uniform convergence rates}

\label{sec:uniform} 
\begin{thm}
\label{thm:martingale approx} Let $U$ be a $[0,1]$-valued atomless
random variable.
The following statements hold.

\begin{enumerate}
\item[(i)] If $\{M_{k}\}_{k\geq0}$ is the variance martingale, $\E[(U-M_{k})^{2}]\leq2^{-2-2k/3}$. 
\item[(ii)] If $\{M_{k}\}_{k\geq0}$ is the Simons martingale, $\E[(U-M_{k})^{2}]\leq2^{1-k}$. 
\item[(iii)] If $\{M_{k}\}_{k\geq0}$ is the minimax martingale, $\E[(U-M_{k})^{2}]\leq0.4\cdot2^{-2k/3}$. 
\item[(iv)] If $\{M_{k}\}_{k\geq0}$ is the median martingale, $\E[(U-M_{k})^{2}]\leq2^{-k}$. 
\end{enumerate}
\end{thm}

The proof of Theorem \ref{thm:martingale approx} is lengthy and technical, which we defer to Appendix \ref{A.2}.

The general case of a bounded $U$ can be derived by a scaling argument,
since the constructions in Definition \ref{def:4martingales} are
scale-invariant. Finding the \textit{optimal} rate $r\in(0,1)$ (where
$\E[(U-M_{k})^{2}]\leq Lr^{k}$ for some universal constant $L$)
is also an intriguing question. Observe that one cannot have $r<1/4$
by Example \ref{ex:uniform} above. In Examples \ref{ex:simons1/2}
and \ref{ex:median} below,
we show that the rate $r$ in Theorem \ref{thm:martingale approx}
is indeed optimal for the Simons and median martingales.

\begin{example}[Rate $r=1/2$ is optimal for the Simons martingale]
\label{ex:simons1/2} Fix an arbitrary $s\in(1/2,1)$ and let $U$
satisfy $\p(U=-1)=1-s$ and for $J\geq0$, 
\begin{align*}
\p\bigg(U=\sum_{j=1}^{J}\Big(\frac{1-s}{s}\Big)^{j}\bigg)=s^{J+1}(1-s).
\end{align*}
It follows that $U$ is a bounded random variable. Next, we compute
that 
\begin{align*}
\E\bigg[U\mid U\geq\sum_{j=1}^{J}\Big(\frac{1-s}{s}\Big)^{j}\bigg] & =\bigg(\sum_{\ell=J}^{\infty}s^{\ell+1}(1-s)\bigg)^{-1}\sum_{\ell=J}^{\infty}s^{\ell+1}(1-s)\sum_{j=1}^{\ell}\Big(\frac{1-s}{s}\Big)^{j}\\
 & =s^{-(J+1)}\sum_{\ell=J}^{\infty}s^{\ell+1}\frac{(1-s)^2}{2s-1}\bigg(1-\Big(\frac{1-s}{s}\Big)^{\ell}\bigg)\\
 &=\frac{1-s}{2s-1}\bigg(1-\Big(\frac{1-s}{s}\Big)^{J+1}\bigg)\\
 & =\sum_{j=1}^{J+1}\Big(\frac{1-s}{s}\Big)^{j}.
\end{align*}
It is then straightforward to verify that in the Simons martingale
$\{M_{k}\}_{k\geq0}$, each $M_{k+1}-M_{k}$ is supported on the set
$\{0,-((1-s)/s)^{k},((1-s)/s)^{k+1}\}$, where $\p(M_{k+1}-M_{k}=-((1-s)/s)^{k})=(1-s)s^{k}$
and $\p(M_{k+1}-M_{k}=((1-s)/s)^{k+1})=s^{k+1}$. It follows that
\begin{align*}
\E[(M_{k+1}-M_{k})^{2}]=(1-s)^{2k+1}s^{-(k+1)}\asymp\Big(\frac{(1-s)^{2}}{s}\Big)^{k}.
\end{align*}
In other words, the rate $r$ is given by $r=(1-s)^{2}/s$. If $s$
is picked close to $1/2$, $r$ can be made close enough to $1/2$.
Therefore, $r=1/2$ is the optimal rate parameter. This example is
illustrated with an atomic distribution $U$, but a slight twist also
leads to an example for an absolutely continuous law $U$. 
\end{example}

\begin{example}[Rate $r=1/2$ is optimal for the median martingale]
\label{ex:median} Let $s\in(0,1)$ and a bounded random variable
$U$ satisfy 
\begin{align*}
\p\Big(U=\sum_{j=1}^{k-1}s^{j-1}-s^{k-1}\Big)=2^{-k},~k\geq1.
\end{align*}
Next, we compute that 
\begin{align*}
\E\Big[U\mid U\geq\sum_{j=1}^{k}s^{j-1}-s^{k}\Big] & =2^{k}\sum_{\ell=k+1}^{\infty}2^{-\ell}\bigg(\sum_{j=1}^{\ell-1}s^{j-1}-s^{\ell-1}\bigg)\\
 & =2^{k}\sum_{\ell=k+1}^{\infty}2^{-\ell}\Big(\frac{1}{1-s}-\frac{s^{\ell-1}(2-s)}{1-s}\Big)\\
 & =\frac{1}{1-s}-\frac{2-s}{s(1-s)}\frac{s^{k+1}/2}{1-s/2}\\
 & =\sum_{j=1}^{k}s^{j-1}.
\end{align*}
Note also that 
\[
\sum_{j=1}^{k}s^{j-1}-s^{k}<\sum_{j=1}^{k}s^{j-1}<\sum_{j=1}^{k+1}s^{j-1}-s^{k+1}.
\]
It is then straightforward to verify that the law of $M_{k+1}-M_{k}$
is supported on $\{0,\pm s^{k}\}$ with $\p(M_{k+1}-M_{k}=s^{k})=\p(M_{k+1}-M_{k}=-s^{k})=2^{-(k+1)}$.
We then obtain that 
\begin{align*}
\E[(M_{k+1}-M_{k})^{2}]=2^{-k}s^{2k}.
\end{align*}
In other words, the rate $r=s^{2}/2$ has a limit $1/2$ as $s\to1$
from below; hence, the optimal rate is $1/2$. Again, a simple twist
yields similar examples, where the law of $U$ is absolutely continuous. 
\end{example}

\subsection{Non-uniform convergence rates}

\label{sec:nonuniform}

The above Theorem \ref{thm:martingale approx} requires that the support
of $U$ is bounded and the asymptotic constant depends only on $\sup\supp U-\inf\supp U$.
We show in the next result that, without the bounded range assumption,
non-uniform convergence rates can still be guaranteed. 
\begin{thm}
\label{thm:non-uniform} Let $U$ be an atomless random variable with
a finite $(2+\ee)$-th moment for some $\ee>0$, and $\{M_{k}\}_{k\geq0}$
be one of the four partition-based martingale approximations given
by Definition \ref{def:4martingales}. Then there exists a constant
$r\in(0,1)$ depending only on $\ee$ and another constant $C>0$ depending on the law of $U$,
such that 
\begin{align*}
\E[(U-M_{k})^{2}]\leq Cr^{k}.
\end{align*}
In particular, $M_{k}\to U$ both a.s.~and in $L^{2}$. 
\end{thm}

\begin{proof}
Let $\supp M_{1}=\{x,y\}$ where $x<y$. Take $M>\sup\{x,y\}$. Let
$\tau_{M}^{+}$ be the first hitting time to $\{x:x\ge M\}$, $\tau_{M}^{-}$
be the first hitting time to $\{x:x\le-M\}$, and $\tau_{M}=\min\{\tau_{M}^{+},\tau_{M}^{-}\}$.
It is straightforward to prove (see the proof of Lemma 5.6 of \citet{zhang2023exact}) that
for some $r\in(0,1)$ and $C>0$, 
\begin{align*}
\E[(U-M_{k})^{2}]\leq\E[\bone_{\{\tau_{M}<\infty\}}U^{2}]+CM^{2}r^{k}.
\end{align*}
Since the supports of $\{M_{k}\}_{k\geq0}$ correspond to the means
of $U$ conditioned on nested partitions of $\R$, we have $\tau_{M}^{+}<\infty$
implies $M_{k}\geq x$ for all $k\geq0$ and $M_{1}=y$. Therefore,
conditional on the event $\tau_{M}^{+}<\infty$, $\{M_{k}\}_{k\geq0}$ is a martingale
that is bounded from below by $x$. By Ville's inequality, 
\begin{align*}
\p(\tau_{M}^{+}<\infty) & \le\p(\tau_{M}^{+}<\infty\mid M_{1}>x)\\
 & =\p\left(\sup_{k\ge0}M_{k}\geq M\mid M_{1}>x\right)\le\frac{\E[U\mid M_{1}>x]-x}{M-x}=O(M^{-1}).
\end{align*}
Similarly, the same analysis holds for $\tau_{M}^{-}$. Hence, we
conclude that $\p(\tau_{M}<\infty)=O(M^{-1}).$ The rest follows line
by line as in \citet{zhang2023exact}, which we sketch for completeness.
By Hölder's inequality, $\E[\bone_{\{\tau_{M}<\infty\}}U^{2}]=O(M^{-\delta'})$
for some $\delta'>0$ depending only on $\ee$. With the choices $M=r^{-k/(2+\delta')}$ and
$q=r^{\delta'/(2+\delta')}\in(0,1)$, we have that for some $C>0$,
\begin{align*}
\E[(U-M_{k})^{2}]\leq C(M^{-\delta'}+M^{2}r^{k})\leq Cq^{k}.
\end{align*}
The final statement follows immediately from the martingale convergence
theorem. 
\end{proof}

If $U$ is bounded, Theorem \ref{thm:non-uniform} is a special case
of Theorem \ref{thm:martingale approx}. The case of the Simons martingale
has been established by \citet{zhang2023exact}, and our proof of
Theorem \ref{thm:non-uniform} follows a similar route. Since
the asymptotic constant is allowed to depend on the law of $U$, the
optimal non-uniform rate $r$ might be smaller than the uniform ones
(in Theorem \ref{thm:martingale approx}).

Recall from Appendix \ref{sec:uniform} that the rates obtained in
Theorem \ref{thm:martingale approx} are asymptotically optimal for
the Simons and median martingales (see Examples \ref{ex:simons1/2}
and \ref{ex:median}). For the variance and minimax martingales, we
show in the next result that the rate $r=1/4$ is optimal under certain
regularity conditions on the law of $U$. 
\begin{thm}
\label{prop:1/4+ee for variance} Let $\{M_{k}\}_{k\geq0}$ be either
the minimax or variance martingale converging to $U$. Suppose that
$U$ is bounded and has a bounded and continuous density $f$ with
$\inf f>0$ on $\supp U$, which we assume is a connected interval.
Then for any $r>1/4$, there exists a constant $C>0$ (depending on
$r$ and the law of $U$) such that 
\begin{align*}
\E[|U-M_{k}|^{2}]\leq Cr^{k}.
\end{align*}
\end{thm}

\begin{proof}
First, we claim that the split points for the nested partitions $\{\pi_{k}\}$
are dense. Indeed, for both martingales, we have shown that the variances
on each component in the partition $\pi_{k}$ go to zero uniformly
as $k\to\infty$. If the split points were not dense, some element
in the partition would cover an interval of positive length, on which
the variance of $U$ cannot vanish since we assumed $\inf f>0$ and
$\sup f<\infty$. Therefore, for any $\ee>0$, there exists $k$ such
that the $k$-th partition $\pi_{k}$ is such that on each interval
$A\in\pi_{k}$, $\sup_{A}f\leq(1+\ee)\inf_{A}f$.

Since one-step variance is always better than one-step minimax, it
remains to show that for an interval $A$ with $(1-\ee^2)\sup_{A}f\leq\inf_{A}f$,
there exists $C(\ee)\downarrow1/4$ as $\ee>0$ such that at each
step the total variance reduces by a factor of $C(\ee)$ for the minimax
martingale.

To see this, without loss of generality we start from $U$ supported
on $[0,1]$ (meaning that $\p(U\in[0,1])=1$) whose density takes
values in $[1-\ee,1+\ee]$, where $\ee\in(0,1)$. Then for any interval
$[a,b]\subseteq[0,1]$,
\begin{align*}
\E[U\mid U\in[a,b]]\in\Big[a+\frac{1-\ee}{2}(b-a),\,a+\frac{1+\ee}{2}(b-a)\Big].
\end{align*}
Therefore,
\begin{align*}
\p(U\in[a,b])\var(U\mid U\in[a,b])
&\leq \sup_{m\in[a+\frac{1-\ee}{2}(b-a),\,a+\frac{1+\ee}{2}(b-a)]}
\int_a^b (1+\ee)(x-m)^2\,\d x\\
&\leq \frac{(1+\ee)(1+3\ee^2)}{12}(b-a)^3.
\end{align*}
Similarly,
\begin{align*}
\p(U\in[a,b])\var(U\mid U\in[a,b])
&\geq \inf_{m\in[a+\frac{1-\ee}{2}(b-a),\,a+\frac{1+\ee}{2}(b-a)]}
\int_a^b (1-\ee)(x-m)^2\,\d x\\
&\geq \frac{1-\ee}{12}(b-a)^3.
\end{align*}
It follows that the split point $x_*\in[0,1]$ must satisfy
\begin{align*}
\frac{1-\ee}{12}x_*^3\leq \frac{(1+\ee)(1+3\ee^2)}{12}(1-x_*)^3
\quad\text{and}\quad
\frac{1-\ee}{12}(1-x_*)^3\leq \frac{(1+\ee)(1+3\ee^2)}{12}x_*^3.
\end{align*}
Therefore, for some $a_\ee\to0$, $|x_*-1/2|\leq a_\ee$.

A direct computation yields $\var(U)\geq \frac{1-\ee}{12}$. The total
remaining variance $\E[(U-M_1)^2]$ is bounded by
\begin{align*}
\E[(U-M_1)^2]
&\leq \frac{(1+\ee)(1+3\ee^2)}{12}\bigl(x_*^3+(1-x_*)^3\bigr)\\
&\leq \frac{(1+\ee)(1+3\ee^2)}{6}\Big(\frac12+a_\ee\Big)^3.
\end{align*}
Hence, as $\ee\to0$, the ratio becomes
\begin{align*}
\frac{\E[(U-M_1)^2]}{\var(U)}
\leq \frac{(1+\ee)(1+3\ee^2)}{4(1-\ee)}(1+2a_\ee)^3\to\frac14,
\end{align*}
as desired.
\end{proof}

In other words, under the assumptions in Theorem \ref{prop:1/4+ee for variance},
the MSE has an asymptotic convergence rate of ``$1/4+\ee$". Unfortunately,
we are unable to remove the assumptions for the density of $U$ in
Theorem \ref{prop:1/4+ee for variance} nor give counterexamples.
We conjecture that the same conclusion of Theorem \ref{prop:1/4+ee for variance}
holds without those assumptions.

\subsection{Proof of Theorem \ref{thm:martingale approx} }

\label{A.2}

\paragraph*{Variance}

\begin{lemma}\label{lemma:split middle} Let $U$ be an atomless
random variable supported on $I=[a,b]$. Then there exists 
\[
u_{*}\in\argmin_{u\in I}\big(\p(U\in[a,u))\var(U\mid U\in[a,u))+\p(U\in[u,b))\var(U\mid U\in[u,b))\big)
\]
such that 
\begin{align}
2u_{*}=\E[U\mid U\leq u_{*}]+\E[U\mid U>u_{*}]\label{eq:2u*}
\end{align}
and moreover, splitting at $u_{*}$ and at $u_{\mathrm{var}}$ defined
through \eqref{eq:var m def} have the same effect,\footnote{Recall that we break ties in \eqref{eq:var m def} by choosing the
largest minimizer.} i.e., $\p(U\leq u_{*})=\p(U\leq u_{\mathrm{var}})$. \end{lemma}
\begin{proof}
Let $Q$ be the quantile function of $U$. Since $U$ is atomless, if $u=Q(p)$, then
\begin{align*}
 & \p(U\in[a,u))\var(U\mid U\in[a,u))+\p(U\in[u,b))\var(U\mid U\in[u,b))\\
 & =p\var(U\mid U\in[a,Q(p)))+(1-p)\var(U\mid U\in[Q(p),b)).
\end{align*}
Let $u_{\mathrm{var}}$ be defined through \eqref{eq:var m def} and
$p_{\mathrm{var}}=\p(U\leq u_{\mathrm{var}})$. Then with the short-hand
notation $m_{a}(p)=\E[U\mid U\in[a,Q(p))]$ and $m_{b}(p)=\E[U\mid U\in[Q(p),b)]$,
we have 
\begin{align*}
p_{\mathrm{var}} & \in\argmin_{p\in[0,1]}\big(p\var(U\mid U\in[a,Q(p)))+(1-p)\var(U\mid U\in[Q(p),b))\big)\\
 & =\argmax_{p\in[0,1]}\big(p\,m_{a}(p)^{2}+(1-p)m_{b}(p)^{2}\big)=:\argmax_{p\in[0,1]}\Psi_{U}(p).
\end{align*}
Observe that for $p\in(0,1)$, 
\[
\frac{\d}{\d p}m_{a}(p)=\frac{Q(p)-m_{a}(p)}{p}\qquad\text{ and }\qquad\frac{\d}{\d p}m_{b}(p)=\frac{m_{b}(p)-Q(p)}{1-p}.
\]
We then have 
\begin{align*}
\frac{\d}{\d p}\Psi_{U}(p) & =m_{a}(p)^{2}+2m_{a}(p)(Q(p)-m_{a}(p))-m_{b}(p)^{2}+2m_{b}(p)(m_{b}(p)-Q(p))\\
 & =(m_{b}(p)-m_{a}(p))(m_{a}(p)+m_{b}(p)-2Q(p)).
\end{align*}
Since $U$ is atomless, it is non-degenerate, and hence $m_{b}(p_{\mathrm{var}})-m_{a}(p_{\mathrm{var}})>0$
and $p_{\mathrm{var}}\in(0,1)$. Note that the functions $m_{a},m_{b}$
are differentiable and hence continuous; $Q$ may have jumps but has
limits on both sides at every point. Since $p_{\mathrm{var}}\in\argmax_{p\in[0,1]}\Psi_{U}(p)$,
there are two cases: 
\begin{itemize}
\item $Q$ is continuous at $p_{\mathrm{var}}$. Then $\Psi_{U}'(p_{\mathrm{var}})=0$,
so $m_{a}(p_{\mathrm{var}})+m_{b}(p_{\mathrm{var}})=2Q(p_{\mathrm{var}})$.
In this case, we also have $Q(p_{\mathrm{var}})=u_{\mathrm{var}}$.
Therefore, taking $u_{*}=u_{\mathrm{var}}$ gives \eqref{eq:2u*}. 
\item $Q$ is discontinuous at $p_{\mathrm{var}}$. Then we must have $\Psi_{U}'(p_{\mathrm{var}}^{-})\geq0$
and $\Psi_{U}'(p_{\mathrm{var}}^{+})\leq0$, where for a generic function
$f$ defined on $\R$ we use the notation $f(x^{-})=\lim_{y\to x^{-}}f(y)$
and $f(x^{+})=\lim_{y\to x^{+}}f(y)$. Equivalently, 
\begin{align}
2Q(p_{\mathrm{var}}^{-})\leq m_{a}(p_{\mathrm{var}})+m_{b}(p_{\mathrm{var}})\leq2Q(p_{\mathrm{var}}^{+}).\label{eq:2eq}
\end{align}
Let $u_{1}=Q(p_{\mathrm{var}}^{-})$ and $u_{2}=Q(p_{\mathrm{var}}^{+})$.
Since $Q$ is discontinuous at $p_{\mathrm{var}}$, we have $\p(u_{1}\leq U\leq u_{2})=0$.
Therefore, the map $u\mapsto\p(U\leq u)$ is constant on $[u_{1},u_{2}]$
and for each $u\in[u_{1},u_{2}]$, $(\E[U\mid U\leq u],\E[U\mid U>u])=(m_{a}(p_{\mathrm{var}}),m_{b}(p_{\mathrm{var}}))$.
It follows from \eqref{eq:2eq} that \eqref{eq:2u*} holds for some
$u_{*}\in[u_{1},u_{2}]$. Using the definition \eqref{eq:var m def}
and $p_{\mathrm{var}}=\p(U\leq u_{\mathrm{var}})$, we have $u_{\mathrm{var}}=\max_{u\in[u_{1},u_{2}]}u=u_{2}$,
which implies that $\p(U\leq u_{*})=\p(U\leq u_{\mathrm{var}})$. 
\end{itemize}
Therefore, in both cases, the desired claim is verified. 
\end{proof}
\begin{rem}
If the support of $U$ is a connected interval, Lemma \ref{lemma:split middle}
is a special case of Theorem 1 of \citet{ishwaran2015effect} applied
with $f$ monotone and $X\lawis \mathrm{Unif}(0,1)$. 
\end{rem}

\begin{proof}[Proof for the variance martingale]
Recall the binary subtree representation \eqref{eq:binary rep}.
Since $U$ is atomless, the binary subtree is in fact always a binary
tree. At level $k$, $\E[(M_{k-1}-M_{k})^{2}]$ is a sum of $2^{k}$
terms, which we may rewrite as $\sum_{j}p_{j}d_{j}^{2}$ using a change
of variable. This corresponds to the $2^{k}$ edges (indexed by $j\in[2^{k}]$)
in the binary tree between levels $k-1$ and $k$, and $p_{j}$ represents
the probabilities of the edge $j$ and $d_{j}$ is the location difference
between the end vertices of the edge, or the \textit{length} of the
edge. Our goal is to bound from above 
\begin{align*}
\E[(M_{k-1}-M_{k})^{2}]=\sum_{j=1}^{2^{k}}p_{j}d_{j}^{2}.
\end{align*}
The key is first to bound the quantity 
\begin{align*}
\max_{1\leq j\leq2^{k-1}}(p_{2j-1}d_{2j-1}^{2}+p_{2j}d_{2j}^{2}).
\end{align*}
That is, the variance increases by splitting the $j$-th mass at level
$k-1$ (which is the common ancestor of the $(2j-1)$-th and $(2j)$-th
masses). By the martingale property, 
\begin{align*}
p_{2j-1}d_{2j-1}=p_{2j}d_{2j}.
\end{align*}
It follows that 
\begin{align}
p_{2j-1}d_{2j-1}^{2}+p_{2j}d_{2j}^{2}=(p_{2j-1}+p_{2j})d_{2j-1}d_{2j}.\label{eq:pdd}
\end{align}
The general strategy to analyze the quantity \eqref{eq:pdd} is to
start from an arbitrary mass/node $\bM$ at level $k-1$, trace back
its ancestors, bound from above the probability of the mass (which
is $p_{2j-1}+p_{2j}$ in the above expression), and establish upper bounds
for the length of the edges (which is $d_{2j-1},\,d_{2j}$).

The path leading to a mass $\bM$ at level $k-1$ in the binary tree
representation \eqref{eq:binary rep} consists of $k$ nodes: $\emptyset=\bN_{0},\bN_{1},\dots,\bN_{k-1}=\bM$,
where $\bN_{j}\in V_{j}$. We denote the lengths of the edges connecting
the node $\bN_{j-1}$ to its two descendants on level $j$ by $a_{j},b_{j},\,j\in[k-1]$,
where $b_{j}=|\ell_{\bN_{j-1}}-\ell_{\bN_{j}}|$. It follows by the
martingale property that the mass $\bM$ has probability 
\begin{align}
\prod_{j=1}^{k-1}q_{j}:=\prod_{j=1}^{k-1}\frac{a_{j}}{a_{j}+b_{j}}.\label{eq:prob prod}
\end{align}
The lengths of the edges of $\bM$ are the mean distances of its children,
which are bounded by the distances from $\sum_{j=1}^{k-1}b_{j}$ (the
location of $\bM$) to the endpoints of the interval that $\bM$ carries
in the splitting process. Therefore, it is natural to update these
two distances as we trace the path from the root of the tree down
to the node $\bM$.

Let $a_{0}=b_{0}=1\geq\sup U-\inf U$. We consider the following algorithm,
which starts from the tuple $(a_{0},b_{0},a_{1},b_{1},q_{1})$ and
updates it $k-1$ times for each $j\in[k-1]$ when we trace the path
down to $\bM$. At step $j$, the tuple represents the two pairs $(a_{k},b_{k}),\,(a_{j},b_{j})$
that govern the maximum length of the edges of the $j$-th node (that
is, the distance from its location to the boundaries of the intervals
this node represents) for a certain $j\in[k-1]$, as well as the probability
weight $q_{j}$ multiplied at this step. In a generic setting, suppose
that we have $(a_{k},b_{k},a_{j},b_{j},q_{j})$ updated at the node
$\bN_{j}\in V_{j}$. This means that the edges forming $a_{k+1},\dots,a_{j}$
are all on one side of the path since only the last one of these contributes
to the bound. Now when updating the node at level $j+1$ there are
two cases: 
\begin{enumerate}
\item[(i)] $a_{j+1}$ and $a_{j}$ are on the same side. We update $(a_{k},b_{k},a_{j},b_{j},q_{j})$
with $(a_{k},b_{k},a_{j+1},b_{j+1},q_{j+1})$, where $a_{j+1}\leq(a_{j}+b_{j})/2$
and $b_{j+1}\leq(a_{k}+b_{k})/2$ (they are the two edges) and $q_{j+1}=a_{j+1}/(a_{j+1}+b_{j+1})$; 
\item[(ii)] $a_{j+1}$ and $a_{j}$ are on distinct sides. We update $(a_{k},b_{k},a_{j},b_{j},q_{j})$
with $(a_{j},b_{j},a_{j+1},b_{j+1},q_{j+1})$, where $b_{j+1}\leq(a_{j}+b_{j})/2$,
$a_{j+1}\leq(a_{k}+b_{k})/2$, and $q_{j+1}=a_{j+1}/(a_{j+1}+b_{j+1})$. 
\end{enumerate}
Here, we have used \eqref{eq:prob prod} and Lemma \ref{lemma:split middle}.

Let $(a,b,a',b',q)$ be generic entries of the tuple and $(\hat{a},\hat{b},\hat{a}',\hat{b}',\hat{q})$
be the updated tuple, where we recall $\hat{q}=\hat{a}'/(\hat{a}'+\hat{b}')$.
We claim that 
\begin{align}
\frac{(\hat{a}+\hat{b})(\hat{a}'+\hat{b}')}{(a+b)(a'+b')}\leq\frac{1}{2\hat{q}}.\label{eq:variance claim}
\end{align}
To see this, we separately consider the two cases above. 
\begin{itemize}
\item In case (i), $\hat{a}=a$ and $\hat{b}=b$, and 
\begin{align*}
\hat{a}'+\hat{b}'=a_{j+1}+b_{j+1}=\frac{a_{j+1}}{\hat{q}}\leq\frac{a_{j}+b_{j}}{2\hat{q}}=\frac{a'+b'}{2\hat{q}}.
\end{align*}
\item In case (ii), $\hat{a}=a'$ and $\hat{b}=b'$, and 
\begin{align*}
\hat{a}'+\hat{b}'=a_{j+1}+b_{j+1}=\frac{a_{j+1}}{\hat{q}}\leq\frac{a_{k}+b_{k}}{2\hat{q}}=\frac{a+b}{2\hat{q}}.
\end{align*}
\end{itemize}
This proves \eqref{eq:variance claim}.


For each interval $I\in\pi_k$, write
\[
p_I:=\p(U\in I),\qquad m_I:=\E[U\mid U\in I],\qquad
L_I:=m_I-\inf I,\qquad R_I:=\sup I-m_I.
\]
The same recursive argument leading to \eqref{eq:variance claim} yields
\begin{align}
\max_{I\in\pi_k} p_IL_I R_I \le 2^{-2-k}.
\label{eq:max-cell-risk}
\end{align}
Indeed, along each update, the quantity $p_I L_I R_I$ is multiplied by at most $1/2$,
while initially $L_I+R_I\le \sup U-\inf U\le 1$, so $L_I R_I\le 1/4$.

Since $U\mid(U\in I)$ is supported on the interval $I$, we have
\[
\var(U\mid U\in I)\le (m_I-\inf I)(\sup I-m_I)=L_I R_I.
\]
Therefore,
\begin{align*}
\E[(U-M_k)^2]
&=\sum_{I\in\pi_k} p_I\var(U\mid U\in I)
 \le \sum_{I\in\pi_k} p_IL_I R_I.
\end{align*}
Set $s_I:=\sqrt{L_I R_I}$. By Hölder's inequality and since $\sum_{I\in\pi_k} p_I=1$,
\begin{align*}
\sum_{I\in\pi_k} p_IL_I R_I
=\sum_{I\in\pi_k} p_I s_I^2
&\le
\Big(\sum_{I\in\pi_k} p_I s_I^3\Big)^{2/3}
\Big(\sum_{I\in\pi_k} p_I\Big)^{1/3}\\
&\le
\Big(\max_{I\in\pi_k} p_I s_I^2 \sum_{I\in\pi_k} s_I\Big)^{2/3}.
\end{align*}
Moreover,
\[
2s_I\le L_I+R_I=\sup I-\inf I,
\]
so, since the intervals in $\pi_k$ are disjoint and contained in $[\inf U,\sup U]$,
\[
\sum_{I\in\pi_k} s_I
\le \frac12 \sum_{I\in\pi_k} (\sup I-\inf I)
\le \frac12(\sup U-\inf U)\le \frac12.
\]
Combining this with \eqref{eq:max-cell-risk}, we obtain
\[
\E[(U-M_k)^2]
\le \Big(2^{-2-k}\cdot \frac12\Big)^{2/3}
\le 2^{-2-2k/3},
\] and the proof is complete.
\end{proof}

\paragraph*{Simons }
\begin{proof}[Proof for the Simons martingale]

It is noted in \citet{zhang2023exact} that the Simons martingale
satisfies the separated tree condition (i.e., the branches of the
tree representation from all different levels, when projected onto
the real line as intervals, are either disjoint or have a containment
relationship), which is a consequence of the construction. We follow
a similar argument to the variance martingale. Consider a node $\bM$
at level $k-1$ and a path $(\emptyset=\bN_{0},\dots,\bN_{k-1}=\bM)$
leading to $\bM$, where $\emptyset$ denotes the root of a binary
tree. We provide an algorithm that recursively defines the tuples
$(A_{j},B_{j},q_{j}),~0\leq j\leq k-1$ for each node $\bN_{j},~0\leq j\leq k-1$
on that path, from the root to $\bM$. The quantities $A_{j}$ and $B_{j}$
represent the maximum possible lengths of the two edges emanating
from a node $\bN_{j}$, and $q$ represents the one-step splitting
probability leading to $\bN_{j}$, based on the constraints from the
separated tree property.

We start from $(A_{0},B_{0},q_{0})=(1,1,1)$. Again, denote the lengths
of the edges connecting the node $\bN_{j-1}$ to its two descendants
at level $j$ by $a_{j},b_{j},\,j\in[k-1]$, where $b_{j}=|\ell_{\bN_{j-1}}-\ell_{\bN_{j}}|$.
Consider the node $\bN_{j}$ at level $j$. Then $(A_{j},B_{j},q_{j})$
must be of the form (while not distinguishing the order of $A_{j}$
and $B_{j}$) 
\begin{align*}
(A_{j},B_{j},q_{j})=\Big(b_{i}-\sum_{\ell=i+1}^{j}b_{\ell},\,b_{j},\,\frac{a_{j}}{a_{j}+b_{j}}\Big)
\end{align*}
for some $i\leq j-1$. This happens when the edges $a_{i+1},\dots,a_{j}$
lie on the same side of the path $(\bN_{1},\dots,\bN_{k})$. For the
descendant $\bN_{j+1}$ of $\bN_{j}$, there are two possibilities: 
\begin{enumerate}
\item[(i)] if $a_{j+1}$ is on the same side of the path as $a_{j}$, we define
\begin{align*}
(A_{j+1},B_{j+1},q_{j+1})=\Big(b_{i}-\sum_{\ell=i+1}^{j+1}b_{\ell},\,b_{j+1},\,\frac{a_{j+1}}{a_{j+1}+b_{j+1}}\Big),
\end{align*}
where $a_{j+1}\leq b_{j}$; 
\item[(ii)] otherwise, we define 
\begin{align*}
(A_{j+1},B_{j+1},q_{j+1})=\Big(b_{j}-b_{j+1},\,b_{j+1},\,\frac{a_{j+1}}{a_{j+1}+b_{j+1}}\Big),
\end{align*}
where $a_{j+1}\leq b_{i}-\sum_{\ell=i+1}^{j}b_{\ell}$. 
\end{enumerate}
The choice of $(A_{j+1},B_{j+1},q_{j+1})$ is justified by the separated
tree property.

Let $q_{j+1}=a_{j+1}/(a_{j+1}+b_{j+1})$. We next express the ratio
$A_{j+1}B_{j+1}/(A_{j}B_{j})$ after the update in step $j+1$ using
$q_{j+1}$. We have in case (i) that 
\begin{align*}
\frac{\Big(b_{i}-\sum_{\ell=i+1}^{j+1}b_{\ell}\Big)b_{j+1}}{\Big(b_{i}-\sum_{\ell=i+1}^{j}b_{\ell}\Big)b_{j}}\leq\frac{b_{j+1}}{b_{j}}\leq\frac{b_{j+1}}{a_{j+1}}=\frac{1}{q_{j+1}}-1,
\end{align*}
and in case (ii), 
\begin{align*}
\frac{(b_{j}-b_{j+1})b_{j+1}}{\Big(b_{i}-\sum_{\ell=i+1}^{j}b_{\ell}\Big)b_{j}}\leq\frac{b_{j+1}}{b_{i}-\sum_{\ell=i+1}^{j}b_{\ell}}\leq\frac{b_{j+1}}{a_{j+1}}=\frac{1}{q_{j+1}}-1.
\end{align*}
In both cases, $A_{j+1}B_{j+1}/(A_{j}B_{j})\leq1/q_{j+1}-1$. In other
words, the contribution of the two edges of $\bM$ is upper bounded
by 
\begin{align*}
A_{k-1}B_{k-1}\prod_{\ell=0}^{k-1}q_{\ell}=\prod_{\ell=1}^{k-1}\frac{A_{\ell}B_{\ell}}{A_{\ell-1}B_{\ell-1}}\prod_{\ell=1}^{k-1}q_{\ell}\leq\prod_{\ell=1}^{k-1}\Big(\frac{1}{q_{\ell}}-1\Big)q_{\ell}=\prod_{\ell=1}^{k-1}(1-q_{\ell}),
\end{align*}
where we recall that $\{q_{\ell}\}_{1\leq\ell\leq k-1}$ are the one-step
probabilities leading to $\bM$. On the other hand, we also know that
the probability that an edge of $\bM$ carries must be bounded by
$\prod_{\ell=1}^{k-1}q_{\ell}$. Therefore, we conclude that, with
the decomposition 
\begin{align*}
\E[(M_{k-1}-M_{k})^{2}]=\sum_{j=1}^{2^{k}}p_{j}d_{j}^{2},
\end{align*}
it holds that for each $j\in[2^{k}]$, there exist $\{q_{\ell}\}_{1\leq\ell\leq k-1}$
such that 
\begin{align*}
p_{j}d_{j}^{2}\leq\prod_{\ell=1}^{k-1}(1-q_{\ell})\quad\text{and}\quad p_{j}\leq\prod_{\ell=1}^{k-1}q_{\ell}.
\end{align*}
Multiplying the two inequalities leads to 
\begin{align*}
p_{j}d_{j}\leq\sqrt{\prod_{\ell=1}^{k-1}q_{\ell}(1-q_{\ell})}\leq2^{1-k}.
\end{align*}
Since this holds uniformly in $j$, we conclude that 
\begin{align*}
\E[(M_{k-1}-M_{k})^{2}]=\sum_{j=1}^{2^{k}}p_{j}d_{j}^{2}\leq\max_{j\in[2^{k}]}(p_{j}d_{j})\sum_{j=1}^{2^{k}}d_{j}\leq2^{1-k}.
\end{align*}
By the martingale property and the main result of \citet{simons1970martingale}, 
\begin{align*}
\E[(M_{k}-U)^{2}]=\sum_{j=k}^{\infty}\E[(M_{j}-M_{j+1})^{2}]\leq\sum_{j=k}^{\infty}2^{-j}\leq2^{1-k}.
\end{align*}
This completes the proof. 
\end{proof}

\paragraph*{Minimax }
\begin{proof}[Proof for the minimax martingale]
At level $k$, there are $2^{k}$ vertices (denoted by $j\in[2^{k}]$)
in the binary tree representation of the partition-based martingale
approximation. Denote the probabilities of the vertices by $p_{j}$
(which correspond to $\p(U\in A_{j}),~A_{j}\in\pi_{k}$) and the locations
by $\ell_{j}$. It follows that 
\begin{align*}
\E[(U-M_{k})^{2}]=\sum_{j=1}^{2^{k}}p_{j}\E[(U-\ell_{j})^{2}\mid U\in A_{j}]=:\sum_{j=1}^{2^{k}}p_{j}d_{j}^{2}.
\end{align*}
Note that $\sum_{j}p_{j}=1$ and $\sum_{j}d_{j}\leq\sup U-\inf U\leq1$.
It follows analogously to the Hölder's inequality argument in \eqref{eq:holder2-1}
that 
\begin{align*}
\E[(U-M_{k})^{2}] & \leq\max_{j\in[2^{k}]}(p_{j}d_{j}^{2})^{2/3}.
\end{align*}
Let $\phi(u)=\p(U<u)\var(U\mid U<u)$ and $\psi(u)=\p(U\geq u)\var(U\mid U\geq u)$.
Note that $\phi$ is increasing and $\psi$ is decreasing in $u$,
and both functions are continuous if we assume that $U$ is atomless.
Recall by the minimax property that at each step we pick $u\in\R$
such that $\max\{\phi(u),\psi(u)\}$ is minimized. This means that
$\phi(u)=\psi(u)$. Since $\var(U)\geq\phi(u)+\psi(u)$, we have both
$\phi(u)\leq\var(U)/2$ and $\psi(u)\leq\var(U)/2$. Inductively,
we have for any $k$, $\max(p_{j}d_{j}^{2})\leq2^{-k}\var(U)\leq2^{-2-k}$.
Combining the above, we arrive at 
\begin{align*}
\E[(U-M_{k})^{2}]\leq\max_{j\in[2^{k}]}(p_{j}d_{j}^{2})^{2/3}\leq2^{-4/3}2^{-2k/3}\leq0.4\cdot2^{-2k/3},
\end{align*}
as desired. 
\end{proof}

\paragraph*{Median }
\begin{proof}[Proof for the median martingale]
In the same setting as in the minimax case, we write 
\begin{align*}
\E[(U-M_{k})^{2}]=\sum_{j=1}^{2^{k}}p_{j}d_{j}^{2},
\end{align*}
where $\sum_{j}p_{j}=1$ and $\sum_{j}d_{j}\leq1$. We further know
that $\max_{j}p_{j}=2^{-k}$ by the median splitting construction.
Therefore, 
\begin{align*}
\E[(U-M_{k})^{2}]=\sum_{j=1}^{2^{k}}p_{j}d_{j}^{2}\leq2^{-k}\sum_{j=1}^{2^{k}}d_{j}^{2}\leq2^{-k},
\end{align*}
as desired. 
\end{proof}

\section{Proofs of main results}

\subsection{Proofs of results in Section \ref{sec:minimax}}

\label{sec:minimax proof}\label{sec:proofs}
\begin{proof}[Proof of Lemma \ref{lemma:continuous}]
We take arbitrary $a<b$ and define events $E=E_{1}\cup E_{2}$, where $E_{1}=\{\bX\in A,\,X_{j}<a\}$
and $E_{2}=E\setminus E_{1}=\{\bX\in A,\,a\leq X_{j}<b\}.$ By the total
variance formula, 
\begin{align}
\begin{split}\var(Y\mid E) & =\E[\var(Y\mid\sigma(E_{1},E_{2}))\mid E]+\var(\E[Y\mid\sigma(E_{1},E_{2})]\mid E)\\
 & \geq\E[\var(Y\mid\sigma(E_{1},E_{2}))\mid E]\geq\p(E_{1}\mid E)\var(Y\mid E_{1}).
\end{split}
\label{eq:totalvarianceformula}
\end{align}
Observe that $\var(Y\mid E)=\var(Y\mid\bX\in A,X_{j}<b)$ and $\var(Y\mid E_{1})=\var(Y\mid\bX\in A,X_{j}<a)$.
Inserting these two expressions back into \eqref{eq:totalvarianceformula},
we obtain 
\[
\var(Y\mid\bX\in A,X_{j}<b)\geq\frac{\p(\bX\in A,X_{j}<a)}{\p(\bX\in A,X_{j}<b)}\var(Y\mid\bX\in A,X_{j}<a).
\]
Rearranging gives 
\[
\p(\bX\in A,\,X_{j}<a)\var(Y\mid\bX\in A,\,X_{j}<a)\leq\p(\bX\in A,\,X_{j}<b)\var(Y\mid\bX\in A,\,X_{j}<b)
\]
and therefore $\varphi_L$ is non-decreasing; and
similarly $\varphi_R$ is non-increasing. The continuity of $\varphi_L$ and $\varphi_R$ follows from the marginally atomless condition.
\end{proof}

We also need the following risk decomposition lemma.
\begin{lemma}\label{lemma:var decomposition} Let $\delta>0$. Then
for any joint distribution $(X,Y)$ with $X\geq\E[Y]$ or $X\leq \E[Y]$, it holds that
\begin{align}
\var(Y)\leq(1+\delta)\,\E[(X-Y)^{2}]+\frac{(1+\delta)^{3}}{\delta^{3}}(\sup\supp X-\inf\supp X)^{2}.\label{eq:var bound}
\end{align}
\end{lemma} 
\begin{proof}
By symmetry, we assume that $X\geq\E[Y]$. 
Let $C=(1+\delta)^{3}/\delta^{3}$. Without loss of generality, we
assume that $\E[Y]=0$. Equivalent to \eqref{eq:var bound} is 
\begin{align*}
(1+\delta)\,\E[X^{2}]-2(1+\delta)\,\E[XY]+\delta\E[Y^{2}]+C(\sup\supp X-\inf\supp X)^{2}\geq0.
\end{align*}
After completing the square, it remains to show 
\begin{align}
\E\Big[(\sqrt{\delta}\,Y-\frac{1+\delta}{\sqrt{\delta}}X)^{2}\Big]+\Big(C(\sup\supp X-\inf\supp X)^{2}-\frac{1+\delta}{\delta}\E[X^{2}]\Big)\geq0.\label{eq:var2}
\end{align}
If $\inf\supp X=0$ or $\sup\supp X/\inf\supp X\geq(1-\sqrt{(1+\delta)/(C\delta)})^{-1}$,
the second term is always non-negative and hence \eqref{eq:var2}
holds. Otherwise, since $C=(1+\delta)^{3}/\delta^{3}$, 
\begin{align*}
(1+\delta)(\inf\supp X)^{2}\geq(1+\delta)(1-\sqrt{(1+\delta)/(C\delta)})(\sup\supp X)^{2}=(\sup\supp X)^{2}.
\end{align*}
It follows that 
\begin{align*}
\E\Big[(\sqrt{\delta}\,Y-\frac{1+\delta}{\sqrt{\delta}}X)^{2}\Big] & \geq\E\Big[\sqrt{\delta}\,Y-\frac{1+\delta}{\sqrt{\delta}}X\Big]^{2}\\
 & =\frac{(1+\delta)^{2}}{\delta}\E[X]^{2}\geq\frac{(1+\delta)^{2}}{\delta}(\inf\supp X)^{2}\geq\frac{(1+\delta)}{\delta}(\sup\supp X)^{2}\geq\frac{1+\delta}{\delta}\E[X^{2}].
\end{align*}
This shows \eqref{eq:var2} and thus proves \eqref{eq:var bound}. 
\end{proof}
\vspace{0.3cm}

\begin{lemma}[Oscillation growth under cyclic axis-aligned refinement]
\label{lem:cyclic_oscillation_bound}
Let $\{\pi_k\}_{k\ge 0}$ be the sequence of partitions generated by the cyclic splitting schedule, and let
$k=dq+r$ with $q\in \mathbb{N}_0$ and $0\le r<d$. Suppose that
\[
g(\bx)=\sum_{i=1}^d g_i(x_i)
\]
is an additive representation with each $g_i$ of bounded variation. For a cell
$A=\prod_{i=1}^d I_i(A)\in \pi_k$, define
\[
\osc_A(g):=\sup_{\bx\in A} g(\bx)-\inf_{\bx\in A} g(\bx).
\]
Then
\[
\sum_{A\in \pi_k}\osc_A(g)
\le
2^{\,k-\lfloor k/d\rfloor}\sum_{i=1}^d \|g_i\|_{\mathrm{TV}}.
\]
Consequently, after taking the infimum over additive representations of $g$,
\[
\sum_{A\in \pi_k}\osc_A(g)\le 2^{\,k-\lfloor k/d\rfloor}\|g\|_{\mathrm{TV}}.
\]
\end{lemma}

\begin{proof}
Define
$$\osc_{I}(g_i):=\sup_{u\in I} g_i(u)-\inf_{u\in I} g_i(u).$$
For every cell $A=\prod_{i=1}^d I_i(A)$, additivity gives
\[
\osc_A(g)\le \sum_{i=1}^d \osc_{I_i(A)}(g_i).
\]
Hence it suffices to control, for each coordinate $i\in[d]$,
\[
S_i(k):=\sum_{A\in\pi_k}\osc_{I_i(A)}(g_i).
\]

At level $k$, the cyclic rule splits along coordinate
\[
j_k:=1+(k \bmod d).
\]
If a cell $A$ is not split, then its contribution to every $S_i$ is unchanged. If a cell $A$ is split into
$A^L$ and $A^R$ along coordinate $j_k$, then for every $i\neq j_k$,
\[
I_i(A^L)=I_i(A^R)=I_i(A),
\]
so the contribution to $S_i$ is duplicated:
\[
\osc_{I_i(A^L)}(g_i)+\osc_{I_i(A^R)}(g_i)
=
2\,\osc_{I_i(A)}(g_i).
\]
For the split coordinate $j_k$, the child intervals partition the parent interval, hence
\[
\osc_{I_{j_k}(A^L)}(g_{j_k})+\osc_{I_{j_k}(A^R)}(g_{j_k})
\le
\osc_{I_{j_k}(A)}(g_{j_k}).
\]
Therefore
\[
S_{j_k}(k+1)\le S_{j_k}(k),
\qquad
S_i(k+1)\le 2S_i(k)\quad\text{for } i\neq j_k.
\]

Fix $i\in[d]$. By time $k=dq+r$, coordinate $i$ has been selected exactly
$q+\mathbf{1}_{\{i\le r\}}$ times, and has not been selected exactly
\[
k-q-\mathbf{1}_{\{i\le r\}}
\]
times. Since $S_i$ never increases when coordinate $i$ is selected and can at most double otherwise, we obtain
\[
S_i(k)\le 2^{k-q-\mathbf{1}_{\{i\le r\}}}S_i(0)\le 2^{k-q}\|g_i\|_{\mathrm{TV}},
\]
because $\osc_{I_i(A_0)}(g_i)\le \|g_i\|_{\mathrm{TV}}$ at level $0$.

Summing over $i$ yields
\[
\sum_{A\in\pi_k}\osc_A(g)
\le
\sum_{i=1}^d S_i(k)
\le
2^{\,k-\lfloor k/d\rfloor}\sum_{i=1}^d \|g_i\|_{\mathrm{TV}}.
\]
Taking the infimum over all additive representations of $g$ gives the final claim.
\end{proof}

\begin{proof}[Proof of Theorem \ref{thm:minimax exp decay 2}]

We divide the proof into three steps.

\paragraph{Step I: reducing to an inequality involving \texorpdfstring{$\var(Y)$}{}.}

We claim that it suffices to prove 
\begin{align}
\E[(Y-M_{k})^{2}]\leq\inf_{g\in\G}\bigg((1+\delta)\,\E[(Y-g(\bX))^{2}]+(1+\delta^{-1})2^{-2\lfloor k/d\rfloor/3}(\n{g}_{\mathrm{TV}}\var(Y))^{2/3}\bigg).\label{eq:toprove}
\end{align}
To this end, we need an upper bound for $\var(Y)$. Fix $g\in\G$ and
let $g^*:=(\sup g+\inf g)/2$. By Minkowski's inequality, 
\begin{align}
\begin{split}\var(Y)^{1/2}=\min_{y\in\R}\E[(Y-y)^{2}]^{1/2}\leq\E[(Y-g^*)^{2}]^{1/2} & \leq\E[(Y-g(\bX))^{2}]^{1/2}+\E[(g(\bX)-g^*)^{2}]^{1/2}\\
 & \leq\E[(Y-g(\bX))^{2}]^{1/2}+\frac{1}{2}\n{g}_{\mathrm{TV}}.
\end{split}
\label{eq:varub}
\end{align}
Squaring both sides yields that for any $\delta>0$, 
\begin{align*}
\var(Y)\leq(1+\delta)\E[(Y-g(\bX))^{2}]+\frac{1+\delta^{-1}}{4}\n{g}_{\mathrm{TV}}^{2},
\end{align*}
where we have used $(a+b)^{2}\leq(1+\delta)a^{2}+(1+\delta^{-1})b^{2}$.
It follows that 
\begin{align}
\begin{split}\n{g}_{\mathrm{TV}}^{2/3}\var(Y)^{2/3} & \leq\n{g}_{\mathrm{TV}}^{2/3}\bigg((1+\delta)^{2/3}\E[(Y-g(\bX))^{2}]^{2/3}+\Big(\frac{1+\delta^{-1}}{4}\Big)^{2/3}\n{g}_{\mathrm{TV}}^{4/3}\bigg)\\
 & \leq\frac{2(1+\delta)}{3}\E[(Y-g(\bX))^{2}]+\bigg(\frac{1}{3}+\Big(\frac{1+\delta^{-1}}{4}\Big)^{2/3}\bigg)\n{g}_{\mathrm{TV}}^{2},
\end{split}
\label{eq:log ineq}
\end{align}
where we have used the inequality $x^{1/3}y^{2/3}\leq x/3+2y/3$,
which follows from the concavity of logarithm and Jensen's inequality.
Inserting \eqref{eq:log ineq} into \eqref{eq:toprove} yields \eqref{eq:d-dim rate}.

\paragraph{Step II: decomposition of the risk.}

It remains to establish \eqref{eq:toprove}. Denote by $\pi_{k}=\{I_{j}\}_{j\in J_{k}}$
the partition of $\R^{d}$ formed at level $k$ by the cyclic minimax
construction. Note that by construction, $\#J_{k}\leq2^{k}$ and $M_{k}\mid\bX\in I_{j}$
is a piecewise constant random variable, which takes a constant value
$y_{k,j}$ inside the box $I_{j}$, so $M_{k}\bone_{\{\bX\in I_{j}\}}=y_{k,j}\bone_{\{\bX\in I_{j}\}}$.
By definition, $y_{k,j}=\E[Y\mid\bX\in I_{j}]$. Define $g_{j}=g|_{I_{j}}$ and $\Ran(g_{j})=[\inf g_{j},\sup g_{j}]$, and recall that $\Delta g_{j}=\sup g_{j}-\inf g_{j}\leq\n{g_{j}}_{\mathrm{TV}}$.

The key step is to decompose the local mean squared error $\E[(Y-M_{k})^{2}\bone_{\{\bX\in I_{j}\}}]$
into two parts for each $j$, with the help of Lemma \ref{lemma:var decomposition}.
Next, we claim that 
\begin{align}
\E[(Y-M_{k})^{2}\bone_{\{\bX\in I_{j}\}}]=\E[(Y-y_{k,j})^{2}\bone_{\{\bX\in I_{j}\}}]\leq U_{j}+V_{j},\label{eq:uv}
\end{align}
where $U_{j},V_{j}$ are defined in the following. We divide into
two separate cases depending on the location of $y_{k,j}$: 
\begin{enumerate}
\item[(i)] $y_{k,j}\in\Ran(g_{j})$. Define $U_{j}=(1+\delta)\,\E[(Y-g(\bX))^{2}\bone_{\{\bX\in I_{j}\}}]$
and $$V_{j}=\min\big\{(1+\delta^{-1})\,\E[(g(\bX)-y_{k,j})^{2}\bone_{\{\bX\in I_{j}\}}],\,\E[(Y-M_{k})^{2}\bone_{\{\bX\in I_{j}\}}]\big\}.$$ 
\item[(ii)] $y_{k,j}\not\in\Ran(g_{j})$. Define $U_{j}=(1+\delta)\,\E[(Y-g(\bX))^{2}\bone_{\{\bX\in I_{j}\}}]$
and $$V_{j}=\min\Big\{\frac{(1+\delta)^{3}}{\delta^{3}}\,p_{j}(\Delta g_{j})^{2},\,\E[(Y-M_{k})^{2}\bone_{\{\bX\in I_{j}\}}]\Big\},$$
where $p_j:=\p(\bX\in I_j)$.
\end{enumerate}
In case (i), \eqref{eq:uv} follows immediately from the polarization
identity $(a+b)^{2}\leq(1+\delta)a^{2}+(1+\delta^{-1})b^{2}$. On
the other hand, for case (ii), \eqref{eq:uv} follows by applying
Lemma \ref{lemma:var decomposition} to the joint distribution $(g(\bX),Y)\mid\bX\in I_{j}$
and using that $y_{k,j}=\E[Y\mid\bX\in I_{j}]$. This completes the
proof of \eqref{eq:uv}.

\paragraph{Step III: controlling \texorpdfstring{$\sum_{j}U_{j}$}{} and \texorpdfstring{$\sum_{j}V_{j}$}{}.}

We have by construction and \eqref{eq:uv} that 
\begin{align}
\E[(Y-M_{k})^{2}]=\sum_{j\in J_{k}}\E[(Y-M_{k})^{2}\bone_{\{\bX\in I_{j}\}}]\leq\sum_{j\in J_{k}}U_{j}+\sum_{j\in J_{k}}V_{j}.\label{eq:construction}
\end{align}
The first term is easy to control by definition: 
\begin{align}
\sum_{j\in J_{k}}U_{j}=\sum_{j\in J_{k}}(1+\delta)\,\E[(Y-g(\bX))^{2}\bone_{\{\bX\in I_{j}\}}]=(1+\delta)\,\E[(Y-g(\bX))^{2}].\label{eq:uj}
\end{align}
To control $\sum_{j}V_{j}$, we apply Hölder's inequality. We have 
\begin{align}
\begin{split}\sum_{j\in J_{k}}V_{j} & \leq\Big(\sum_{j\in J_{k}}p_{j}\big(\frac{V_{j}}{p_{j}}\big)^{3/2}\Big)^{2/3}\Big(\sum_{j\in J_{k}}p_{j}\Big)^{1/3}\\
 & \leq\Bigg((\max_{j\in J_{k}}V_{j})\sum_{j\in J_{k}}\sqrt{\frac{V_{j}}{p_{j}}}\Bigg)^{2/3}\leq\Bigg((\max_{j\in J_{k}}\E[(Y-M_{k})^{2}\bone_{\{\bX\in I_{j}\}}])\sum_{j\in J_{k}}\sqrt{\frac{V_{j}}{p_{j}}}\Bigg)^{2/3}.
\end{split}
\label{eq:holder2}
\end{align}
To further bound the right-hand side of \eqref{eq:holder2}, we first recall
from \eqref{eq:2-k} that as a consequence of the marginally atomless
property, 
\begin{align}
\max_{j\in J_{k}}\E[(Y-M_{k})^{2}\bone_{\{\bX\in I_{j}\}}]\leq2^{-k}\var(Y).\label{eq:maxvar1}
\end{align}
Next, we claim that 
\begin{align}
\sqrt{\frac{V_{j}}{p_{j}}}\leq(1+\delta^{-1})^{3/2}\Delta g_{j}.\label{eq:v/p}
\end{align}
This is immediate for case (ii) above (when $y_{k,j}\not\in\Ran(g_{j})$).
By definition, we have for case
(i) above (when $y_{k,j}\in\Ran(g_{j})$), 
\[
V_{j}\leq(1+\delta^{-1})p_{j}(\Delta g_{j})^{2}<p_{j}(1+\delta^{-1})^{3}(\Delta g_{j})^{2},
\]
as desired. On the other hand, suppose that $k=dq+r$ for some integer
$q$ and $0\leq r<d$, and denote by $\pi_{k}$ the resulting partition
of $\R^{d}$. 
We use Lemma \ref{lem:cyclic_oscillation_bound} and observe that 
\begin{align*}
\sum_{A\in\pi_{k}}\Big(\sup_{\bx\in A}g(\bx)-\inf_{\bx\in A}g(\bx)\Big)\leq2^{q(d-1)+r}\n{g}_{\mathrm{TV}}.
\end{align*}
In other words, we have that for $k\in\N$, 
\begin{align}
\sum_{j\in J_{k}}\Delta g_{j}=\sum_{A\in\pi_{k}}\Big(\sup_{\bx\in A}g(\bx)-\inf_{\bx\in A}g(\bx)\Big)\leq2^{k-\lfloor k/d\rfloor}\n{g}_{\mathrm{TV}},\label{eq:d dim TV}
\end{align}
where $\lfloor\cdot\rfloor$ is the floor function. Combined with
\eqref{eq:v/p}, we arrive at 
\begin{align*}
\sum_{j\in J_{k}}\sqrt{\frac{V_{j}}{p_{j}}}\leq(1+\delta^{-1})^{3/2}2^{k-\lfloor k/d\rfloor}\n{g}_{\mathrm{TV}}.
\end{align*}
Inserting into \eqref{eq:holder2} gives 
\begin{align}
\sum_{j\in J_{k}}V_{j}\leq((1+\delta^{-1})^{3/2}2^{k-\lfloor k/d\rfloor}\n{g}_{\mathrm{TV}}2^{-k}\var(Y))^{2/3}\leq(1+\delta^{-1})2^{-2\lfloor k/d\rfloor/3}(\n{g}_{\mathrm{TV}}\var(Y))^{2/3}.\label{eq:sum V_j}
\end{align}

Finally, inserting \eqref{eq:uj} and \eqref{eq:sum V_j} into \eqref{eq:construction}
yields 
\begin{align*}
\E[(Y-M_{k})^{2}] & \leq(1+\delta)\,\E[(Y-g(\bX))^{2}]+(1+\delta^{-1})2^{-2\lfloor k/d\rfloor/3}(\n{g}_{\mathrm{TV}}\var(Y))^{2/3}.
\end{align*}
This proves \eqref{eq:toprove} and hence concludes the proof. 
\end{proof}
\begin{proof}[Proof of Remark \ref{rem:better bound}]
In the proof of Theorem \ref{thm:minimax exp decay 2}, the bound \eqref{eq:varub} can be improved to 
\[
\var(Y)^{1/2}\leq\E[(Y-g(\bX))^{2}]^{1/2}+\frac{1}{2}\Delta g.
\]
Inserting the following inequality (which follows from $(a+b)^{p}\leq2^{p-1}(a^{p}+b^{p})$
for $a,b\geq0$ and $p\geq1$) 
\[
\var(Y)^{2/3}\leq2^{1/3}\E[(Y-g(\bX))^{2}]^{2/3}+2^{-2/3}(\Delta g)^{4/3}
\]
into \eqref{eq:toprove} yields \eqref{eq:better bound}. 
\end{proof}
\begin{proof}[Proof of Theorem \ref{thm:atomic}] Recall that $\Delta Y:=\sup\supp Y-\inf\supp Y$. 
The only difference from the proof of Theorem \ref{thm:minimax exp decay 2}
is \eqref{eq:maxvar1}, since for non-marginally atomless measures
\eqref{eq:1/2var} and \eqref{eq:2-k} may not hold. Instead, by our
assumptions, \eqref{eq:1/2var} is replaced by 
\begin{align}
\max\Big\{\varphi_{L}(\hat{x}_{j}),\varphi_{R}(\hat{x}_{j})\Big\}\leq\frac{1}{2}\E[(Y-\E[Y\mid\bX\in A])^{2}\bone_{\{\bX\in A\}}]+\frac{(\Delta Y)^{2}}{N}.\label{eq:2+}
\end{align}
To see this, let $B$ and $C$ be disjoint events with $\p(C)=m$, and suppose that
$Y\in[a,b]$ almost surely, so $\Delta Y=b-a$. We let $\mu_B:=\E[Y\mid B]$. Since the conditional mean minimizes squared loss,
\begin{align*}
    \p(B\cup C)\Var(Y\mid B\cup C)
&\le
\E[(Y-\mu_B)^2\bone_{B\cup C}]\\
&=
\p(B)\Var(Y\mid B)+\E[(Y-\mu_B)^2\bone_C]
\le
\p(B)\Var(Y\mid B)+m(\Delta Y)^2.
\end{align*}
Therefore, adding an atom of weight $\leq1/N$ to a node increases the risk by at most $(\Delta Y)^{2}/N$.

As a consequence of \eqref{eq:2+}, if we denote
by 
\begin{align*}
u_{k}:=\max_{A\in\pi_{k}}\E[(Y-\E[Y\mid\bX\in A])^{2}\bone_{\{\bX\in A\}}],
\end{align*}
then $u_{k}$ satisfies the recursive inequalities 
\begin{align*}
u_{k+1}\leq\frac{u_{k}}{2}+\frac{(\Delta Y)^{2}}{N},~k\geq0;\quad u_{0}=\var(Y).
\end{align*}
To solve this, let $v_{k}:=u_{k}-2(\Delta Y)^{2}/N$. Suppose that
$u_{0}\geq2(\Delta Y)^{2}/N$. Then $v_{k}$ satisfies $v_{k+1}\leq v_{k}/2$
and $v_{0}\leq\var(Y)$. It follows that $v_{k}\leq2^{-k}\var(Y)$
for all $k\geq0$, and hence the maximum risk among the nodes at level
$k$ is 
\begin{align}
\max_{A\in\pi_{k}}\E[(Y-\E[Y\mid\bX\in A])^{2}\bone_{\{\bX\in A\}}]=u_{k}\leq2^{-k}\var(Y)+\frac{2(\Delta Y)^{2}}{N}.\label{eq:2}
\end{align}
It is also easy to check that \eqref{eq:2} is also satisfied for
the case $u_{0}<2(\Delta Y)^{2}/N$, that is, \eqref{eq:2} holds
in general. Replacing \eqref{eq:maxvar1} by \eqref{eq:2} leads to
\begin{align*}
\E[(Y-M_{k})^{2}]\leq\inf_{g\in\G}\bigg((1+\delta)\,\E[(Y-g(\bX))^{2}]+(1+\delta^{-1})2^{-2\lfloor k/d\rfloor/3}\Big(\n{g}_{\mathrm{TV}}(\var(Y)+2^{k+1}\frac{(\Delta Y)^{2}}{N})\Big)^{2/3}\bigg).
\end{align*}
Applying concavity of the function $x\mapsto x^{2/3}$ for $x>0$,
\eqref{eq:log ineq}, and that 
\begin{align*}
\n{g}_{\mathrm{TV}}^{2/3}\Big(2^{k+1}\frac{(\Delta Y)^{2}}{N}\Big)^{2/3}\leq\frac{\n{g}_{\mathrm{TV}}^{2}}{3}+\frac{2^{k+2}}{3}\frac{(\Delta Y)^{2}}{N},
\end{align*}
we obtain \eqref{eq:atomic}.
\end{proof}
\begin{proof}[Proof of Theorem \ref{thm:oracle}]

Note that if $\bX_{*}$ is marginally atomless, the assumption \eqref{eq:1/N}
for the joint distribution $(\bX,Y)$ in Theorem \ref{thm:atomic}
holds $\p_{*}$-almost surely. Consequently, the proof is almost verbatim
compared to Theorem 4.3 of \citet{klusowski2024large}, by applying
our Theorem \ref{thm:atomic} instead of Theorem 4.2 therein. The
only difference here is the extra term 
\begin{align*}
(1+\delta^{-1})2^{-2\lfloor k/d\rfloor/3}\frac{2^{k+2}(\Delta Y)^{2}}{3N}
\end{align*}
appearing in \eqref{eq:atomic}, where we remind the reader that $\Delta Y=\sup\supp Y-\inf\supp Y$.
Nevertheless, since the proof of Theorem 4.3 of \citet{klusowski2024large}
deals first with the case with bounded data, we may simply replace
$\Delta Y$ by that bound. Specifically, the first part of the proof
of Theorem 4.3 of \citet{klusowski2024large} assumes $\max_{i}|Y_{i}|\leq U$.
We may set $\Delta Y=2U$ and add the term 
\begin{align*}
(1+\delta^{-1})2^{-2\lfloor k/d\rfloor/3}\frac{2^{k+2}(\Delta Y)^{2}}{3N}\leq\frac{2^{k+5}U^{2}}{3N}
\end{align*}
in (B.28) therein (where we used $\delta\geq2^{-2\lfloor k/d\rfloor/3}$).
This extra term is carried until (B.34) therein, where it can be absorbed
by the term $U^{4}2^{k}\log(Nd)/N$ therein. The rest of the proof
remains unchanged. 
\end{proof}

\subsection{Proofs of results in Section \texorpdfstring{\ref{sec:RF}}{}}
Before we formally prove Theorem \ref{thm:oracle forest}, we point
out a key ingredient that distinguishes it from previous proofs where
only a single decision tree is involved. 
Recall that $n$ is the number of
trees in the forest, $\Sigma$ is the law of the splitting
dimensions, and $\E_{n}$ denotes the expectation with respect to the
empirical measure on $n$ samples from $\Sigma$. 
In Theorem \ref{thm:minimax exp decay 2},
the cyclicity of the cyclic MinimaxSplit algorithm is essential to
proving \eqref{eq:d dim TV}. However, a stronger upper bound for
\[
\sum_{j\in J_{k}}\Delta g_{j}=\sum_{A\in\pi_{k}}\Big(\sup_{\bx\in A}g(\bx)-\inf_{\bx\in A}g(\bx)\Big)
\]
exists by counting the (random) number of splits in each dimension,
which is generally of the form 
\begin{align}
\sum_{A\in\pi_{k}}\Big(\sup_{\bx\in A}g(\bx)-\inf_{\bx\in A}g(\bx)\Big)\leq\sum_{i=1}^{d}\xi_{k,i}\n{g_{i}}_{\mathrm{TV}},\label{eq:sik}
\end{align}
for some (random) positive integers $\xi_{k,i},~i\in[d]$ depending on the samples
from $\Sigma$. {See also \eqref{eq:2s} below}. The intuition in
proving Theorem \ref{thm:oracle forest} is that we keep \eqref{eq:sik}
as an upper bound, instead of \eqref{eq:d dim TV}. Then, with enough
randomness injected into the splitting dimensions (and hence the random
variables $\{\xi_{k,i}\}$), the right-hand side of \eqref{eq:sik} is small if averaged over all
trees in the forest, as shown by the next lemma.

\begin{lemma}\label{lemma:recursion} Fix $d\geq2$. Consider a stochastic
recursion defined as follows. Let $(\xi_{0,1},\dots,\xi_{0,d})=(1,\dots,1)$
and for each $k\geq1$, given $(\xi_{k-1,1},\dots,\xi_{k-1,d})$,
we define $(\xi_{k,1},\dots,\xi_{k,d})$ through the following procedure:
\begin{itemize} 

\item sample independent copies $({\xi}'_{k-1,1},\dots,{\xi}'_{k-1,d})$
and $({\xi}''_{k-1,1},\dots,{\xi}''_{k-1,d})$ of $(\xi_{k-1,1},\dots,\xi_{k-1,d})$; 

\item uniformly sample $\zeta_{k}\in[d]$; 

\item define $\xi_{k,\zeta_{k}}=\max\{\xi'_{k-1,\zeta_{k}},{\xi}''_{k-1,\zeta_{k}}\}$; 

\item for $j\in[d]$ and $j\neq\zeta_{k}$, define $\xi_{k,j}=\xi'_{k-1,j}+{\xi}''_{k-1,j}$.
\end{itemize} Then $\E[\xi_{k,j}]\leq2^{k}e^{-k/(4d)}$ for all $j\in[d]$.
\end{lemma}
\begin{rem}
The rate $e^{-1/(4d)}$ is not tight. Finding the tight exponent appears
to be a non-trivial task, even for $d=2$. 
\end{rem}

\begin{proof}[Proof of Lemma \ref{lemma:recursion}]
By construction, for each fixed $k\geq0$, the law of each vector
$(\xi_{k,1},\dots,\xi_{k,d})$ is exchangeable. As a consequence,
$\xi_{k,j}\dd W_{k}$ for all $j\in[d]$, where $\{W_{k}\}_{k\geq0}$ is defined as follows:
$W_{0}=1$, and for each $k\geq1$, we take independent copies ${W}'_{k-1},W''_{k-1}$
of $W_{k-1}$ and define 
\begin{align*}
W_{k}=\begin{cases}
{W}'_{k-1}+{W}''_{k-1} & \text{ with probability }1-1/d;\\
\max\{{W}'_{k-1},{W}''_{k-1}\} & \text{ with probability }1/d.
\end{cases}
\end{align*}
It remains to find $\E[W_{k}]$. It follows from construction that
$\E[(W_{0})^{2}]=1$ and for each $k\geq1$, 
\begin{align*}
\E[(W_{k})^{2}] & =(1-\frac{1}{d})\E\Big[\Big({W}'_{k-1}+{W}''_{k-1}\Big)^{2}\Big]+\frac{1}{d}\E\Big[\Big(\max\{{W}'_{k-1},{W}''_{k-1}\}\Big)^{2}\Big]\\
 & \leq4(1-\frac{1}{d})\E[(W_{k-1})^{2}]+\frac{2}{d}\E[(W_{k-1})^{2}]\\
 & =(4-\frac{2}{d})\E[(W_{k-1})^{2}].
\end{align*}
where we have used the inequality $(a+b)^{2}\leq2(a^{2}+b^{2})$.
It follows that 
\[
\E[W_{k}]\leq\sqrt{\E[(W_{k})^{2}]}\leq(4-\frac{2}{d})^{k/2}\leq2^{k}e^{-k/(4d)},
\]
as desired. 
\end{proof}
\begin{proof}[Proof of Theorem \ref{thm:oracle forest}]
In the following, we use $S$ to denote a random variable with law
$\Sigma$. We first establish an upper bound for 
\[
\E[(Y-\E_{n}[M_{k}^{S}])^{2}].
\]
Since the randomness from $\E_{n}$ is independent from that of $(\bX,Y)$,
we have by Jensen's inequality, 
\begin{align}
\E[(Y-\E_{n}[M_{k}^{S}])^{2}]\leq\E[\E_{n}[\E[(Y-M_{k}^{S})^{2}]]].\label{eq:jensen}
\end{align}
Here, the outer expectation is with respect to the randomness from
the empirical measure $\E_{n}$ (the law of $(\bX,Y)$ is fixed in
the settings of Theorems \ref{thm:minimax exp decay 2} and \ref{thm:atomic}),
and the inner expectation is with respect to the joint law of $(\bX,Y,\{M_{k}^{S}\})$.
In other words, we trivially bound the risk of the average by the
average of the risks among the $n$ trees.

Denote the empirical samples from $\Sigma$ by  $S_{1},\dots,S_{n}$.
Our next goal is to control 
\[
\E_{n}[\E[(Y-M_{k}^{S})^{2}]]=\frac{1}{n}\sum_{\ell=1}^{n}\E[(Y-M_{k}^{S_{\ell}})^{2}].
\]
We apply the same decomposition as in step II of the proof of Theorem
\ref{thm:minimax exp decay 2}, with the specific choice of $\delta=1$.
For $j\in J_{k}$, we let $g_{j}^{\ell}=g|_{I_{j}^{\ell}}$, where
$\{I_{j}^{\ell}\}_{j\in J_{k}}$ is the partition at level $k$ of
a tree whose splitting rule is given by $S_{\ell}$. Then \eqref{eq:v/p}
now has the form 
\[
\sqrt{\frac{V_{j}^{\ell}}{p_{j}^{\ell}}}\leq2^{3/2}\Delta g_{j}^{\ell}.
\]
Proceeding as in step III of the proof of Theorem \ref{thm:minimax exp decay 2},
\eqref{eq:sum V_j} becomes 
\begin{align}
\sum_{j\in J_{k}^{\ell}}V_{j}^{\ell}\leq\Big(2^{-k}\var(Y)2^{3/2}\sum_{j\in J_{k}^{\ell}}\Delta g_{j}^{\ell}\Big)^{2/3},\label{eq:sum V_j 2}
\end{align}
and hence for all $g\in\G$ with the representation $g=g_{1}+\dots+g_{d}$,
\begin{align}
\E_{n}[\E[(Y-M_{k}^{S})^{2}]]\leq2\E[(Y-g(\bX))^{2}]+2\var(Y)^{2/3}2^{-2k/3}\frac{1}{n}\sum_{\ell=1}^{n}\Big(\sum_{j\in J_{k}^{\ell}}\Delta g_{j}^{\ell}\Big)^{2/3}.\label{eq:EnE}
\end{align}
Next, we construct the (random) values $\{\xi_{k,i}^{\ell}\}_{\ell\in[n],\,i\in[d]}$
such that for each $k\in\N$ and $\ell\in[n]$, 
\begin{align}
\sum_{j\in J_{k}^{\ell}}\Delta g_{j}^{\ell}\lst\sum_{i=1}^{d}\xi_{k,i}^{\ell}\n{g_{i}}_{\mathrm{TV}},\label{eq:2s}
\end{align}
where $\lst$ denotes stochastic dominance. Define $\{\xi_{k,i}^{\ell}\}_{\ell\in[n],\,i\in[d]}$
as follows: 
\begin{itemize}
\item $\{\xi_{k,i}^{\ell}\}_{i\in[d]},\,\ell\in[n]$ are i.i.d.~copies of $\{\xi_{k,i}\}_{i\in[d]}$. 
\item $\xi_{0,i}=1$ for all $i\in[d]$. 
\item For $k\geq1$, given the law of $\{\xi_{k-1,i}\}_{i\in[d]}$, generate
two i.i.d.~copies $\{\xi'_{k-1,i}\}_{i\in[d]}$ and $\{\xi''_{k-1,i}\}_{i\in[d]}$,
independently sample a uniform $\zeta_{k}\in[d]$, and define 
\[
\xi_{k,\zeta_{k}}=\max\{\xi'_{k-1,\zeta_{k}},\xi''_{k-1,\zeta_{k}}\};\quad\xi_{k,i}=\xi'_{k-1,i}+\xi''_{k-1,i},~i\neq\zeta_{k}.
\]
\end{itemize}
See Figure \ref{fig:boxes} for an illustration
of the first few values of $\{\xi_{k,i}\}$.

\begin{figure}[h!]
\centering \includegraphics[width=0.7\linewidth]{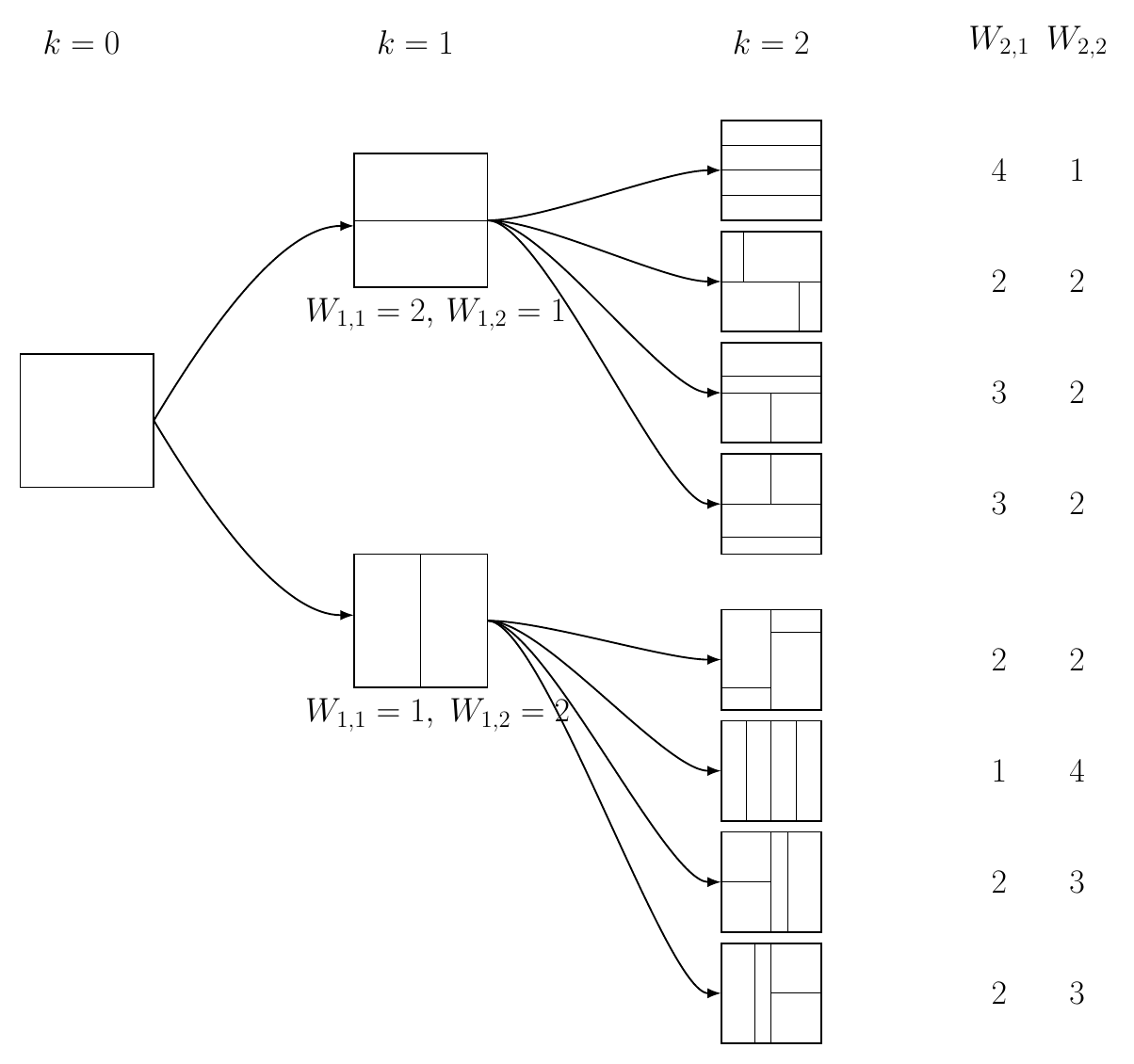}
\caption{Showcasing the possibilities of the splits for levels
$k=0,1,2$, with $d=2$. }
\label{fig:boxes}
\end{figure}

To see that this definition yields \eqref{eq:2s}, we apply induction. The base case $k=0$ is
easy as $\sup g-\inf g\leq\sum_{i=1}^{d}\n{g_{i}}_{\mathrm{TV}}$.
Suppose that \eqref{eq:2s} holds with $k-1$. Then given a tree of
depth $k$ (where we recall that the split dimensions of all nodes
are i.i.d.~uniform in $[d]$), we condition on the split dimension
at level zero (node $\emptyset$, say into child nodes $u,v$),
which we denote by $\zeta_{k}$. Next, we decompose each variation
$\Delta g_{j}^{\ell}$ in the $d$ dimensions by $\Delta g_{j}^{\ell}\leq\sum_{i=1}^{d}\Delta g_{j,i}^{\ell}$,
where $\Delta g_{j,i}^{\ell}$ is the variation of the $i$-th coordinate
of $g$ in $I_{j}$ for sample $\ell$.
\begin{itemize}
\item For a dimension $i\neq\zeta_{k}$, we apply the trivial bound that
the variations over dimension $i$ for $j\in J_{k}^{\ell}$ is at
most the sum of the variations among the descendants of $u$ and of
$v$. By induction hypothesis, we have 
\begin{align*}
\sum_{j\in J_{k}^{\ell}}\Delta g_{j,i}^{\ell} & \leq\sum_{\substack{j\in J_{k}^{\ell}\\
j\succ u
}
}\Delta g_{j,i}^{\ell}+\sum_{\substack{j\in J_{k}^{\ell}\\
j\succ v
}
}\Delta g_{j,i}^{\ell}\\
 & \lst\xi'_{k-1,i}\n{g_{i}}_{\mathrm{TV}}+\xi''_{k-1,i}\n{g_{i}}_{\mathrm{TV}}=\xi_{k,i}\n{g_{i}}_{\mathrm{TV}},
\end{align*}
where $j\succ u$ means that the set $I_{j}$ is a descendant of $u$
in the binary subtree representation.
\item For dimension $\zeta_{k}$, we again use the fact that the variation
$\sum_{j\in J_{k}^{\ell}}\Delta g_{j,\zeta_{k}}^{\ell}$ is at most the sum
of variations among the descendants of $u,v$, but the total variation
$\n{g_{\zeta_{k}}}_{\mathrm{TV}}$ is already decomposed into two,
say $\n{g_{\zeta_{k}}^{(i)}}_{\mathrm{TV}},\,i=1,2$, corresponding
to $u,v$. Therefore, by inductive hypothesis, 
\begin{align*}
\sum_{j\in J_{k}^{\ell}}\Delta g_{j,\zeta_{k}}^{\ell} & \leq\sum_{\substack{j\in J_{k}^{\ell}\\
j\succ u
}
}\Delta g_{j,\zeta_{k}}^{\ell}+\sum_{\substack{j\in J_{k}^{\ell}\\
j\succ v
}
}\Delta g_{j,\zeta_{k}}^{\ell}\\
 & \lst\xi'_{k-1,{\zeta_{k}}}\n{g_{\zeta_{k}}^{(1)}}_{\mathrm{TV}}+\xi''_{k-1,{\zeta_{k}}}\n{g_{\zeta_{k}}^{(2)}}_{\mathrm{TV}}\\
 & \leq\max\{\xi'_{k-1,{\zeta_{k}}},\xi''_{k-1,{\zeta_{k}}}\}\n{g_{\zeta_{k}}}_{\mathrm{TV}}\\
 & =\xi_{k,\zeta_{k}}\n{g_{\zeta_{k}}}_{\mathrm{TV}}.
\end{align*}
\end{itemize}
This proves \eqref{eq:2s}. 

Applying \eqref{eq:2s} and Hölder's inequality, we obtain 
\begin{align*}
\frac{1}{n}\sum_{\ell=1}^{n}\Big(\sum_{j\in J_{k}^{\ell}}\Delta g_{j}^{\ell}\Big)^{2/3} & \lst\frac{1}{n}\sum_{\ell=1}^{n}\Big(\sum_{i=1}^{d}\xi_{k,i}^{\ell}\n{g_{i}}_{\mathrm{TV}}\Big)^{2/3}\\
 & \leq2^{2k/3}\Big(\frac{1}{n}\sum_{\ell=1}^{n}\sum_{i=1}^{d}2^{-k}\xi_{k,i}^{\ell}\n{g_{i}}_{\mathrm{TV}}\Big)^{2/3}\\
 & \leq2^{2k/3}\Big(\frac{1}{n}\sum_{i=1}^{d}\n{g_{i}}_{\mathrm{TV}}\max_{i\in[d]}\sum_{\ell=1}^{n}2^{-k}\xi_{k,i}^{\ell}\Big)^{2/3}\\
 & \leq2^{2k/3}\n{g}_{\mathrm{TV}}^{2/3}\Big(\frac{1}{n}\max_{i\in[d]}\sum_{\ell=1}^{n}2^{-k}\xi_{k,i}^{\ell}\Big)^{2/3},
\end{align*}
where in the last step we recall that for $g\in\G$, $\n{g}_{\mathrm{TV}}=\sum_{i=1}^{d}\n{g_{i}}_{\mathrm{TV}}$.
The next step is to bound from above the inner maximum, where we recall
that under the outer expectation (as $S_{1},\dots,S_{n}$ are i.i.d.
sampled from $\Sigma$), the random variables $\{\xi_{k,i}^{\ell}\}_{\ell\in[n],i\in[d]}$
are i.i.d.~in $\ell$. Employing Lemma \ref{lemma:recursion}, we
have for each $i\in[d]$ that 
\begin{align*}
\E\bigg[\sum_{\ell=1}^{n}2^{-k}\xi_{k,i}^{\ell}\bigg] & =n2^{-k}\E[\xi_{k,1}^{\ell}]\leq ne^{-\frac{k}{4d}}.
\end{align*}
By the union bound and the Hoeffding inequality (note that $2^{-k}\xi_{k,i}^{\ell}\in[0,1]$),
we have for $u\geq0$ that 
\begin{align*}
\p\Big(\max_{i\in[d]}\frac{1}{n}\sum_{\ell=1}^{n}2^{-k}\xi_{k,i}^{\ell}\geq e^{-\frac{k}{4d}}+u\Big) & \leq d\,\p\Big(\frac{1}{n}\sum_{\ell=1}^{n}2^{-k}\xi_{k,i}^{\ell}\geq e^{-\frac{k}{4d}}+u\Big)\\
 & \leq d\,\p\Big(\sum_{\ell=1}^{n}2^{-k}\xi_{k,i}^{\ell}\geq\E\Big[\sum_{\ell=1}^{n}2^{-k}\xi_{k,i}^{\ell}\Big]+nu\Big)\leq de^{-2nu^{2}}.
\end{align*}
Therefore, a standard computation yields 
\begin{align*}
\E\Big[\Big(\max_{i\in[d]}\frac{1}{n}\sum_{\ell=1}^{n}2^{-k}\xi_{k,i}^{\ell}\Big)^{2/3}\Big]\leq e^{-\frac{k}{6d}}+\frac{d}{n^{1/3}}.
\end{align*}
Inserting into \eqref{eq:EnE} and taking expectation, we see that
for all $g\in\G$, 
\begin{align*}
\E[\E_{n}[\E[(Y-M_{k}^{S})^{2}]]]\leq2\E[(Y-g(\bX))^{2}]+2\var(Y)^{2/3}\n{g}_{\mathrm{TV}}^{2/3}\Big(e^{-\frac{k}{6d}}+\frac{d}{n^{1/3}}\Big).
\end{align*}
Using \eqref{eq:jensen} and \eqref{eq:log ineq}, we obtain that
for all $g\in\G$, 
\begin{align*}
\E[(Y-\E_{n}[M_{k}^{S}])^{2}]\leq2\E[(Y-g(\bX))^{2}]+2\Big(\frac{4}{3}\E[(Y-g(\bX))^{2}]+\n{g}_{\mathrm{TV}}^{2}\Big)\Big(e^{-\frac{k}{6d}}+\frac{d}{n^{1/3}}\Big).
\end{align*}
Following the same lines as the proof of Theorem \ref{thm:atomic},
we obtain that in the presence of marginal atoms of weights at most
$1/N$, it holds for all $g\in\G$, 
\begin{align*}
\E[(Y-\E_{n}[M_{k}^{S}])^{2}]\leq2\E[(Y-g(\bX))^{2}]+\frac{8}{3}\Big(\E[(Y-g(\bX))^{2}]+\n{g}_{\mathrm{TV}}^{2}+\frac{2^{k}M^{2}}{N}\Big)\Big(e^{-\frac{k}{6d}}+\frac{d}{n^{1/3}}\Big).
\end{align*}
Finally, we proceed with the arguments in proving Theorem \ref{thm:oracle}.
Again we follow the arguments in the proof of Theorem 4.3 for CART
of \citet{klusowski2024large}. By setting $M=2U$ where $U=\max|Y_{i}|$,
the extra term of 
\[
\frac{8}{3}\frac{2^{k}M^{2}}{N}\Big(e^{-\frac{k}{6d}}+\frac{d}{n^{1/3}}\Big)\leq C\Big(\frac{2^{k}U^{2}}{N}+\frac{2^{k}U^{2}d}{Nn^{1/3}}\Big)
\]
\sloppy is carried until (B.34) of \citet{klusowski2024large}. The
term $2^{k}U^{2}/N$ is absorbed by the term $U^{4}2^{k}\log(Nd)/N$
therein. The other term $2^{k}U^{2}d/(Nn^{1/3})$ is carried into
(B.35), and is bounded by $C2^{k}(\log N)d/(Nn^{1/3})$ under the
choice of $U=\n{g_{*}}_{\infty}+2\sigma\sqrt{2\log N}$ at the end
of that proof, where $C>0$ may depend on $\n{g_{*}}_{\infty}$ and
$\sigma$. This concludes \eqref{eq:oracle forest}. 
\end{proof}


To prove Theorem~\ref{thm:random-minimax-visible}, we first record the expected oscillation estimate that replaces the deterministic cyclic counting
argument.

\begin{lemma} 
\label{lem:visible-oscillation}
Assume the setting of Definition~\ref{def:positive-visibility}, and let
\[
g(\bx)=\sum_{j\in S} g_j(x_j)\in \mathcal G_S.
\]
Then for every $k\ge 0$,
\[
\E\left[
\sum_{A\in \pi_k^\Xi}
\left(
\sup_{\bx\in A} g(\bx)-\inf_{\bx\in A} g(\bx)
\right)
\right]
\le
2^k e^{-\rho k/2}\sum_{j\in S}\|g_j\|_{\mathrm{TV}}.
\]
Consequently,
\[
\E\left[
\sum_{A\in \pi_k^\Xi}
\left(
\sup_{\bx\in A} g(\bx)-\inf_{\bx\in A} g(\bx)
\right)
\right]
\le
2^k e^{-\rho k/2}\|g\|_{\mathrm{TV}}.
\]
\end{lemma}

\begin{proof}Recall the notation from Lemma~\ref{lem:cyclic_oscillation_bound}. 
For $j\in S$, define
\[
T_j^\Xi(k):=
\sum_{A\in \pi_k^\Xi}\osc_{I_j(A)}(g_j).
\]
As in the proof of Lemma~\ref{lem:cyclic_oscillation_bound},
\[
\sup_A g-\inf_A g\le \sum_{j\in S}\osc_{I_j(A)}(g_j),
\]
so it suffices to bound $\E[T_j^\Xi(k)]$.

Fix $j\in S$. Conditional on $\pi_k^\Xi$, each splittable cell $A\in \pi_k^\Xi$ is split along
coordinate $j$ with conditional probability at least $\rho$. If $A$ is split along coordinate $j$,
then the sum of the oscillations of $g_j$ over the two children is at most $\osc_{I_j(A)}(g_j)$.
If $A$ is split along a coordinate different from $j$, then the $j$th interval is duplicated, so the
sum of the oscillations over the two children is exactly $2\osc_{I_j(A)}(g_j)$.

Therefore,
\[
\E\!\left[T_j^\Xi(k+1)\,\middle|\,\pi_k^\Xi\right]
\le
\bigl(\rho\cdot 1+(1-\rho)\cdot 2\bigr)T_j^\Xi(k)
=
(2-\rho)T_j^\Xi(k).
\]
Taking expectations and iterating,
\[
\E[T_j^\Xi(k)]\le (2-\rho)^k T_j^\Xi(0).
\]
Since $2-\rho = 2(1-\rho/2)\le 2e^{-\rho/2}$ and $T_j^\Xi(0)\le \|g_j\|_{\mathrm{TV}}$, we obtain
\[
\E[T_j^\Xi(k)]\le 2^k e^{-\rho k/2}\|g_j\|_{\mathrm{TV}}.
\]
Summing over $j\in S$ proves the first claim, and taking the infimum over additive representations
proves the second.
\end{proof}

\begin{proof}[Proof of Theorem~\ref{thm:random-minimax-visible}]
Fix a realization of the randomization $\Xi$. Conditional on $\Xi$, the tree is deterministic, and
the part of the proof of Theorem \ref{thm:minimax exp decay 2} that leads to \eqref{eq:toprove} goes through verbatim, with a suitable adaptation of the upper bound \eqref{eq:d dim TV} for the oscillation term.
More precisely, for every $g\in \mathcal G_S$, the conditional predictor $M_k^\Xi$ satisfies
\begin{align}
\E\!\left[(Y-M_k^\Xi)^2\,\middle|\,\Xi\right]
\le\,
&
(1+\delta)
\E[(Y-g(\bX))^2]
+
(1+\delta^{-1})
(\Var(Y)\Theta_k(\Xi;g))^{2/3},\label{eq:cond bound}
\end{align}
where
\[
\Theta_k(\Xi;g):=
2^{-k}
\sum_{A\in \pi_k^\Xi}
\left(
\sup_{\bx\in A} g(\bx)-\inf_{\bx\in A} g(\bx)
\right).
\]

Taking expectations over $\Xi$ and using Jensen's inequality for the concave map
$x\mapsto x^{2/3}$ yields
\[
\E\!\left[\Theta_k(\Xi;g)^{2/3}\right]
\le
\left(\E[\Theta_k(\Xi;g)]\right)^{2/3}.
\]
By Lemma~\ref{lem:visible-oscillation},
\[
\E[\Theta_k(\Xi;g)]
\le
2^{-k}\cdot 2^k e^{-\rho k/2}\|g\|_{\mathrm{TV}}
=
e^{-\rho k/2}\|g\|_{\mathrm{TV}}.
\]
Hence
\[
\E\!\left[\Theta_k(\Xi;g)^{2/3}\right]
\le
e^{-\rho k/3}\|g\|_{\mathrm{TV}}^{2/3}.
\]
Substituting this estimate into \eqref{eq:cond bound} and then following step I of the proof of Theorem~\ref{thm:minimax exp decay 2} gives \eqref{eq:visible-random-rate}.
\end{proof}

\subsection{Proofs of results in Section \ref{sec:Minimax-classify}}

The analysis in the classification setting also relies on two monotonicity envelopes for the conditional
probability. Recall \eqref{eq:Hh}. Lemma~\ref{lemma:continuous-1} provides left and right
bounds for the probability of an order event inside a cell in terms
of the functions 
\begin{align*}
    &\varphi_{L}(x)\coloneqq \mathbb{P}(\bX\in A,\,X_{j}< x)\,h \left(\mathbb{P}(Y=1\mid \bX\in A,\,X_{j}< x)\right);\\
    &\varphi_{R}(x)\coloneqq \mathbb{P}(\bX\in A,\,X_{j}\ge x)\,h \left(\mathbb{P}(Y=1\mid \bX\in A,\,X_{j}\ge x)\right),
\end{align*}
which are nondecreasing and nonincreasing, respectively. Lemma~\ref{lemma:entropy bound}
further provides an information inequality that
upper bounds the entropy. These two lemmas are the entropic analogues of the unconditional variance 
contraction used in the regression section; they are the technical
bridges that allow us to carry over the exponential
decay arguments from squared loss to cross-entropy.

\begin{lemma}\label{lemma:continuous-1} Suppose that $\bX$ is marginally
atomless and $j\in[d]$. Then the function $\varphi_L$ 
 is non-decreasing and continuous, and $\varphi_R$ 
is non-increasing and continuous. Moreover, denote by $\pi_{k}=\{I_{j}\}_{j\in J_{k}}$ the partition of $\R^{d}$
formed at level $k$ from the cyclic MinimaxSplit construction. Then 
$$\max_{j\in J_{k}}\E[\log(1+e^{-YM_{k}})\bone_{\{\bX\in I_{j}\}}]\leq2^{-k}\E[\log(1+e^{-YM_{0}})]=2^{-k}H(Y).$$
\end{lemma}

\begin{proof}
To see the monotonicity of $\varphi_L$, we take arbitrary $a<b$. Since the conditional entropy is bounded from above by the entropy, we have
\begin{align*}
    h \left(\mathbb{P}(Y=1\mid \bX\in A,\,X_{j}< b)\right)\geq \p(X_j<a\mid \bX\in A,\,X_{j}< b)\,h \left(\mathbb{P}(Y=1\mid \bX\in A,\,X_{j}< a)\right).
\end{align*}
Rearranging yields that $\varphi_L(a)\leq\varphi_L(b)$.
The continuity of $\varphi_L$ and $\varphi_R$ follows from the marginally atomless condition.
The final claim follows in the same way as the derivation of \eqref{eq:2-k}.
\end{proof}

\begin{lemma}\label{lemma:entropy bound} Let $Y$ be a $\{\pm 1\}$-valued random variable. For any joint distribution
$(X,Y)$, it holds that 
\begin{align}
H(Y)\leq\E[\log(1+e^{-XY})]+\frac{1}{4}|\sup\supp X-\inf\supp X|.\label{eq:var bound-1}
\end{align}
\end{lemma} 
\begin{proof}
We may assume that $\sup\supp X-\inf\supp X
<\infty$, otherwise the claim is trivial. 
Given the marginal law of $X$, the right-hand side of \eqref{eq:var bound-1}
is minimized if $X$ and $Y$ are comonotonic. Moreover, with $\sup\supp X$ and $\inf\supp X
$ fixed, the right-hand side of \eqref{eq:var bound-1}
is minimized when we have the equivalence of events $\{X=\sup\supp X\}=\{Y=1\}$ and $\{X=\inf\supp X\}=\{Y=-1\}$. This is achieved by setting, with $a=\inf\supp X\leq\sup\supp X=b$, that $p=\p(X=b,Y=1)=1-\p(X=a,Y=-1)$, for some $p\in[0,1]$. It remains to show the deterministic inequality
\[
\kappa_{a,b}(p):=-h(p)+(p\log(1+e^{-b})+(1-p)\log(1+e^{a}))+\frac{1}{4}|b-a|\geq 0.
\]
Note that $\kappa_{a,b}$ is convex in $p$ and $\kappa_{a,b}'(p)=0$ if and only if 
\[
\log\Big(\frac{1-p}{p}\Big)=\log\Big(\frac{1+e^{-b}}{1+e^{a}}\Big),
\]
or $p=(1+e^{a})/((1+e^{a})+(1+e^{-b}))$. It remains to prove 
\[
\log\Big(\frac{(1+e^{a})(1+e^{-b})}{(1+e^{a})+(1+e^{-b})}\Big)\geq-\frac{1}{4}|b-a|,
\]
or equivalently, 
\begin{align}
    \log((1+e^{a})^{-1}+(1+e^{-b})^{-1})\leq\frac{1}{4}|b-a|.\label{eq:ab~}
\end{align}
With $D:=b-a>0$ fixed, the left-hand side of \eqref{eq:ab~} 
is maximized at $a=-D/2$ and $b=D/2$ by elementary calculus. 
On the other hand, it is easy to verify by elementary calculus that $\log(2/(1+e^{-D/2}))\leq D/4=|b-a|/4$. 
This proves \eqref{eq:ab~} and hence completes the proof.
\end{proof}

\begin{proof}[Proof of Theorem \ref{thm:minimax exp decay 2-1}]
We first recall the notations from the proof of Theorem \ref{thm:minimax exp decay 2}. 
Denote by $\pi_{k}=\{I_{j}\}_{j\in J_{k}}$ the partition of $\R^{d}$
formed at level $k$ from the cyclic MinimaxSplit construction. Note that $M_{k}\bone_{\{\bX\in I_{j}\}}=y_{k,j}\bone_{\{\bX\in I_{j}\}}$, where $y_{k,j}=\log(\p(Y=1\mid\bX\in I_{j})/\p(Y=-1\mid\bX\in I_{j}))$.
Define $g_{j}=g|_{I_{j}}$ and recall that $\Delta g_{j}=\sup g_{j}-\inf g_{j}\leq\n{g_{j}}_{\mathrm{TV}}$.

By applying Lemma \ref{lemma:entropy bound} to the joint distribution
$(g(\bX),Y)\mid\bX\in I_{j}$, decompose 
\begin{align}
\E[\log(1+e^{-YM_{k}})\bone_{\{\bX\in I_{j}\}}]=\E[\log(1+e^{-Yy_{k,j}})\bone_{\{\bX\in I_{j}\}}]\leq U_{j}+V_{j},\label{eq:uv-1}
\end{align}
 where $U_{j}=\E[\log(1+e^{-Yg(\bX)})\bone_{\{\bX\in I_{j}\}}]$ and
\[
V_{j}=\min\Big\{\E[\log(1+e^{-YM_{k}})\bone_{\{\bX\in I_{j}\}}],\frac{p_{j}\Delta g_{j}}{4}\Big\},
\]
where $p_j:=\p(\bX\in I_j)$. 
 We have by construction and \eqref{eq:uv-1} that 
\begin{align}
\E[\log(1+e^{-YM_{k}})]=\sum_{j\in J_{k}}\E[\log(1+e^{-YM_{k}})\bone_{\{\bX\in I_{j}\}}]\leq\sum_{j\in J_{k}}U_{j}+\sum_{j\in J_{k}}V_{j}.\label{eq:construction-1}
\end{align}
The first term is easy to control by definition: 
\begin{align}
\sum_{j\in J_{k}}U_{j}\leq\sum_{j\in J_{k}}\E[\log(1+e^{-Yg(\bX)})\bone_{\{\bX\in I_{j}\}}]=\E[\log(1+e^{-Yg(\bX)})].\label{eq:uj-1}
\end{align}
To control $\sum_{j}V_{j}$, we apply the same Hölder's inequality
argument as in the proof of Theorem \ref{thm:martingale approx} in
the regression setting. More precisely, we have 
\begin{align}
\begin{split}\sum_{j\in J_{k}}V_{j} & \leq\Big(\sum_{j\in J_{k}}p_{j}\big(\frac{V_{j}}{p_{j}}\big)^{2}\Big)^{1/2}\Big(\sum_{j\in J_{k}}p_{j}\Big)^{1/2}\\
 & \leq\Bigg((\max_{j\in J_{k}}V_{j})\sum_{j\in J_{k}}\frac{V_{j}}{p_{j}}\Bigg)^{1/2}\leq\Bigg((\max_{j\in J_{k}}\E[\log(1+e^{-YM_{k}})\bone_{\{\bX\in I_{j}\}}])\sum_{j\in J_{k}}\frac{V_{j}}{p_{j}}\Bigg)^{1/2}.
\end{split}
\label{eq:holder2-1}
\end{align}
To further bound the right-hand side of \eqref{eq:holder2-1}, we
recall from Lemma \ref{lemma:continuous-1} that as a consequence of the marginally
atomless property, 
\begin{align}
\max_{j\in J_{k}}\E[\log(1+e^{-YM_{k}})\bone_{\{\bX\in I_{j}\}}]\leq2^{-k}\E[\log(1+e^{-YM_{0}})]=2^{-k}H(Y).\label{eq:maxvar1-1}
\end{align}
Note that ${V_{j}}/{p_{j}}\leq\Delta g_{j}/4.$ Following the proof of Theorem \ref{thm:martingale approx}, we arrive at 
\begin{align*}
\sum_{j\in J_{k}}\frac{V_{j}}{p_{j}}\leq2^{-2+k-\lfloor k/d\rfloor}\n{g}_{\mathrm{TV}}.
\end{align*}
Inserting into \eqref{eq:holder2-1} gives (using that $H(Y)\leq\log2$)
\begin{align}
\sum_{j\in J_{k}}V_{j}\leq(2^{k-2-\lfloor k/d\rfloor}\n{g}_{\mathrm{TV}}2^{-k}H(Y))^{1/2}\leq\sqrt{\frac{\log2}{4}}\,2^{-\lfloor k/d\rfloor/2}\n{g}_{\mathrm{TV}}^{1/2}.\label{eq:sum V_j-1}
\end{align}
Finally, inserting \eqref{eq:uj-1} and \eqref{eq:sum V_j-1} into
\eqref{eq:construction-1} yields \eqref{eq:d-dim rate-1}. 
\end{proof}
\begin{proof}[Proof of Theorem \ref{thm:atomic-1}]
The only difference from the proof of Theorem \ref{thm:minimax exp decay 2-1}
is \eqref{eq:maxvar1-1}. Denoting by $\hat{x}_j$ the optimal MinimaxSplit location on $A$ in dimension $j$, we have 
\begin{align*}
\max\Big\{\varphi_{L}(\hat{x}_{j}),\varphi_{R}(\hat{x}_{j})\Big\}\leq\frac{1}{2}\p(\bX\in A)\,h(\p(Y=1\mid\bX\in A))+\frac{\log N}{N},
\end{align*}
because adding an atom of weight $\leq1/N$ to a node increases the
risk by at most 
$$\frac{1}{N}\,\log N+\Big(1-\frac{1}{N}\Big)\log\Big(\frac{1}{1-\frac{1}{N}}\Big)\leq \frac{1+\log N}{N},$$
which can be seen from the extreme case
with a sample $1$ added to samples of $-1$. As a consequence, if
we denote by 
\begin{align*}
u_{k}:=\max_{A\in\pi_{k}}\p(\bX\in A)\,h(\p(Y=1\mid\bX\in A)),
\end{align*}
then $u_{k}$ satisfies the recursive inequalities 
\begin{align*}
u_{k+1}\leq\frac{u_{k}}{2}+\frac{\log (eN)}{N},~k\geq0;\quad u_{0}=H(Y).
\end{align*}
Solving this yields 
\begin{align}
\max_{A\in\pi_{k}}\p(\bX\in A)\,h(\p(Y=1\mid\bX\in A))=u_{k}\leq2^{-k}H(Y)+\frac{2\log (eN)}{N}.\label{eq:2-1}
\end{align}
It is also easy to check that \eqref{eq:2-1} is also satisfied for
the case $u_{0}<2\log (eN)/N$, that is, \eqref{eq:2-1} holds in general.
Replacing \eqref{eq:maxvar1-1} by \eqref{eq:2-1} leads to 
\begin{align*}
\E[\log(1+e^{-YM_{k}})]\leq\inf_{g\in\G}\bigg(\E[\log(1+e^{-Yg(\bX)})]+\sqrt{\frac{\log2}{4}}\,2^{-\lfloor k/d\rfloor/2}\Big(\n{g}_{\mathrm{TV}}(1+2^{k+1}\frac{\log (eN)}{N})\Big)^{1/2}\bigg).
\end{align*}
Applying concavity of the function $x\mapsto x^{1/2}$ for $x>0$ and that 
\begin{align*}
\n{g}_{\mathrm{TV}}^{1/2}\Big(2^{k+1}\frac{\log N}{N}\Big)^{1/2}\leq\frac{\n{g}_{\mathrm{TV}}}{2}+\frac{2^{k}\log (eN)}{N},
\end{align*}
we obtain \eqref{eq:atomic-1}. 
\end{proof}

\begin{proof}[Proof of Theorem \ref{thm:oracle-1}]
    The proof follows in exactly the same
way as Theorem 4.3 of \citet{klusowski2024large} in the C4.5 case,
as the arguments therein do not depend on how the tree is constructed,
as long as the splitting rule is axis-aligned. 
\end{proof}

\section{Additional example for anti-symmetry-breaking preference}

\label{sec:Additional-Example-for}

In this appendix, we provide an empirical version of Example \ref{ex:ASBP}, showing that the ASBP persists even when samples are involved (meaning that the symmetry is not perfect under the empirical measure).
\begin{example}
\label{ex:ASBP2} Define $f(\bx)=x_{1}+|x_{2}|,~\bx=(x_{1},x_{2})\in[-1,1]^{2}$,
which is Lipschitz and bounded. Let $\p_{*}$ denote the law of $\bX$
in Example \ref{ex:ASBP} and $\p_{n}$ denote the empirical law of
$\{(\bX_{i},Y_{i})\}_{1\leq i\leq n}$, where each $\bX_{i}$ is sampled
from $\p_{*}$ and $Y_{i}=f(\bX_{i})$. Denote by $\mathcal{R}$ the
set of all axis-aligned rectangles in $[-1,1]^{2}$.


Since $\n{f}_\infty\leq 2$ and the set of axis-aligned rectangles forms a VC class,
the class
\[
\mathcal H:=\{(f-m)^2\mathbf 1_R:\ R\in\mathcal R,\ m\in[-2,2]\}
\]
is a bounded Donsker class. For
\[
T_R:=\inf_{m\in[-2,2]} \E^{\p_*}[(f(\bX)-m)^2\mathbf 1_{\{\bX\in R\}}],
\qquad
V_R:=\inf_{m\in[-2,2]} \E^{\p_n}[(f(\bX)-m)^2\mathbf 1_{\{\bX\in R\}}],
\]
we therefore have
\[
\sup_{R\in\mathcal R}|V_R-T_R|\leq \sup_{h\in\mathcal H}|\E^{\p_*}[h]-\E^{\p_n}[h]|=O_\p\Big(\frac{1}{\sqrt{n}}\Big).
\]
Here, we use the notation $\xi_{n}=O_{\p}(\alpha_{n})$ if the sequence
$\{\xi_{n}/\alpha_{n}\}$ is tight. 
Moreover, for each $R\in\mathcal R$, if $M_R:=\sum_{i=1}^n \mathbf 1_{\{\bX_i\in R\}}$,
then conditional on $M_R=m\ge1$,
\[
\E[V_R\mid M_R=m]=\frac{m-1}{n}\Var(f(\bX)\mid \bX\in R),
\]
while $V_R=0$ when $m=0$. Hence
\[
0\le T_R-\E[V_R]
=\frac{\p(M_R\ge1)}{n}\Var(f(\bX)\mid \bX\in R)
\le \frac{\|f\|_\infty^2}{n}\leq \frac{4}{n}.
\]
By the triangle inequality,
\begin{align}
\sup_{R\in\cR}|V_{R}-\E[V_{R}]|=O_{\p}\Big(\frac{1}{\sqrt{n}}\Big).\label{eq:VR}
\end{align}

Next, we apply \eqref{eq:VR} to analyze the split dimensions and
locations when applying MinimaxSplit to $\p_{n}$. For any $\ee>0$,
there exists $C=C(\ee)>0$ such that $\p(\forall R\in\cR,\,|V_{R}-\E[V_{R}]|<C/\sqrt{n})>1-\ee$.
Our goal is to show the deterministic fact that for some $C_{0}=C_{0}(\ee)>0$,
if $n\geq C_{0}64^{K}$, on the event 
\begin{align}
    \Big\{\forall R\in\cR,\,|V_{R}-\E[V_{R}]|<\frac{C}{\sqrt{n}}\Big\},\label{eq:Ce}
\end{align}
the splits of the first $K$ levels are all along the first coordinate
$x_{1}$. Consider a rectangle $R_{a,b}=[a,b]\times[-1,1]\in\cR$
where $-1\leq a<b\leq1$. Define 
\begin{align*}
R_{a,b;u}^{\uparrow} & =[a,b]\times[u,1],\quad R_{a,b;u}^{\downarrow}=[a,b]\times[-1,u],\quad u\in[-1,1];\\
R_{a,b;u}^{\leftarrow} & =[a,u]\times[-1,1],\quad R_{a,b;u}^{\rightarrow}=[u,b]\times[-1,1],\quad u\in[a,b].
\end{align*}
Also denote by $m(p)=\E[(M-1)\bone_{\{M\geq1\}}/M]$ where $M\lawis\mathrm{Bin}(n,p)$.
A standard computation (by first conditioning on the number of samples
in the rectangle) yields that 
\begin{align}
\begin{split}
    \E\left[V_{R_{a,b}}\right] & =m\Big(\frac{b-a}{2}\Big)\frac{b-a}{2}\frac{(b-a)^{2}+1}{12};\\
\E\left[V_{R_{a,b;u}^{\uparrow}}\right] & =m\Big(\frac{(1-u)(b-a)}{4}\Big)\frac{(1-u)(b-a)}{4}\Big(\frac{(b-a)^{2}}{12}+\phi(u)\Big);\\
\E\left[V_{R_{a,b;u}^{\downarrow}}\right] & =m\Big(\frac{(u+1)(b-a)}{4}\Big)\frac{(u+1)(b-a)}{4}\Big(\frac{(b-a)^{2}}{12}+\phi(-u)\Big);\\
\E\left[V_{R_{a,b;u}^{\leftarrow}}\right] & =m\Big(\frac{u-a}{2}\Big)\frac{u-a}{2}\frac{(u-a)^{2}+1}{12};\\
\E\left[V_{R_{a,b;u}^{\rightarrow}}\right] & =m\Big(\frac{b-u}{2}\Big)\frac{b-u}{2}\frac{(b-u)^{2}+1}{12},
\end{split}\label{eq:5eqs}
\end{align}
where 
\begin{align*}
\phi(u)=\begin{cases}
\frac{(1-u)^{2}}{12} & \text{ if }u\geq0;\\
\frac{u^{4}-4u^{3}-6u^{2}-4u+1}{12(1-u)^{2}} & \text{ if }u<0.
\end{cases}
\end{align*}
In the following, the constants $C$ (depending on $\ee$ only) may
vary from line to line. Note that $m(p)=1-O(1/(np))$ jointly as $np\to\infty$
(see e.g., Corollary 1 of \citet{wang2008asymptotic}). The following
results hold if $b-a\geq C/n^{1/6}$ for $n$ large enough. 
\begin{itemize}
\item Consider a split of $R_{a,b}$ into $R_{a,b;u}^{\leftarrow}$ and
$R_{a,b;u}^{\rightarrow}$. The MinimaxSplit location satisfies $|u-(a+b)/2|\leq C/\sqrt{n}$ by \eqref{eq:Ce} and \eqref{eq:5eqs}.
The risk decay is at least 
\[
\E\left[V_{R_{a,b}}\right]-\E\left[V_{R_{a,b;u}^{\leftarrow}}\right]-\E\left[V_{R_{a,b;u}^{\rightarrow}}\right]-\frac{3C}{\sqrt{n}}\geq\frac{(b-a)^{3}}{32}-\frac{C}{\sqrt{n}}.
\]
\item Consider a split of $R_{a,b}$ into $R_{a,b;u}^{\uparrow}$ and $R_{a,b;u}^{\downarrow}$.
The function $u\mapsto(1-u)\phi(u)$ is decreasing and differentiable
on $[-1,1]$ with a strictly negative derivative at $u=0$. Therefore,
the MinimaxSplit location satisfies $|u|\leq C/((b-a)\sqrt{n})$ by \eqref{eq:Ce} and \eqref{eq:5eqs}.
The risk decay is at most 
\[
\E\left[V_{R_{a,b}}\right]-\E\left[V_{R_{a,b;u}^{\downarrow}}\right]-\E\left[V_{R_{a,b;u}^{\uparrow}}\right]+\frac{3C}{\sqrt{n}}\leq\frac{C}{\sqrt{n}}.
\]
\end{itemize}
Therefore, if $b-a\geq C/n^{1/6}$, the MinimaxSplit dimension is
always $x_{1}$. Next, we show that under our assumption $n\geq C_{0}64^{K}$,
we always have $b-a\geq C/n^{1/6}$ for the first $K$ levels of splits.
The split at level $k=0$ is obvious. Denote by $a_{k}$ the smallest
size (measured in dimension $x_{1}$) of the rectangles at level $k$.
It follows that $a_{0}=2$ and for each $k$, $a_{k+1}\geq a_{k}/2-C/\sqrt{n}$.
Solving this recursion yields $a_{k}\geq2^{1-k}-2C/\sqrt{n}$ for
all $k$. In particular, $a_{K}\geq2^{1-K}-2C/\sqrt{n}\geq C/n^{1/6}$.

In summary, we have shown that if $n\geq C_{0}64^{K}$, with high
probability, the MinimaxSplit dimension is always along the first
dimension for the first $K$ levels of the tree. As a consequence,
the consistency of the MinimaxSplit algorithm is not always guaranteed
in the large-sample regime, unlike VarianceSplit (Theorem 4.3 of \citet{klusowski2024large})
and cyclic MinimaxSplit (Theorem \ref{thm:oracle}). 
\end{example}
\section{Additional synthetic experiments for the classification setting}
\label{sec:appendix-classification-synthetic}

In this appendix, we complement Section~\ref{sec:Minimax-classify} with
additional synthetic experiments that isolate the behavior of the
classification splitting rules in one dimension. Throughout, we consider
$Y\in\{\pm 1\}$ and write
\[
\eta(x)\coloneqq \p(Y=1\mid X=x),\qquad g_{*}(x)\coloneqq \log\Big(\frac{\eta(x)}{1-\eta(x)}\Big).
\]
We compare three axis-aligned tree constructions:
(i) entropy-CART, which minimizes the weighted sum of child entropies;
(ii) Gini-CART, which minimizes the weighted sum of child Gini impurities;\footnote{The criterion $h(p)=-p\log p-(1-p)\log(1-p)$ was used for entropy-CART, while Gini-CART uses the Gini impurity criterion $2p(1-p)$.}
and (iii) the entropy-based MinimaxSplit rule from Section~\ref{sec:Minimax-classify},
which minimizes the maximum of the two weighted child entropies.
Since $d=1$ in all experiments below, MinimaxSplit and cyclic MinimaxSplit coincide.
Unless otherwise stated, each tree is grown with minimum leaf size $8$, training sample size $250$,
and depth budget $k\in\{1,\dots,8\}$.

\subsection{Piecewise signal: predictive risk as a function of depth}
\label{subsec:appendix-classification-piecewise}

We first consider a piecewise-constant Bayes log-odds,
\[
g_{*}(x)=
\begin{cases}
2, & 0\le x<0.30,\\
-1.5, & 0.30\le x<0.60,\\
1, & 0.60\le x\le 1,
\end{cases}
\qquad X\lawis \mathrm{Unif}(0,1),
\]
so that $\eta(x)=e^{g_{*}(x)}/(1+e^{g_{*}(x)})$ is piecewise constant with two change-points.
The left panel of Figure~\ref{fig:appendix-classification-depth} reports the expected log-loss,
which is the empirical analogue of the surrogate risk in
\eqref{eq:d-dim rate-1} and \eqref{eq:atomic-1}.
The middle panel reports the corresponding misclassification error.

Two features of Figure~\ref{fig:appendix-classification-depth} are worth emphasizing.
First, entropy-CART and Gini-CART achieve the best shallow-tree performance:
at depth $k=2$ they attain an expected log-loss of around $0.52$, whereas MinimaxSplit is still near $0.54$.
This is consistent with the fact that the CART rules are more aggressive in optimizing the immediate impurity decrease.
Second, once the depth exceeds the first few levels, the ranking reverses.
From depth $k=3$ onward, MinimaxSplit has the smallest expected log-loss,
and the gap becomes substantial by depths $k=5,6,7,8$.
In particular, the MinimaxSplit curve stabilizes near $0.541$, while the two CART baselines drift upward toward $0.589$.
Thus, on this piecewise signal, the worst-child criterion trades some initial greediness for a substantially more stable partition at moderate and large depths.

The middle panel shows that the advantage is more pronounced for cross-entropy than for the $0$--$1$ loss.
This is also in line with Section~\ref{sec:Minimax-classify}: the primary theoretical control is on
$\E[\log(1+e^{-Y M_{k}})]$, and the excess classification risk is obtained only after the calibration step of Theorem~\ref{thm:oracle-1}.
Empirically, the misclassification curves are close after depth $k=3$, whereas the cross-entropy curves remain clearly separated.
The right panel shows that MinimaxSplit reaches a larger number of terminal cells by depth $5$,
but, as we explain next, these leaves are much more evenly populated.

\begin{figure}[t]
\centering
\includegraphics[width=0.98\textwidth]{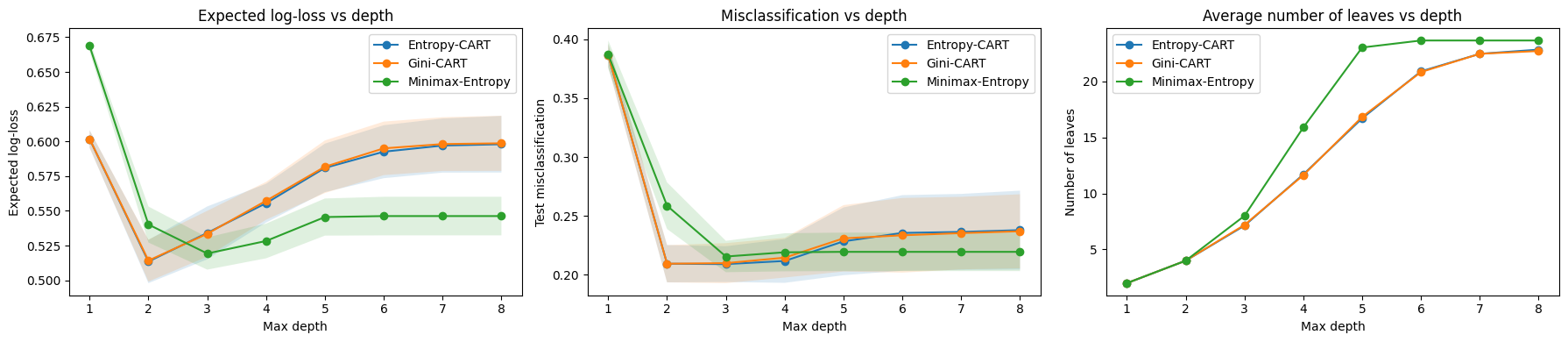}
\caption{Synthetic one-dimensional classification experiment with piecewise Bayes log-odds.
Left: expected log-loss versus depth. Middle: test misclassification versus depth.
Right: average number of leaves versus depth. In dimension $d=1$, the entropy-based MinimaxSplit rule coincides with cyclic MinimaxSplit.
The main empirical advantage appears in the surrogate risk: after a short shallow-tree disadvantage,
MinimaxSplit attains a lower and markedly more stable expected log-loss than entropy-CART and Gini-CART.}
\label{fig:appendix-classification-depth}
\end{figure}

\subsection{Leaf-balance as the geometric mechanism}
\label{subsec:appendix-classification-leaf-balance}

Figure~\ref{fig:appendix-classification-leaf-balance} makes the geometric mechanism behind the previous risk comparison explicit.
The left panel plots the mean number of samples per leaf on a logarithmic scale, and the right panel plots the standard deviation of the leaf sizes, again on a logarithmic scale.
The mean leaf sizes are broadly comparable once the depth is moderate, but the dispersion is dramatically different.
For example, at depth $k=4$ the standard deviation of the leaf sizes is about $4.9$ for MinimaxSplit, compared with roughly $13$ for the two CART baselines;
at depth $k=6$ the corresponding values are about $3.3$ versus about 6.5.

This is the empirical signature of the worst-child criterion.
Entropy-CART and Gini-CART reduce the average impurity, and therefore can continue to benefit from highly unbalanced splits as long as one child becomes pure enough.
By contrast, MinimaxSplit penalizes any split that leaves one child both large and highly impure.
The resulting trees are therefore more regular in the sense that the sample sizes of the terminal cells remain much closer to one another.
Viewed together with Figure~\ref{fig:appendix-classification-depth}, Figure~\ref{fig:appendix-classification-leaf-balance} suggests that the deeper-tree log-loss advantage of MinimaxSplit is not merely a numerical accident:
it is explained by a clear difference in partition geometry.

\begin{figure}[t]
\centering
\includegraphics[width=0.78\textwidth]{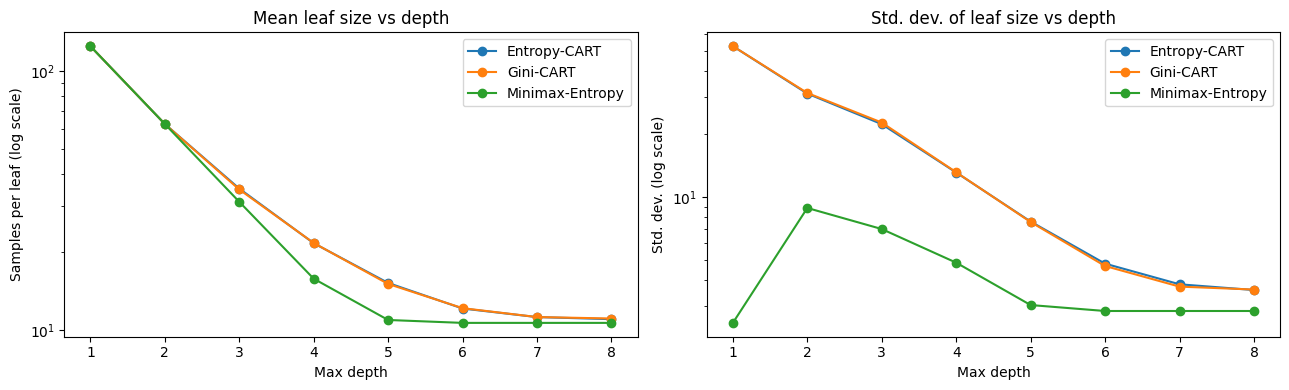}
\caption{Leaf-size statistics for the one-dimensional piecewise classification problem.
The left panel shows the mean number of samples per leaf, and the right panel shows the standard deviation of the leaf sizes.
MinimaxSplit produces terminal cells with much smaller leaf-size dispersion, which is the classification analogue of avoiding highly unbalanced end-cuts.}
\label{fig:appendix-classification-leaf-balance}
\end{figure}

\subsection{Continuous versus atomic covariates}
\label{subsec:appendix-classification-atomic}

Finally, Figure~\ref{fig:appendix-classification-atomic} compares the same piecewise signal in two sampling regimes.
In the left panel, $X\lawis \mathrm{Unif}(0,1)$ is marginally atomless, corresponding to the setting of Theorem~\ref{thm:minimax exp decay 2-1}.
In the right panel, $X$ is sampled from a uniform grid of $41$ equally spaced points in $[0,1]$, which places us in the purely atomic setting of Theorem~\ref{thm:atomic-1}.
The plotted quantity is again the expected log-loss.

The qualitative ordering is the same in both regimes.
In each case, the CART baselines improve quickly at the first few levels and then deteriorate as the tree continues to refine,
whereas the MinimaxSplit curve becomes nearly flat once it reaches its best value.
The advantage of MinimaxSplit is somewhat smaller in the atomic experiment than in the atomless one,
but it remains visible throughout the deeper levels.
This is consistent with the form of \eqref{eq:atomic-1}, where exponential decay in the atomless case is accompanied by an additional discrete-resolution term.
Empirically, the atomic curves flatten earlier, which is natural because once the grid resolution and the minimum-leaf constraint become active,
all three procedures have limited room for further refinement.

Overall, Figure~\ref{fig:appendix-classification-atomic} shows that the classification analogue of MinimaxSplit is not tied to the marginally atomless setting:
its advantage in cross-entropy persists when the covariate distribution is discrete, although the attainable gain is moderated by the finite support.

\begin{figure}[t]
\centering
\includegraphics[width=0.78\textwidth]{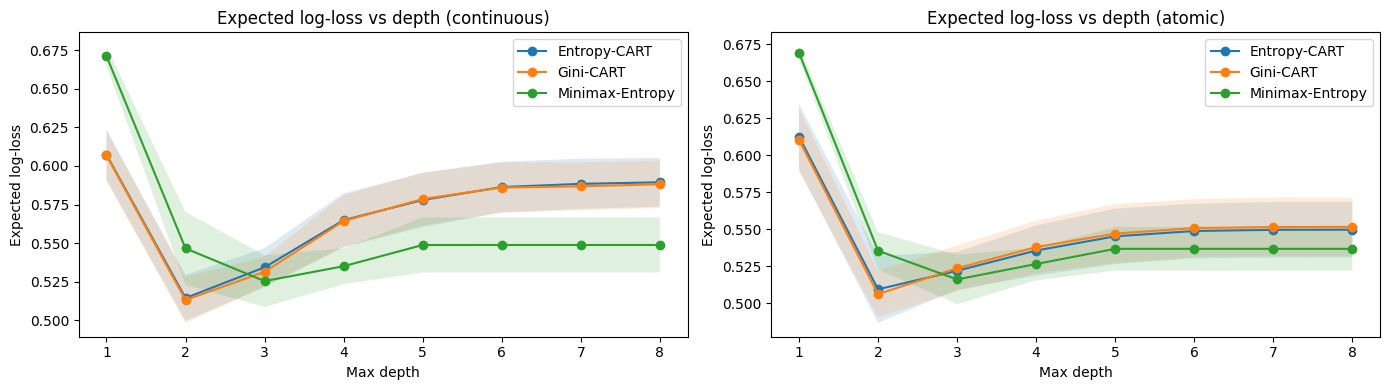}
\caption{Expected log-loss for the one-dimensional piecewise classification problem under continuous and atomic covariates.
Left: $X\lawis \mathrm{Unif}(0,1)$. Right: $X$ sampled from a uniform grid on $[0,1]$.
The same qualitative pattern appears in both settings: entropy-CART and Gini-CART improve fastest at very small depths, while MinimaxSplit is more stable and attains the best deep-tree surrogate risk.}
\label{fig:appendix-classification-atomic}
\end{figure}

\section{Further numerics}

\label{app:further numerics}

\subsection{Performance metrics}

\label{subsec:Performance-Metrics}

\paragraph{Mean Squared Error (MSE).}

Given a matrix (or image) $y\in\mathbb{R}^{H\times W}$ and an estimate
$\hat{y}$, we define
\[
\mathrm{MSE}(y,\hat{y})=\frac{1}{HW}\sum_{i=1}^{H}\sum_{j=1}^{W}\bigl(y_{ij}-\hat{y}_{ij}\bigr)^{2}.
\]
MSE emphasizes large errors (squaring) and is simple to optimize. We also define the Root Mean Squared Error (RMSE) as
\begin{align*}
\operatorname{RMSE}(y,\hat y)
  \;=\; \sqrt{\operatorname{MSE}(y,\hat y)}
  \;=\; \sqrt{\frac{1}{HW}\sum_{i=1}^{H}\sum_{j=1}^{W}\bigl(y_{ij}-\hat y_{ij}\bigr)^{2}}\,.
\end{align*}

\paragraph{Mean Absolute Error (MAE).} We define
\[
\mathrm{MAE}(y,\hat{y})=\frac{1}{HW}\sum_{i=1}^{H}\sum_{j=1}^{W}\lvert y_{ij}-\hat{y}_{ij}\rvert.
\]
MAE (an $\ell_{1}$ criterion) is \emph{more robust to outliers} than
MSE and often produces sharper restorations with less over-smoothing
in image restoration tasks; in practice, it can outperform $\ell_{2}$ metrics
on denoising and related problems \citep{zhao2016loss,jadon2022regression}.

\paragraph{Coefficient of Determination ($R^2$).} Let the spatial mean be
\begin{align*}
\bar y \;=\; \frac{1}{HW}\sum_{i=1}^{H}\sum_{j=1}^{W} y_{ij}.
\end{align*}
Define the residual and total sums of squares
\begin{align*}
\mathrm{SS}_{\mathrm{res}}
  \;=\; \sum_{i=1}^{H}\sum_{j=1}^{W}\bigl(y_{ij}-\hat y_{ij}\bigr)^{2},
\qquad
\mathrm{SS}_{\mathrm{tot}}
  \;=\; \sum_{i=1}^{H}\sum_{j=1}^{W}\bigl(y_{ij}-\bar y\bigr)^{2}.
\end{align*}
Then
\begin{align*}
R^{2}(y,\hat y)
  \;=\; 1 - \frac{\mathrm{SS}_{\mathrm{res}}}{\mathrm{SS}_{\mathrm{tot}}}
  \;=\; 1 - \frac{\operatorname{MSE}(y,\hat y)}{\operatorname{Var}(y)}\,,
\quad
\text{where }\;
\operatorname{Var}(y)
  \;=\; \frac{1}{HW}\sum_{i=1}^{H}\sum_{j=1}^{W}\bigl(y_{ij}-\bar y\bigr)^{2}.
\end{align*}

\paragraph{Structural Similarity (SSIM) \citep{wang2004image}.}


For a window $(\Omega\subset\{1,\ldots,H\}\times\{1,\ldots,W\})$ with non-negative weights
$(\{w_{uv}\}_{(u,v)\in\Omega})$ satisfying $\sum_{(u,v)\in\Omega} w_{uv}=1$, define the local
means, variances and covariance
\begin{align*}
\mu_y(\Omega)=\sum_{(u,v)\in\Omega} w_{uv}\, y_{uv},\qquad
\mu_{\hat y}(\Omega)=\sum_{(u,v)\in\Omega} w_{uv}\, \hat y_{uv},
\end{align*}
\begin{align*}
\sigma_y^2(\Omega)=\sum_{(u,v)\in\Omega} w_{uv}\,\bigl(y_{uv}-\mu_y(\Omega)\bigr)^2,\qquad
\sigma_{\hat y}^2(\Omega)=\sum_{(u,v)\in\Omega} w_{uv}\,\bigl(\hat y_{uv}-\mu_{\hat y}(\Omega)\bigr)^2,
\end{align*}
\begin{align*}
\sigma_{y\hat y}(\Omega)=\sum_{(u,v)\in\Omega} w_{uv}\,\bigl(y_{uv}-\mu_y(\Omega)\bigr)\bigl(\hat y_{uv}-\mu_{\hat y}(\Omega)\bigr).
\end{align*}
With stability constants $C_1=(K_1L)^2$ and $C_2=(K_2L)^2$ (dynamic range $L$; typically $K_1=0.01$,
$K_2=0.03$), the window-level structural similarity is
\begin{align*}
\mathrm{SSIM}_\Omega(y,\hat y)\;=\;
\frac{\bigl(2\,\mu_y(\Omega)\,\mu_{\hat y}(\Omega)+C_1\bigr)\bigl(2\,\sigma_{y\hat y}(\Omega)+C_2\bigr)}
     {\bigl(\mu_y(\Omega)^2+\mu_{\hat y}(\Omega)^2+C_1\bigr)\bigl(\sigma_y^2(\Omega)+\sigma_{\hat y}^2(\Omega)+C_2\bigr)}.
\end{align*}

The image-level SSIM is the average over all window centers $\{\Omega_m\}_{m=1}^M$,
\[
\mathrm{SSIM}(y,\hat y)\;=\;\frac{1}{M}\sum_{m=1}^M \mathrm{SSIM}_{\Omega_m}(y,\hat y).
\]
In our experiments the images are normalized to $[0,1]$, hence $L=1$, and we use Gaussian
weights $w_{uv}$ and mean aggregation as in \texttt{skimage.metrics.structural\_similarity}.

\subsection{Two-dimensional input domain}

\label{sec:2}

For this experiment, we use a synthetic dataset generated from a complex,
non-linear function of two features, $\bx=(x_{1},x_{2})$. 
The true function $f:\mathbb{R}^{2}\rightarrow\mathbb{R}$ used for
generating the target values is defined as: 
\begin{align}
\begin{split}f(\bx) & =\left(2-2.1u_{1}^{2}+\frac{u_{1}^{4}}{3}\right)u_{1}^{2}+u_{1}u_{2}+\left(-4+4u_{2}^{2}\right)u_{2}^{2};\\
 & u_{1}=3(x_{1}-0.5),\\
 & u_{2}=3(x_{2}-0.5),
\end{split}
\label{eq:function}
\end{align}
where $x_{1},x_{2}$ are independently and uniformly sampled from
$[0,1]$. This function incorporates both quadratic and higher-order
polynomial terms, creating a complex surface with multiple peaks and
valleys. The visualizations of the prediction surfaces, as shown in
Figures \ref{fig:The-predictions-comparison-1} and \ref{fig:The-predictions-comparison-2},
further illustrate the improved accuracy and smoother transitions
in the predicted values when using the proposed methods. In addition
to decision trees fitted with MSE ($L^{2}$) loss with these three
kinds of splitting criteria, we also provide trees fitted with $L^{1}$
loss for completeness.

\begin{figure}[h!]
\centering \includegraphics[width=0.95\textwidth]{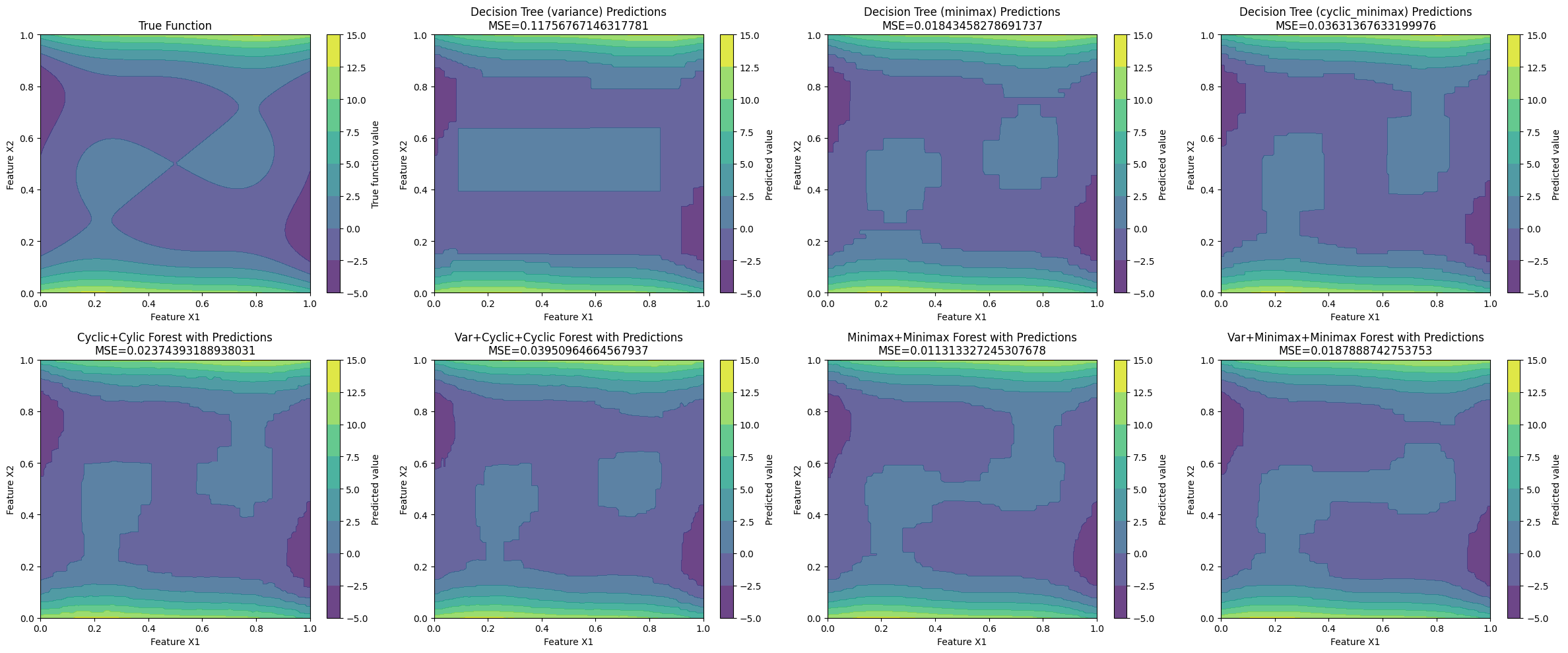}

\caption{\label{fig:The-predictions-comparison-1}Predictions and true function
values of \eqref{eq:function}, visualized using heatmaps. Top Row:
The first subplot shows the heatmap of the true function values across
the input space, serving as the benchmark for evaluating the models.
The next three subplots depict the predictions of the variance-based
decision tree, minimax decision tree, and cyclic minimax decision
tree (depth $k=10$), respectively. Each plot is annotated with the
corresponding MSE to quantitatively assess the accuracy. Bottom Row:
The final four subplots visualize the predictions from four different
random forest models (using Algorithm 1 of Appendix \ref{sec:Algorithms}),
each built with varying combinations of the three error methods. Again,
each plot is labeled with the MSE to facilitate direct comparison.}
\end{figure}

\begin{figure}[t!]
\centering

\includegraphics[width=0.95\textwidth]{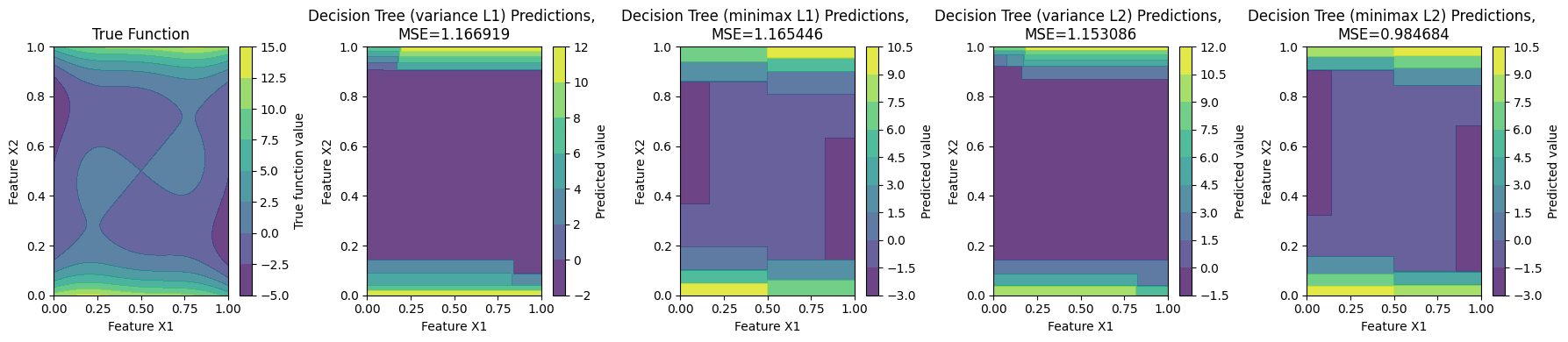}

\includegraphics[width=0.95\textwidth]{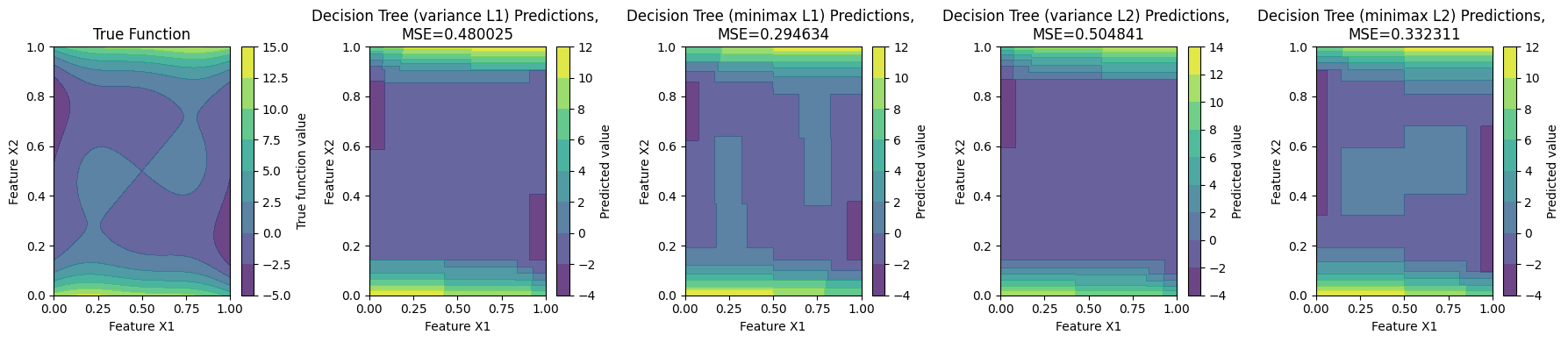}

\includegraphics[width=0.95\textwidth]{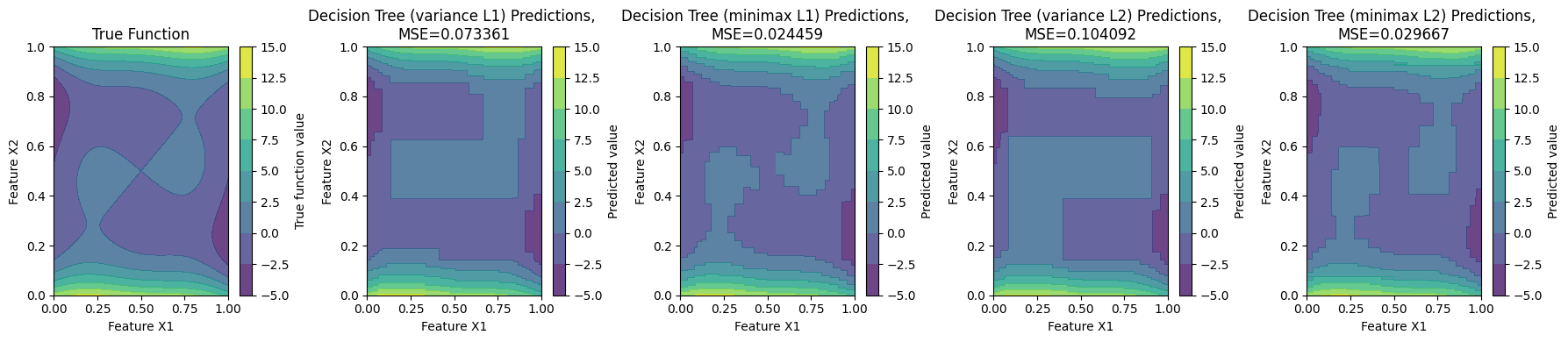}

\caption{\label{fig:The-predictions-comparison-2}The predictions and true
function values are visualized using heatmaps, allowing for a clear
comparison of how well each model approximates the underlying function
\eqref{eq:function}. The first subplot shows the heatmap of the true
function values across the input space, serving as the benchmark for
evaluating the models. The next three subplots depict the predictions
of the $L^{1}$ VarianceSplit, $L^{1}$ MinimaxSplit, $L^{2}$ VarianceSplit,
and $L^{2}$ MinimaxSplit. Top Row: maximum depth 2; Middle Row: maximum depth
6; Bottom Row: maximum depth 10.}
\end{figure}

In this example (as well as in Sections \ref{sec:app denoise} and
\ref{sec:RF} below), we also incorporate an $L^{1}$ variation of
the VarianceSplit and MinimaxSplit algorithms, where in the splitting
rule, we replace the variances in \eqref{eq:variance tree def}, \eqref{eq:minimaxpt},
and \eqref{eq:minimax split point} by the $L^{1}$ distance from
the mean. For instance, the $L^{1}$ variation of \eqref{eq:variance tree def}
is 
\begin{align}
\begin{split}(j,x_{j}) & =\argmin_{(j,\,x_{j})}\Big(\E\big[|Y-\E[Y\mid\bX\in A,\,X_{j}<x_{j}]|\bone_{\{\bX\in A,\,X_{j}<x_{j}\}}\big]\\
 & \hspace{2cm}+\E\big[|Y-\E[Y\mid\bX\in A,\,X_{j}\geq x_{j}]|\bone_{\{\bX\in A,\,X_{j}\geq x_{j}\}}\big]\Big).
\end{split}
\label{eq:minimax tree def}
\end{align}
The optimal $L^{1}$ split point may not be unique, in which case
we break ties arbitrarily.

\subsection{Higher-dimensional input domain}

\label{sec:high-dim}

Consider a dimension $d\geq1$ that is a multiple of $4$. The Powell
function on $[0,1]^d$, 
\[
f(\bx)=\sum_{i=1}^{d/4}[(x_{4i-3}+10x_{4i-2})^{2}+5(x_{4i-1}-x_{4i})^{2}+(x_{4i-2}-2x_{4i-1})^{4}+10(x_{4i-3}-x_{4i})^{4}],
\]
with uniformly random samples on $[0,1]^{d}$, serves as an exemplary
testbed for decision tree methods in high-dimensional regression problems
due to its inherent non-linearity and intricate variable interactions.
The quartic terms in the function introduce sharp curvatures in the
response surface, necessitating sophisticated splitting strategies
to approximate these non-linear relationships accurately. 

\begin{table}[t!]
\centering
\begin{tabular}{ccccc}
\hline 
Sample Size & Dimension & Scikit-learn & VarianceSplit & MinimaxSplit\tabularnewline
\hline 
10 & 4 & $4.48\times10^{2}$ & $\boxed{2.01\times10^{2}}$ & $2.57\times10^{2}$\tabularnewline
10 & 16 & $1.60\times10^{4}$ & ${1.52\times10^{4}}$ & $\boxed{1.52\times10^{4}}$\tabularnewline
10 & 64 & $3.70\times10^{4}$ & $4.24\times10^{4}$ & $\boxed{3.79\times10^{4}}$\tabularnewline
10 & 256 & $1.35\times10^{5}$ & $1.24\times10^{5}$ & $\boxed{1.14\times10^{5}}$\tabularnewline
10 & 1024 & $9.31\times10^{5}$ & $8.04\times10^{5}$ & $\boxed{4.28\times10^{5}}$\tabularnewline
100 & 4 & $5.49\times10^{1}$ & $\boxed{5.21\times10^{1}}$ & $5.31\times10^{1}$\tabularnewline
100 & 16 & $2.70\times10^{3}$ & $2.67\times10^{3}$ & $\boxed{2.57\times10^{3}}$\tabularnewline
100 & 64 & $1.99\times10^{4}$ & $\boxed{1.95\times10^{4}}$ & $2.20\times10^{4}$\tabularnewline
100 & 256 & $1.01\times10^{5}$ & $1.02\times10^{5}$ & $\boxed{9.39\times10^{4}}$\tabularnewline
100 & 1024 & $4.65\times10^{5}$ & $4.29\times10^{5}$ & $\boxed{4.02\times10^{5}}$\tabularnewline
1000 & 4 & $4.68\times10^{1}$ & $4.65\times10^{1}$ & $\boxed{4.49\times10^{1}}$\tabularnewline
1000 & 16 & $\boxed{1.85\times10^{3}}$ & $1.85\times10^{3}$ & $2.03\times10^{3}$\tabularnewline
1000 & 64 & $1.62\times10^{4}$ & $1.62\times10^{4}$ & $\boxed{1.59\times10^{4}}$\tabularnewline
1000 & 256 & $6.46\times10^{4}$ & $6.47\times10^{4}$ & $\boxed{6.30\times10^{4}}$\tabularnewline
1000 & 1024 & $2.94\times10^{5}$ & $2.96\times10^{5}$ & $\boxed{2.88\times10^{5}}$\tabularnewline
10000 & 4 & $\boxed{4.61\times10^{1}}$ & $4.61\times10^{1}$ & $4.63\times10^{1}$\tabularnewline
10000 & 16 & $\boxed{1.87\times10^{3}}$ & $1.87\times10^{3}$ & $1.97\times10^{3}$\tabularnewline
10000 & 64 & $\boxed{1.36\times10^{4}}$ & $1.36\times10^{4}$ & $1.43\times10^{4}$\tabularnewline
10000 & 256 & $6.22\times10^{4}$ & $6.22\times10^{4}$ & $\boxed{6.11\times10^{4}}$\tabularnewline
10000 & 1024 & $2.80\times10^{5}$ & $2.80\times10^{5}$ & $\boxed{2.78\times10^{5}}$\tabularnewline
\hline 
\end{tabular}
\caption{Comparison of Decision Tree Methods ($\text{maximum depth}=3$) using the Powell
function without noise. Scikit-learn's DecisionTreeRegressor uses
MSE as its default splitting criterion, which is equivalent to minimizing
the variance of the target variable within each split and serves as
a baseline comparison here. The winner in each setting is boxed.}
\label{tab:tree-comparison}
\end{table}

Table \ref{tab:tree-comparison} compares three decision tree methods---scikit-learn's
standard implementation, VarianceSplit and MinimaxSplit---across
various sample sizes $n$ and dimensions $d$ for the Powell function.
In the dense design regime ($n>d$), 
we observe relatively consistent performance across all three methods.
However, as we transition into sparse designs ($d>n$), 
more pronounced differences emerge. The MinimaxSplit method often
demonstrates superior performance in these high-dimensional, data-scarce
scenarios. For instance, with 10 samples and 256 dimensions, the MinimaxSplit
approach achieves an MSE of 1.14 \texttimes{} 10\textsuperscript{5},
outperforming both scikit-learn (1.24 \texttimes{} 10\textsuperscript{5})
and the VarianceSplit method (1.35 \texttimes{} 10\textsuperscript{5}).

\section{Pseudocode for algorithms }

\label{sec:Algorithms}





\begin{algorithm}[p]
\caption{Node-level splitting rules in the notation of Section~2}
\label{alg:nodesplit-sec2}
\DontPrintSemicolon
\KwIn{Node $A_t$ at depth $k$ with sample $D_t=\{(x_i,y_i): x_i\in A_t\}$; feature set $[d]=\{1,\dots,d\}$ and $m_{try}=d$ by default}
\KwOut{$S_j(A_t)$: candidate thresholds for feature $j$ at $A_t$; $A^L_{j,s}=A_t\cap\{x_j\le s\}$, $A^R_{j,s}=A_t\cap\{x_j> s\}$; $P_t(B)=\Pr(X\in B\mid X\in A_t)\approx n(B)/n(A_t)$}
\vspace{0.15em}

\textbf{VarianceSplit (CART)} \tcp*[r]{cf. Fig.~2, left}
\For{$(j,s)\in\{(j,s): j\in[d],\ s\in S_j(A_t)\}$}{
  $\mathrm{Crit}(j,s)\gets P_t(A^L_{j,s})\,\Var(Y\mid A^L_{j,s}) + P_t(A^R_{j,s})\,\Var(Y\mid A^R_{j,s})$ \;
}
$(j^\star,s^\star)\gets \arg\min_{(j,s)} \mathrm{Crit}(j,s)$; \quad \textbf{split } $A_t\to A^L_{j^\star,s^\star},\,A^R_{j^\star,s^\star}$.
\vspace{0.35em}

\textbf{MinimaxSplit} \tcp*[r]{cf. Fig.~2, middle}
\For{$(j,s)\in\{(j,s): j\in[d],\ s\in S_j(A_t)\}$}{
  $\mathrm{Crit}(j,s)\gets \max\big\{P_t(A^L_{j,s})\,\Var(Y\mid A^L_{j,s}),\ P_t(A^R_{j,s})\,\Var(Y\mid A^R_{j,s})\big\}$ \;
}
$(j^\star,s^\star)\gets \arg\min_{(j,s)} \mathrm{Crit}(j,s)$; \quad \textbf{split } $A_t\to A^L_{j^\star,s^\star},\,A^R_{j^\star,s^\star}$.
\vspace{0.35em}

\textbf{Cyclic MinimaxSplit} \tcp*[r]{cf. Fig.~2, right}
$j_k\gets 1+(k\bmod d)$ \tcp*[r]{cyclic feature at depth $k$}
\For{$s\in S_{j_k}(A_t)$}{
  $\mathrm{Crit}(j_k,s)\gets \max\big\{P_t(A^L_{j_k,s})\,\Var(Y\mid A^L_{j_k,s}),\ P_t(A^R_{j_k,s})\,\Var(Y\mid A^R_{j_k,s})\big\}$ \;
}
$s^\star\gets \arg\min_{s\in S_{j_k}(A_t)} \mathrm{Crit}(j_k,s)$; \quad \textbf{split } $A_t\to A^L_{j_k,s^\star},\,A^R_{j_k,s^\star}$.
\end{algorithm}


\begin{figure}
\centering

\includegraphics[width=0.7\textwidth]{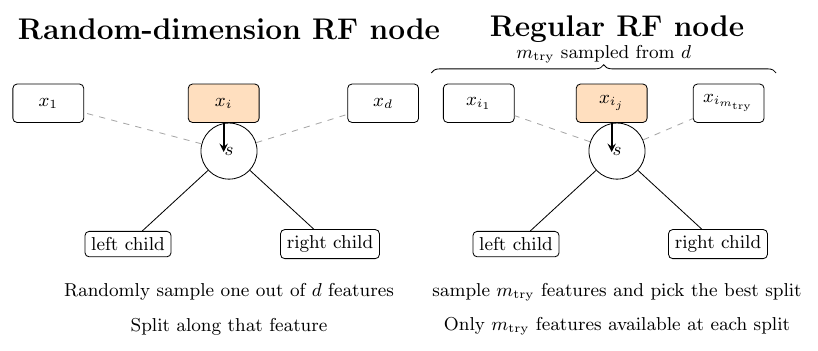}

\caption{Comparison of node-level split selection in two random forest variants.
Left: a random-dimension random forest
selects a random feature uniformly for each node and chooses the best split over the chosen feature. 
Right: a regular RF node samples $m_{\mathrm{try}}$ coordinates uniformly from the $d$
features 
and chooses the best split over all sampled features
according to the chosen impurity/criterion (the
highlighted feature illustrates the winner). 
}

\end{figure}




\begin{algorithm}[p]
\caption{Random Forest with generic split rule}
\label{alg:rf-regular}
\KwIn{Sample $\mathcal{D}=\{(X_i,Y_i)\}_{i=1}^N$, number of trees $B\in\mathbb{N}$, maximum depth $K\in\mathbb{N}$, minimum leaf size $n_{\min}\in\mathbb{N}$, per-node feature budget $m_{\mathrm{try}}$, split rule $\mathrm{SplitRule}$.}
\KwOut{Ensemble predictor $\hat f_B(x)=\frac{1}{B}\sum_{b=1}^B \hat f^{(b)}(x)$.}

\For{$b \leftarrow 1$ {\bf to} $B$}{
  \tcp{Bootstrap sample}
  Draw $N$ indices $I^{(b)}=(i_1,\dots,i_N)$ uniformly with replacement from $[N]$ and set $\mathcal{D}^{(b)}=\{(X_{i_\ell},Y_{i_\ell})\}_{\ell=1}^N$.

  \tcp{Grow a full tree on $\mathcal{D}^{(b)}$}
  Initialize root node with index set $S_{\text{root}}=[N]$ and depth $0$.

  \While{there exists a node $(S,\text{depth})$ with $|S|>n_{\min}$ and $\text{depth}<K$}{
    \tcp{Sample a random subset of coordinates at this node}
    Draw without replacement a subset $\mathcal{J}\subseteq [d]$ with $|\mathcal{J}|=m_{\mathrm{try}}\in [d]$.

    \tcp{Choose best coordinate/threshold according to the chosen impurity criterion}
    $(j^\star,t^\star)\leftarrow \mathrm{SplitRule}(S,\mathcal{J},\mathcal{D}^{(b)})$.

    \tcp{Partition the node}
    $S_L \leftarrow \{i\in S:\, X_{ij^\star}\le t^\star\}$,\quad
    $S_R \leftarrow \{i\in S:\, X_{ij^\star}> t^\star\}$.

    \tcp{Create children; stop if a child is too small}
    \eIf{$|S_L|\le n_{\min}$ \textbf{ or } $\text{depth}+1=K$}{
      mark left child as a leaf with prediction $\hat\mu_{S_L}=\frac{1}{|S_L|}\sum_{i\in S_L}Y_i$;
    }{
      push $(S_L,\text{depth}+1)$ to the queue;
    }
    \eIf{$|S_R|\le n_{\min}$ \textbf{ or } $\text{depth}+1=K$}{
      mark right child as a leaf with prediction $\hat\mu_{S_R}=\frac{1}{|S_R|}\sum_{i\in S_R}Y_i$;
    }{
      push $(S_R,\text{depth}+1)$ to the queue;
    }
  }

  \tcp{Define the tree predictor by leaf means}
  For any $x\in\mathbb{R}^d$, let $\hat f^{(b)}(x)$ be the mean $\hat\mu_{S}$ of the leaf whose cell contains $x$.
}
\end{algorithm}





\begin{algorithm}[p]
\caption{Random-Dimension Random Forest (feature subsampling at each node)}
\label{alg:rf-randdim}
\KwIn{Sample $\mathcal{D}=\{(X_i,Y_i)\}_{i=1}^N$, number of trees $B$, maximum depth $K$, minimum leaf size $n_{\min}$, per-node feature budget $m_{\mathrm{try}}=1$, split rule $\mathrm{SplitRule}$.}
\KwOut{Ensemble predictor $\hat f_B(x)=\frac{1}{B}\sum_{b=1}^B \hat f^{(b)}(x)$.}

\For{$b \leftarrow 1$ {\bf to} $B$}{
  Draw a bootstrap sample $\mathcal{D}^{(b)}$ of size $N$ from $\mathcal{D}$.

  Initialize the root node with index set $S_{\text{root}}=[N]$ and depth $0$.

  \While{there exists a node $(S,\text{depth})$ with $|S|>n_{\min}$ and $\text{depth}<K$}{
    \tcp{Sample a random subset of coordinates at this node}
    Draw without replacement a subset $\mathcal{J}\subseteq [d]$ with $|\mathcal{J}|=m_{\mathrm{try}}$, where we set $m_{\mathrm{try}}=1$.

    \tcp{Choose best coordinate/threshold using only $\mathcal{J}$}
    $(j^\star,t^\star)\leftarrow \mathrm{SplitRule}(S,\mathcal{J},\mathcal{D}^{(b)})$.

    \tcp{Partition and create children}
    $S_L \leftarrow \{i\in S:\, X_{ij^\star}\le t^\star\}$,\quad
    $S_R \leftarrow \{i\in S:\, X_{ij^\star}> t^\star\}$.

    \eIf{$|S_L|\le n_{\min}$ \textbf{ or } $\text{depth}+1=K$}{
      make left child a leaf with prediction $\hat\mu_{S_L}$;
    }{
      push $(S_L,\text{depth}+1)$;
    }
    \eIf{$|S_R|\le n_{\min}$ \textbf{ or } $\text{depth}+1=K$}{
      make right child a leaf with prediction $\hat\mu_{S_R}$;
    }{
      push $(S_R,\text{depth}+1)$;
    }
  }

  Define $\hat f^{(b)}(x)$ by the mean of the leaf containing $x$.
}
\end{algorithm}


\end{document}